%% file: main.tex
\documentclass{amsart}
\usepackage[utf8]{inputenc}
\usepackage[T1]{fontenc}
\usepackage{amssymb,amsthm,amsmath,graphicx, longtable, wrapfig, rotating, microtype, xparse, enumitem, mathtools, caption, subcaption, chngcntr, xcolor, bookmark}
\usepackage{hyperref}

\theoremstyle{plain} 
\theoremstyle{plain}
\newtheorem{theorem}{Theorem}
\newtheorem{lemma}[theorem]{Lemma}
\newtheorem{corollary}[theorem]{Corollary}
\newtheorem{remark}[theorem]{Remark}
\newtheorem{example}[theorem]{Example}
\newtheorem{proposition}[theorem]{Proposition}
\newtheorem*{definition}{Definition}

\newtheorem{claim}[theorem]{Claim}

\input{mycommands}

\title[Goodstein at the Second Threshold]{\texorpdfstring{Goodstein at the Second Threshold:\\ An independence Result for $\sf ID_2$}{Goodstein at the Second Threshold: An Independence Result for ID2}}
\author[O. Gjetaj]{Oriola Gjetaj}
\author[A. Weiermann]{Andreas Weiermann}

\begin{document}

\begin{abstract}
The classical Goodstein process, defined via hereditary base-$k$ exponential normal form, is a well-known example of a principle unprovable in Peano Arithmetic. 
In this paper, we generalize this framework by constructing a new Goodstein process based on the Hardy hierarchy. We develop an ordinal notation system utilizing a two-step collapsing procedure, which yields a proof-theoretic ordinal of $\psi_0\psi_1(\varepsilon_{\Om_2+1})$. By defining $k$-normal forms for natural numbers within this system, we introduce a Goodstein-type process and demonstrate that the theory of non-iterated positive inductive definitions for two operators ($\ID_2$) cannot prove its termination. This result establishes a new independence result at the second proof-theoretic threshold, further extending the reach of Goodstein-type principles beyond the Bachmann–Howard level.
\end{abstract}
\maketitle


\section{Introduction}\label{sec: Intro}
Goodstein’s theorem \cite{Goodstein1944} is a natural instance of Gödel’s incompleteness phenomenon: its formalization in first–order arithmetic is true but unprovable in $\PA$. The usual proof proceeds by transfinite induction up to $\varepsilon_0$, which is itself not provable in $\PA$. Kirby and Paris later showed that the Goodstein principle is independent of $\PA$ \cite{Kirby.Paris1982}.
Goodstein’s original process is defined on natural numbers written in hereditary base–$k$ exponential normal form and concerns the termination of certain fast-growing sequences of natural numbers.

Say that $m$ is in {\em nested (exponential) base-$k$ normal form} if it is written in standard exponential base $k$, with each exponent written in turn in base $k$.
Thus for example, $20 $ would become $2^{2^2}+2^2$ in nested base-$2$ normal form.
Then, define a sequence $(g_k(0))_{m\in\mathbb N}$ by setting $g_0(m) = m$ and defining $g_{k+1}(m)$ recursively by writing $ g_k(m)$ in nested base-$(k+2)$ normal form, replacing every occurrence of $k+2$ by $k+3$, then subtracting one (unless $g_k(m)=0$, in which case $g_{k+1}(m) =0$).
In the case that $m=20$, we obtain
\begin{align*}
g_0(20) & =20 = 2^{2^2}+2^2\\
g_1(20) &  = 3^{3^3}+3^3 -1 = 3^{3^3}+3^2\cdot 2 + 3\cdot 2 + 2\\
g_2(20) & = 4^{4^4}+4^2\cdot 2 + 4\cdot 2 +2 -1 = 4^{4^4}+4^2\cdot 2 + 4\cdot 2 +1,
\end{align*} 
and so forth.
At first glance, these numbers seem to grow superexponentially.
It should thus be a surprise that, as Goodstein showed, for every $m$ there is $k^*$ for which $g_{k^*}(m)=0$. 
By coding finite Goodstein sequences as natural numbers in a standard way, Goodstei{n'}s principle can be formalized in the language of arithmetic, but this formalized statement is unprovable in $\PA$. Independence can be shown by proving that the Goodstein process takes at least as long as stepping down the {\em fundamental sequences} below $\varepsilon_0$; these are canonical sequences $( \xi [n])_{n<\omega}$ such that $ \xi [n]< \xi$ for all $\xi$ and for limit $\xi$, $ \xi [n] \to \xi$ as $n\to \infty$.
For standard fundamental sequences below $\varepsilon_0$, $\PA$ does not prove that the sequence $\xi >  \xi [1] > \xi[1][2] > \xi[1][2][3] \ldots $ is finite. 
\smallskip

Cichon later gave a shorter proof by relating Goodstein process to the fast‑growing hierarchy $F_\al$ and the ordinal‑recursion characterization of the provably total functions of $\PA$.\cite{Cichon1983}
These developments are part of a broader interaction between Goodstein‑type processes, fast‑growing hierarchies and proof‑theoretic ordinal analysis \cite{Buchholz.Wainer1987, Weiermann95, Cichon.Wainer1983}.
Buchholz, Cichon, and Weiermann \cite{BuchholzCW1994} gave a uniform treatment of fundamental sequences and Hardy fast‑growing hierarchy, which supports many later independence results. 
Ketonen and Solovay’s analysis of the Paris–Harrington theorem linked finite combinatorics to the same hierarchy, yielding another natural
$\PA$‑independent statement. \cite{ketonen1981rapidly}

Goodstein processes have since been extended well beyond the $\varepsilon_0$‑level.
Arai, Wainer, and Weiermann \cite{AraiWainerWeiermann} introduced  Goodstein sequences using parametrized Ackermann-P\'eter functions. They showed that each of these sequences terminates and is independent of a range of theories: $\PRA, \PA, \sf \Sigma_1^1-DC_0, \ATR_0$ up to  $\ID_1$. Crucially, they employed tree ordinals to obtain their results. Arai, Fern\'andez-Duque, Wainer, and Weiermann \cite{Arai.etal2020} developed  Ackermannian Goodstein principles independent of $\ATR_0$, using an elaborate ‘sandwiching’ procedure first introduced
in \cite{Weiermann2017}. Later, Fern\'andez-Duque et al. \cite{Fernandez-Duque.etal2023} consider simpler, and arguably more intuitive, normal forms, also based on the Ackermann function, giving rise to  three Goodstein-like processes, independent of $\ACA_0, \ACA_0^{'}$ and $\ACA_0^+$, respectively.

More recently, Fern\'andez‑Duque and Weiermann introduced a new variant of the classic Goodstein process, {\em Fast Goodstein Walks} \cite{FernandezDuqueWeiermann2022} based on a family $(\mathcal{A}_k)_{k<\omega}$ of  fast-growing functions defined by transfinite recursion below $\varepsilon_0$. They demonstrate that these walks always terminate and obtained independence results at the Bachmann-Howard ordinal for $\ID_1$ and related systems such as Kripke-Platek set theory.
In further work \cite{FernandezDuqueWeiermann2024WalkGoodstein} they developed the notion of {\em base–change maximality} and show that by varying the initial base of the Goodstein process, one readily obtains independence results for each of the fragments $\sf I\Sigma^0_n$ of Peano arithmetic.
\smallskip

We extend this line of work from the Bachmann–Howard level to a new independence result for $\ID_2$.
We construct an ordinal notation system for $|\ID_2|$
following Buchholz–Sch\"utte\cite{BuchholzSchutte} and design a Goodstein process using a Bachmann–Howard Hardy hierarchy.
We show the termination of this Goodstein principle and obtain a new natural statement independent of $\ID_2$.
In our setting, the collapse is a two step procedure, 
where $\psi_0$ ordinals collapse to Hardy functions and $\psi_1$ ordinals to $\psi_0$ ordinals, avoiding the simultaneous definition of \cite{BuchholzSchutte}. 
This gives a more detailed framework for analyzing Goodstein process, while working entirely with countable ordinals and avoiding a Hardy hierarchy at $\varepsilon_{\Omega_1+1}$. The underlying system of fundamental sequences enjoys the Bachmann property (due to Buchholz \cite{Buchholz2017SurveyOrdinalNotations}, in the formulation of Weiermann \cite{Weiermann95}), which is crucial for transferring proof–theoretic facts to our Goodstein process.\smallskip

The proof proceeds in four steps.
First, an ordinal notation system is constructed for the theory $|\ID2|$ via a two-step collapsing procedure, providing a framework for ordinals below the proof-theoretic ordinal of the theory.
Second, define a fast-growing hierarchy to represent natural numbers as $k$-normal forms, embedding these numbers into our ordinal system. In our case, this is the Hardy hierarchy. 
Third, define a monotone base-change operation on these representations, which controls the dynamics of our Goodstein process. 
Finally, assign an ordinal to each natural number such that every step of our Goodstein process corresponds to a strict descent in our notation system. By showing that this process grows at least as quickly as the fast-growing hierarchy $F_{|\ID2|}$, it is shown that its termination cannot be proven in $\ID_2$.\smallskip
 
The paper is organized as follows.
In Section~\ref{sec: PrelimID2}, following Buchholz and Sch\"utte \cite{BuchholzSchutte}, we construct the relevant ordinal notation system and recall the proof-theoretic facts needed later. 
Section~\ref{sec:PropID2} develops structural properties of this system for $|\ID2|$, providing the groundwork for our Goodstein process. 
In Section~\ref{sec:GoodsteinID2} we introduce a Hardy-like hierarchy to express natural numbers and define the associated Goodstein process. Then we obtain termination via an ordinal assignment into $|\ID2|$. 
Finally, in Section~\ref{sec: Indep} we show that $\ID_2$ does not prove this termination, yielding an independence result at the second proof-theoretic threshold.

\subsection{Notation and Conventions}

We collect here the standard notation used throughout the paper. We denote the set of natural numbers by $\Nat = \{0, 1, 2, \dots\}$. 
For a set of ordinals $A$ and an ordinal $\al$, we write $A < \al$ if $\forall\be \in A(\be < \al)$ and $\al < A$ if $\exists\be \in A(\al < \be)$.

\section{Construction of ordinal notation system for \texorpdfstring{$\mathsf{ID_2}$}{ID2}}
\label{sec: PrelimID2}
In this section we introduce the notation system we'll be using for $|\ID_2|$. 
For our purposes there are two natural ways to construct ordinal notation systems: the simultaneous Buchholz–Sch\"utte system \cite{BuchholzSchutte} and a step by step variant. We adopt the latter approach for a fine–grained analysis of Goodstein–type processes at the level of $\ID2$. This makes the interaction between the collapsing functions and the Hardy hierarchy more transparent.

\subsection{Function \texorpdfstring{$\psi$}{psi}} \label{subsec:FunctionPsi}
We construct an ordinal notation system for $\ID_2$ via a two-step collapse. 
First, we construct inductively the set $C_1(\al)$ and the collapsing function $\psi_1(\al)$ which is defined as the smallest ordinal not belonging to the set $C_1(\al)$. The value of $\psi_1$ will collapse $C_1$ below $\Omega_2$.
Second, we construct $C_0$ and $\psi_0$ by utilizing the previously defined $\psi_1$. 
In the simultaneous approach of Buchholz–Schütte \cite{BuchholzSchutte}, the proof theoretic ordinal of $\ID_2$ is $\psi_0(\varepsilon_{\Omega_2+1})$, whereas our construction results in the proof-theoretic ordinal $\psi_0\psi_1(\varepsilon_{\Omega_2+1})$.
\\
Define $\psi_2(\beta) \coloneqq \omega^{\Omega_2+\beta}$, and for brevity, we often omit parentheses, writing $\psi_0\beta$ for $\psi_0(\beta)$.
Abbreviate $\psi_2(\be) := \om^{\Om_2+\be}$.

\begin{definition} 
	Let $C_1(\al)$ be the least set of ordinals such that:
	\begin{enumerate}
		\item $ \Om \subset C_1(\al)$.
		\item If $\be \in C_1(\al)$ then $\psi_2(\be) \in C_1(\al)$.
		\item If $\be_0,\ldots , \be_n \in C_1(\al) $  with $n \geq 1$ are additive principal ordinals then $\be_1+  \ldots  + \be_n \in C_1(\al)$.
		\item If $\be\in C_1(\al)\cap \al \cap C_1(\be)$ then $\psi_1 \be\in C_1(\al)$.
		\item 	 $\psi_1\al:=\min\{\be:\be\not\in C_1(\al)\}$.
	\end{enumerate}
\end{definition}
The sets $C_1(\alpha)$ are defined by recursion on $\alpha$. The set is closed under addition and the $\psi_2$ function, collecting the countable ordinals and their images under $\psi_2$. Iterating the construction of $C_1(\alpha)$ generates the function $\psi_1(\alpha)$. Its primary properties are summarized as follows:
\begin{lemma} Properties of function $\psi_1$:
	\begin{enumerate}
		\item $\psi_1\al<\Om_2$. 
		\item $\psi_10=\Om$.
		\item If $\al\in C_1(\al)$ then $\psi_1(\al+1)=\psi_1(\al)\cdot \om$.
		\item If $\al\in C_1(\al) \cap \Om_2$ then $\psi_1(\al)=\om^{\Om+\al}$.
		\item $\psi_1(\Om_2) = \varepsilon_{\Om+1}$.
		\item $\psi_1(\al)=C_1(\al)\cap \Om_2$.
	\end{enumerate}
\end{lemma}
Now define the sets $C_0(\alpha)$ and the collapsing function $\psi_0$ by recursion on $\alpha$.

\begin{definition} Let $C_0(\al)$ be the least set of ordinals such that:
	\begin{enumerate}
		\item $ 0,1 \in C_0(\al)$.
		\item If $\be \in C_0(\al)$ then $\psi_2(\be) \in C_0(\al)$.
		\item If $\be_0,\ldots , \be_n \in C_0(\al) $  with $n \geq 1$ are additive  principal ordinals then $\be_1+  \ldots  + \be_n \in C_0(\al)$.
		\item  If $\be\in C_0(\al)\cap C_1(\be)$  then $\psi_1 \be\in C_0(\al)$. 
		\item If $\be\in C_0(\al)\cap \al \cap C_0(\be)$ then $\psi_0 \be\in C_0(\al)$.
		\item 	$\psi_0\al:=\min\{\be:\be\not\in C_0(\al)\}$.
	\end{enumerate}
\end{definition}

The construction of $C_0(\alpha)$ includes the natural numbers(due to the closure of addition function), $\Om_2$, and the images of $\psi_1$. The smallest ordinal it does not contain is $\om$, hence $\psi_00=\om$.\medskip
\\
Similarly, $C_0(1)$ contains all the natural numbers, and this time also $\om$. Since $0 \in C_0(1)\cap C_1(0)$, then $C_0(1$) also contains $\psi_10=\Om$. By closure under addition, the set $C_0(1)$ contains elements such as $\om+1,\om+\om,\ldots, \Om+1, \Om+2,\ldots$. The first smallest ordinal that it does not contain is $\om^2$. Hence, $\psi_01=\om^2$.
Therefore we have $\psi_0\al=\om^{1+\al}$ for all $\al<\varepsilon_0$, the smallest fixed point of the function $\al\rightarrow\om^{1+\al}$.\\
Some of the properties of the $\psi_0$ function are stated in the following Lemma.

\begin{lemma} Properties of $\psi_0$:
	\begin{enumerate}
		\item $\psi_0\al<\Om$. 
		\item $\psi_00=\om$.
		\item  If $\al\in C_0(\al)$ then $\psi_0(\al+1)=\psi_0(\al)\cdot \om$.
		\item  If $\al \in C_0(\al)\cap \Om$ then $\psi_0(\al) = \om^{1+\al}$.
		\item $\psi_0(\Om)=\varepsilon_0$
		\item $\psi_0(\al)=C_0(\al)\cap \Om$ for $\al<\Om_2$.
	\end{enumerate}
\end{lemma}




\subsection{Notation system \texorpdfstring{$\OT$}{OT}}\label{subsec:NotationSys}

We develop the ordinal notation system following the construction by Buchholz, starting with a syntactical definition of the non wellfounded superset $\T$. Then we will construct the wellfounded part $\OT \subseteq \T$.

\begin{definition}
    An ordinal $\al$ is an {\em additive principal ordinal} if $\al\neq 0$ and  $\xi, \eta <\al$ also implies $\xi+\eta<\al$. We denote the class of principal ordinals by AP.
\end{definition}
Thus an additive principal ordinal cannot be decomposed into the sum of smaller ordinals. The first additive principal number is 1, and the other additive principal numbers are limit ordinals, $\psi_0\al, \psi_1\be, \psi_2\be$. One can prove that $\al \rightarrow \om^\al$ is the enumerating function of AP.

\begin{definition}[The term system $(\T,<)$] Inductive definition of a set $\T$ of terms and binary relation $<$ on $\T$.:    \begin{enumerate} 
        \item $0,1 \in T$ and $1$ is a principal term.
        \item If $\be \in T$  then $\psi_2\be \in T$, $\psi_1\be \in T$, $\psi_0\be \in T$ are principal terms.
        \item If $\al_0,\ldots ,\al_n \in T$  with $ n \geq 1 $ are principal terms with $\al_0 \geq \ldots   \geq \al_n $ then $\al_0 +\ldots + \al_n \in T$.
        \item If $0 \neq \al$ then $0<\al$.
        \item $1<\psi_0\al$.
        \item If $\al<\be$ then $\psi_0\al<\psi_0\be$ and  $\psi_1\al<\psi_1\be$ and $\psi_2\al < \psi_2\be$.
        \item If $i<j$ then $\psi_i\al<\psi_j\be$
        \item If $\al= \al_0+\ldots +\al_n$ and $\be=\be_0+\ldots+\be_m$ with $n+m\geq 1 $ 
        and $\al_0<\be_0 \lor ( \al_0 =\be_0 $ and $ \al_1 +\ldots +\al_n < \be_1 + \ldots + \be_m )$ then $\al<\be$.
    \end{enumerate}
\end{definition}
\begin{definition}
$NF(\al,\be)$ denotes that $ \al$ is a principal term and $\be = \be_0+\ldots+\be_{m-1}\in T$ such that $m\geq 0$ and $(0<m \Rightarrow \be_0\leq\al)$.
\end{definition}

\begin{lemma}
	$<$ is a linear order on $\T$.
\end{lemma}

{\em Abbreviation.} $\om=\psi_00$ and $\Om_s :=\psi_s0$.

\begin{definition}[Operations on $\T$]
We extend the structure of $\T$ with the following operations:
\begin{enumerate}
\item \textbf{Addition}: 
If $\alpha = \sum \alpha_i$ and $\beta = \sum \beta_j$ are principal, $\alpha + \beta = \alpha_0 + \dots + \alpha_{i-1} + \beta$, where $i = \max{j \le n : \alpha_j \ge \beta_0}$.

\item \textbf{Multiplication}: For a principal term $\alpha \in \T$ and $n \in \mathbb{N}$, $\alpha \cdot 0 = 0$, and $\alpha \cdot (n+1) = (\alpha \cdot n) + \alpha$.
\end{enumerate}
\end{definition}

\begin{definition} Let the norm function $N(\al)$ for $\al \in \T$ be defined recursively as follows:
\begin{itemize}
    \item $N(0) := 0$ and $N(1):=1$.
    \item If $\al_0, \ldots , \al_n \in T$ with $n\geq 1$ and $\al=\al_0 +\ldots +\al_n$ then $N\al = N\al_0 +\ldots + N\al_n$.
    \item If $\be \in T$ then $N\psi_i(\be) = 1 +N(\be)$ for $i \in\{0,1,2\}$.
\end{itemize}
\end{definition}
\medskip

The elements of $\OT$ are ordinal terms in normal form. The sets $G_0$ and $G_1$  are defined to characterize the subterms of $\T$ necessary for defining $\OT$:
\begin{definition}[Set $G_1(\al)$]
For $\alpha \in \T$:
\begin{enumerate}
    \item If $\al<\Om$ then $G_1\al:=\emptyset$.
    \item If $\be \in T$ and $\xi = \om^{\Om_2 +\be }$ then $G_1\xi = G_1\be$.
    \item If $\al_0,\ldots,\al_n \in T$ and $\xi=\al_0+\ldots+\al_n$ then $G_1\xi = G_1\al_0 \cup  \ldots \cup G_1\al_n$.
    \item If $\be\in T$ then $G_1 \psi_1 \be=\{\be\}\cup G_1\be$ and  $G_1 \psi_0 \be=\emptyset$.
\end{enumerate}
\end{definition}

\begin{definition}[Set $G_0(\al)$]
For $\alpha \in \T$:
\begin{enumerate}
    \item $G_00:=\emptyset$.
    \item If $\be \in T$ and $\xi = \om^{\Om_2 +\be }$ then $G_0\xi = G_0\be$.
    \item If $\al_0,\ldots,\al_n \in T$ and $\xi=\al_0+\ldots+\al_n$ then $G_0\xi = G_0\al_0 \cup  \ldots \cup G_0\al_n$.
    \item If $\be\in T$  then $G_0 \psi_1\be=G_0\be$ and $G_0 \psi_0 \be=\{\be\}\cup G_0\be$.
\end{enumerate}
\end{definition}
\smallskip


With this definition, we can now define the well-founded system $\OT$.
\begin{definition}[The Ordinal Notation System $(\OT, <\restriction \OT)$]  Inductive definition of a subset $\OT$ of $\T$. The binary relation $<$ is restricted to $\OT$.
\begin{enumerate}
    \item $0,1\in \OT$.
    \item $\be \in \OT$ then $\psi_2\be \in \OT$.
    \item If $\al= \al_0+\ldots+\al_n \in T$ and $\al_0,\ldots,\al_n\in \OT$, with $n\geq 1$ then $\al \in \OT$ .
    \item If $\be \in \OT$ and  $G_1\be <\be$  then $\psi_1\be \in \OT$. 
    \item If $\be \in \OT$ and $G_0\be <\be$ and $\be < \Om_2$ then $\psi_0\be \in \OT$.
\end{enumerate}
\end{definition}

\begin{lemma}
Assume $\al,\be \in \OT$.
\begin{enumerate}
    \item 	$\al \in C_0(\be)$ if and only if $G_0\al <\be$.
    \item 	$\al \in C_1(\be)$ if and only if $G_1\al <\be$.
\end{enumerate}
\end{lemma}
\begin{proof}
Proofs can be found in \cite{BuchholzSchutte}.
\end{proof}

\begin{lemma}
If	$\be\in \OT$ then $G_0\be\subseteq \OT$ and $G_1\be\subseteq \OT$.
\end{lemma}
We say that {\em $\be$ is in normal form,} denoted $\be =_{NF} \psi_0(\al)$ if and only if $G_0(\al)<\al$.
Similarly, $\be =_{NF} \psi_1(\al)$ if and only if $G_1(\al)<\al$.  Therefore sets $G_0$ and $G_1$ are used to restrict the terms in $\T$ to terms in normal form in $\OT$. This way we ensure there is no descending sequence $\Om>\psi_0\Om>\psi_0(\psi_0\Om)>...$ since no terms such as $\psi_0(\psi_0(\psi_0\Om))$ are allowed in $\OT$.

\begin{theorem}[Buchholz and Sch\"utte \cite{BuchholzSchutte}]
$\;$
\begin{enumerate} 
    \item $\OT$ is well–ordered by $<$.
     \item  $|\ID_2| = \psi_0\psi_1\varepsilon_{\Omega_2+1}$.
\end{enumerate}
\end{theorem}

\subsection{Fundamental sequences for \texorpdfstring{$|\ID_2|$}{ID2} }\label{subsec:FundSeqBach}
Fundamental sequences provide canonical increasing approximations to limit ordinals
by sequences of smaller ordinals. They are essential in proof theory, in particular for defining fast-growing hierarchies of functions, ordinal notation systems and in ordinal analysis. 

\begin{definition}
Let $\Lambda$ be an ordinal. A {\em system of fundamental sequences on $\Lambda$} is a function $\cdot [\cdot] \colon \Lambda \times \mathbb N \to \Lambda$ such that $  \alpha [n] \leq \alpha$ with equality holding if and only if $\alpha = 0$, and $\alpha[n] \leq \alpha[m]$ whenever $n\leq m$.
The system of fundamental sequences is {\em convergent} if $\lambda  = \lim_{n\to \infty} \lambda[n]$ whenever $\lambda $ is a limit.
\end{definition}

Another key property of a system of fundamental sequences is the {\em Bachmann property}. The proof-theoretic ordinals for the theories we are interested in come equipped with natural system of fundamental sequences that satisfy this property. 
\begin{proposition}[Bachmann property]
    If $\al,\be <\Lambda$ and $k\in \Nat$ satisfy  $\al[k] < \be < \al$, then $\al[k] \leq \be[1]$.
\end{proposition}
The 
To illustrate these notions better, consider the natural system of fundamental sequences for $\Lambda=\varepsilon_0$, the first fixed point of the function $\xi \to \om^{\xi}$. 

\begin{definition}
    For $\al<\varepsilon_0$ and $n\in \Nat$, define $\al[n]$ recursively by:
    \begin{itemize}
        \item $0[n]:=1[n]:=0$,
        \item $(\om^\al+\be)[n]:=\om^\al+\be[n]$ if $\om^\al+\be$  is in Cantor normal form and $\be>0$,
        \item $\om^{\al+1}[n] = \om^\al n$,
        \item $\om^\la[n] = \om^{\la[n]} \text{ if } \la \in Lim.$
        \end{itemize}
\end{definition}
For $\al\neq 0$, we have$\al[n] <\al$  for all $n$, and if  $\al<\varepsilon_0$ is a limit ordinal, then $(\al[n])_{n<\om}$  is an increasing sequence converging to $\al$. This system of fundamental sequences also satisfies the Bachmann property, see \cite{Schmidt1977,Weiermann2006} for details.\medskip
\\
For example, $\om^\om[4] = \om^4$ and $\om^4 < \om^6 < \om^\om$, while $\om^4 < \om^5 = \om^6[1]$.
Intuitively, if $\be \in (\al[k], \al)$ and we repeatedly apply fundamental sequences to $\be$, we remain in the interval $(\al[k], \al)$, unless we pass through $\al[k]$. Note that this may fail if we replace $1$ by $0$ in the condition of Bachmann property as $\om^6[0] = 0 < \om^\om[4]$.\medskip\\


Now define a system of fundamental sequences for the ordinal $|\ID_2|$ which satisfies the Bachmann property.
For each term in $\T$ its fundamental sequence $\al[x]$ is defined for $x<tp(\al)$. 
Since fundamental sequences are defined on all of $\T$, the terms $\alpha[x]$ need not lie in the well–founded part $\OT \subseteq \T$. As usual, the fundamental sequence of a successor $\alpha+1$ is constantly $\alpha$.

\begin{definition} \label{def:FundamentalSeq} Definition of $tp(\al)$ and $\al[x]$ for $\al\in T$ and $x\in T$ with $x<tp(\al)$.
    \begin{itemize}
        \item $tp(0)=0$ and $tp(1)=1$ and $1[0]=0$.
        \item If $tp(\Om_i)=\Om_i$  for $i\geq 0$, and $\Om_i[x]=x$.
        
        \item If $\al =\al_0+\ldots +\al_n$ with $n\geq 1$ then $tp(\al):=tp(\al_n)$ and $(\al_0+\ldots +\al_n)[x] = \al_0+\ldots +\al_n[x]$.
        
        \item  If $tp(\al) < \Om_{i+1}$, then  $tp(\psi_i\al)= tp(\al)$ and $\psi_i(\al)[x]=\psi_i(\al[x])$.
        
        \item If $tp(\al) = \Om$ then  $tp(\psi_0\al)=\om$ and $\psi_0(\al)[x]=\psi_0(\al[{z_x}])$
        with $z_0:=0$ and $z_{x+1}:=\psi_0(\al[z_x])$.
        
        \item If $tp(\al) = \Om_2$ then  $tp(\psi_1\al)=\om$ and $\psi_1(\al)[x]=\psi_1(\al[{z_x}])$
        with $z_0:=0$ and $z_{x+1}:=\psi_1(\al[z_x])$.
        
        \item If $tp(\al) =1$ then $tp(\psi_i(\al))=\om$ and $\psi_i(\al)[x]= \psi_i(\al[0])\cdot x$.
    \end{itemize}
\end{definition}
For technical reasons, we put $0[0]=0$, $(\al+1)[x] =\al$ for all $x\in T$ and if $tp(\al)\in\{0,1\}$, we put $\al[m]:=\al[0]$ for all $m<\om$. 

\begin{lemma} For all $\alpha \in \T$ the following hold:
\begin{itemize}
    \item $tp(\al) =0 \Longleftrightarrow \al=0$.
    \item $tp(\al)=1 \Longrightarrow \al =\al[0]+1$.
    \item $x<tp(\al) \Rightarrow \al[x]<\al$
    \item $x<y <tp(\al) \Rightarrow \al[x]<\al[y]$.
\end{itemize}\smallskip
\end{lemma}
\begin{proof}
    By induction on the length of term $a$.
\end{proof}

\begin{definition}[Single–step descent]
For each $a,b \in \T$ and any natural number $n$, define the {\em single-step descent relation} by 
$a \triangleleft_n b :\iff \exists x<n, (a = b[x])$.
\\
Let $<_k$ be the transitive closure of $\{(\al, \be) : \al = \be[k]\}$, and $\leq_k$ the reflexive closure of $<_k$.
Thus, $b<_n a$ means there is a finite chain
$b = a_0\triangleleft_n a_1\triangleleft_n \cdots \triangleleft_n a_k = a.$
\end{definition}

\begin{lemma}\label{lm:propertyTC} The following properties hold:
\begin{enumerate}
    \item If $x<_n y<tp(\al)=\Om_{v+1}$ then $\al[x]<_n \al[y]$.
    \item \label{item:1 lm:propertyTC}If $tp(\al)=\om$ then $\al[x]<_n \al[x+1]$.
    \item  \label{item:2 lm:propertyTC}If $\be <_n \al$ and $tp(\al)=\om$ then $\be\leq_n\al[n]$.
\end{enumerate}
\end{lemma}
\begin{proof}
See \cite{Weiermann95}.
\end{proof}

The next theorem summarizes some key properties of the fundamental sequences in $\OT$.
\begin{theorem}\label{thm:PropFundSeqOT} Properties of system $\OT$.
    \begin{itemize}
        \item If $\al,x \in \OT$ and $x<tp(\al)$ then $\al[x] \in \OT$.
        \item If $\al \in \OT$ then $tp(\al)\in \OT$.
        \item If $\al,\be \in \OT$ and $\al[0]<\be<\al$ then $\exists x \in \OT$ such that $x<tp(\al)$ and $\al[x]\leq \be <\al[x+1]$.
    \end{itemize}
\end{theorem}
\begin{proof}
    Proofs can be found in \cite{BuchholzSchutte}.
\end{proof}

We now restrict ourselves to ordinal terms $\alpha < \varepsilon_{\Omega+1}$, defined in terms of $\OT$.
Define 
$$\OT_0:= \{\al \in \OT | \al< \varepsilon_{\Om+1} \text{ and } G_0\al <\varepsilon_{\Om+1} \}. $$ 
This restriction will be used in Section \ref{sec:PropID2}.

\begin{theorem} 
Let $\al,\be \in \OT$. If  $\al<\be$ in $\OT$ then  $\al<\be$ in $\OT_0$.
\end{theorem}

\begin{lemma}The following properties hold in $\OT_0$;
\begin{enumerate}
    \item If $\al\in \OT_0$ and $\be <_n \al$ then $\be+1 <_{n+1}\al$.
    \item $\al,\be \in \OT_0$ and $\al<\be$ then $\exists n (\be<_n\al)$.
\end{enumerate}
\end{lemma}

The system $\OT_0$ together with the fundamental sequences in Definition~\ref{def:FundamentalSeq} and the Norm function $N$, has the Bachmann property as shown in \cite{Weiermann95}.
\begin{theorem}\label{thm:bachmannProp} For $\al,\be ,\in \OT_0$ and $x\in \mathbf{N}$ the following holds: 
\begin{itemize}
    \item $\al[x] < \be <\al$ then $\al[x] \leq \be[1]$.
    \item $\al[x] < \be <\al$ then $N(\al[x])<N(\be)$.
    \item $N(\al)\leq N(\al[1]) +1$.
\end{itemize}
\end{theorem}
\begin{proof}
The proof of \cite{Weiermann95} can be adapted to the current context.
\end{proof}

\subsection{Proof-theoretic properties}\label{subsec:ProofTheo}

We will use the system of fundamental sequences to establish new independence results for Goodstein processes by appealing to proof-theoretic ordinals. If $\Lambda$ is an ordinal then for every $\alpha<\Lambda $ and every $m$ there is $n$ such that $ \la [0][1] \ldots [n] = 0$. However, this fact is not always provable when $\Lambda$ is very large.

Proof-theoretic ordinals mentioned earlier can also be defined in the context of an ordinal notation system with fundamental sequences. 
Each of the theories $\mathbb{T}$ we are interested in can be assigned a proof-theoretic ordinal $|\mathbb{T}|$, which provides substantial information about what is (un)provable in $\mathbb{T}$.
First, it characterizes the amount of transfinite induction available, such that $|\mathbb{T}|$ can be characterized as the least ordinal $\xi$ such that $\mathbb{T}\not\vdash {\rm TI}(\xi)$.

The ordinal $|\mathbb{T}|$ can also be used to bound the provably total computable functions in $\mathbb{T}$.
Fundamental sequences can be used to define fast-growing functions on
the natural numbers, such as the Fast Growing Hierarchy $F_\al$ in the following definition.
\begin{definition}
    For $n\in \Nat$ and $\al<\varepsilon_0$ define:
    \begin{itemize}
        \item $F_0(n)=n+1$
        \item $F_{\al+1}(n) = F^n_\al(n)$
        \item $F_\la(n) = F_{\la[n]}(n)$ for $\la\in Lim$.
    \end{itemize}
\end{definition}
The intuition is that each $F_\al$ is an increasing function, which grows
more quickly for larger $\al$.
The totality of these functions cannot be proven over weak theories.
\\

Given $\xi\leq |\mathbb{T}| $, let $H(\xi)$ be the least $n$ so that $\xi[0][1]\ldots [n] =0$, and for $n\in\mathbb N$, let $F_\xi(n) = H ( \xi [n])$.
Then, the proof-theoretic ordinal of $\mathbb{T}$ is also the least $\xi\leq |\mathbb{T}|$ such that $\mathbb{T} \not\vdash \forall \zeta<\xi \ \forall n \  \exists m \ ( m = F_\xi(n) )$, if it exists (otherwise, we may define it to be $\infty$, or just leave it undefined).

Say that a partial function $f\colon \mathbb N \to \mathbb N$ is {\em computable} if there is a $\Sigma_1$ formula $\varphi_f (x,y)$ in the language of first order arithmetic (with no other free variables) such that for every $m,n$, $f(m) = n$ if and only if $\varphi_f (m,n)$ holds.
The function $f$ is {\em provably total} in a theory $\mathbb{T}$ if $\mathbb{T}\vdash \forall x \exists y \varphi _f (x,y) $ (more precisely, $f$ is provably total if there is at least one such choice of $ \varphi _f $).

\begin{theorem}
For $\mathbb{T}\in \{ \ID_2\}$ and $\al\leq \psi_0(\psi_1\varepsilon_{\Om_2+1})$, the following are equivalent~\cite{BuchholzCW1994}:
\begin{enumerate}

\item $\al<|\mathbb{T}|$.

\item $\mathbb{T} \vdash {\rm TI}( \al)$.

\item $F_\al$ is provably total in $\mathbb{T}$.

\item There exists a provably total computable function $f\colon \Nat\to \Nat$ in $\mathbb{T}$ such that $f(n) > F_{|\mathbb{T}|}(n)$ for all $n$.
\end{enumerate}
\end{theorem}
Thus a general strategy for proving independence of $\Pi^0_2$ statements is
showing that they require witnesses growing faster than the suitable fast-growing
function.
The Bachmann property will be useful in transferring unprovability results stemming from proof-theoretic ordinals to the setting of Goodstein processes in view of the following.

\begin{proposition}\label{prop:Majorize}
Let $\Lambda$ be an ordinal with a system of fundamental sequences satisfying the Bachmann property, and let $(\xi_n)_{n\in \mathbb N}$ be a sequence of elements of $\Lambda$ such that, for all $n$, $\xi_n[n+1] \leq \xi_{n+1} \leq \xi_n$.
Then, for all $n$, $\xi_n \geq \xi_0 [0][1] \ldots [n]$.
\end{proposition}

\proof
Let $\leq_k$ be the reflexive transitive closure of $\{(\al[k],\al):\al<\varphi_2(0)\}$. We need a few properties of these orderings.
           Clearly, if $\al \leq_k \be$, then $\al \leq \be$.
                     It can be checked by a simple induction and the Bachmann property that, if $\alpha[n] \leq \beta < \alpha$, then $\alpha[n] \leq_1 \beta$.
Moreover, $\leq_k$ is monotone in the sense that if
          $\al\leq_k \be$, then $\al\leq_{k+1}\be$, and if $\alpha\leq _k \beta$, then $\alpha[k] \leq_k \beta[k]$ (see, e.g., \cite{Schmidt1977} for details).

We claim that for all $n$, $\xi_n \succeq _n \xi_0[1]\ldots[n]$, from which the desired inequality immediately follows.
For the base case, we use the fact that $\geq_0$ is transitive by definition.
For the successor, note that the induction hypothesis yields $\xi_0[1]\ldots[n] \leq_n \xi_n $, hence $\xi_0[1]\ldots[n+1] \leq_{n+1} \xi_n[n+1] $.
Then, consider three cases.

\begin{enumerate}[label*={\sc Case \arabic*},wide, labelwidth=!, labelindent=0pt]

\item ($\xi_{n+1} = \xi_n$).
By transitivity and monotonicity, $\xi_0[1]\ldots [n+1] \leq_{n+1} \xi_0[1]\ldots [n ] \leq _n \xi_n =  \xi_{n+1}$
yields $\xi_0[1]\ldots [n+1] \leq_{n+1}  \xi_{n+1}$.

\item ($\xi_{n+1} = \xi_n[n+1]$).
Then, $\xi_0[1]\ldots [n+1] \leq_{n+1} \xi_n[n+1] = \xi_{n+1}$.

\item ($\xi_n[n+1] < \xi_{n+1} < \xi_n$).
The Bachmann property yields $\xi_n[n+1] \leq_1 \xi_{n+1}$, and since $\xi_0[1]\ldots[n+1] \leq_{n+1} \xi_n[n+1] $, monotonicity and transitivity yield $\xi_0[1]\ldots[n+1] \leq_{n+1} \xi_{n+1} $.\qedhere

\end{enumerate}
\endproof

As an immediate corollary we obtain that for such a sequence of ordinals $(\xi_n)_{n\in\mathbb N}$, if $m$ is least so that $\xi_m = 0$, then $m\geq F_{\xi_0}(n)$.
So, our strategy will be to show that the respective Goodstein process for $\mathbb{T}$ grows at least as quickly as $F_{|\mathbb{T}|}$, from which we obtain that the theorem is unprovable in $\mathbb{T}$.

\section{Properties of the ordinal notation system for \texorpdfstring{$\ID_2$}{ID2}} \label{sec:PropID2}

In this section we provide the technical foundation for the Goodstein process analyzed in Section~\ref{sec:GoodsteinID2}. Throughout this section, we work exclusively within the ordinal notation system $\OT_0$ established in Section~\ref{subsec:FundSeqBach}, relying on the system of fundamental sequences and the Bachmann property (Theorem~\ref{thm:bachmannProp}).
We introduce the concept of an \emph{ordinal context} to facilitate the comparison of uncountable and countable ordinals within $\OT_0$. The technical results derived here, specifically the comparison criteria and growth bounds in Lemmas~\ref{lm:Context1} and \ref{lm:Context2}, are essential for proving the monotonicity and normal form preservation of the base change operation defined in Section~\ref{subsec:BaseChange}.

\subsection{Ordinal Context }
To compare ordinals $\alpha < \beta$ in $\OT_0$, we must analyze the first difference between their structures. In standard Cantor normal form, this corresponds to identifying the first index $i$ such that the exponents $\alpha_i \neq \beta_i$. Because our notation system involves nested $\psi$-collapsing functions, this structural difference is obscured. We introduce an \emph{ordinal context} to factor out the common prefix of $\alpha$ and $\beta$, isolating the subterm where the divergence occurs.\medskip
\\
\noindent \textbf{Schematic Representation:}
We view an ordinal context $\lambda\scbr{\cdot}$ as a function with a slot $\scbr{\cdot}$. For two ordinals $\alpha < \beta$ that share a common prefix and suffix, we express them as:
\[
\begin{array}{cc}
\alpha := & \lambda[[\gamma]] \\
\beta  := & \lambda^- [[\delta]],
\end{array}
\]
where $\lambda\scbr{\cdot}$ captures the shared structure, $\gamma$ and $\delta$ are the {\em differing} subterms (with $\gamma < \delta$), 
and $\lambda^- \scbr{\cdot}$ accounts for the truncation of terms following the divergence point. This is formally stated in the following definitions.

\begin{definition}[Ordinal Context]
Let $\psi_1\al_0 \geq \cdots \geq \psi_1\al_n $ and $\psi_0\be_0 \geq \ldots \geq \psi_0\be_m$ be additive principal ordinals in $\OT_0$ and $l$ a natural number.

\begin{enumerate}[leftmargin=*]
    \item An ordinal context $\lambda\scbr{\cdot}$ is a formal expression with a placeholder $\scbr{\cdot}$ representing an ordinal term at which $\alpha$ and $\beta$ differ.
    \begin{enumerate}[label=(\roman*)]
        \item $\la\scbr{\cdot} := \psi_1{\al_0} +\ldots  + \psi_1{\al_{i-1}}+ \scbr{\cdot}+\psi_1{\al_{i+1}} +\ldots + \psi_0\be_n+l$ where $\scbr{\cdot}$ is a placeholder for the ordinal term $\psi_1\al_i$, and $\la\scbr{\psi_1\al_i} = \psi_1{\al_0} +\ldots  + \psi_1{\al_{i-1}}+ \psi_1\al_i+\psi_1{\al_{i+1}} +\ldots + \psi_0\be_n+l = \al$ for an $\al \in \OT_0$.
        
        \item Similarly, $\la\scbr{\cdot} := \psi_1{\al_0} +\ldots + \psi_0\be_0 + \ldots + \psi_0\be_{j-1} +\scbr{\cdot} + \psi_0\be_{j+1} +\ldots +l$ is an ordinal context where $\scbr{\cdot}$ is a placeholder for the ordinal term $\psi_0\be_j$, and 
        $\la\scbr{\psi_0\be_j} = \psi_1{\al_0} + \ldots + \psi_0\be_{j-1} + \psi_0\be_i+ \psi_0\be_{j+1} +\ldots +l =\al$ for an $\al \in \OT_0$.
        
        \item  $\la\scbr{\cdot} :=  \psi_1\al_0 + \ldots + \psi_0\be_m+ \scbr{\cdot}$ is an ordinal context where $\scbr{\cdot}$ is a placeholder for $l$, and $\la\scbr{l} =  \psi_1\al_0 + \ldots + \psi_0\be_m+ l=\al$ for an $\al \in \OT_0$.        
    \end{enumerate}
    
    \item Let $\mu\scbr{\cdot}$ be an ordinal context then,
    \begin{enumerate}[label=(\roman*)]
        \item $\la\scbr{\cdot}:= \psi_1{\al_0} +\ldots+ \psi_1(\mu\scbr{\cdot})+\ldots +l$ is an ordinal context where $\scbr{\cdot}$ is a placeholder 
        for $\al_i$ and $\la\scbr{\al_i} = \psi_1{\al_0} +\ldots+ \psi_1(\mu\scbr{\al_i})+\ldots +l= \al$ for an $\al \in \OT_0$.
        \item $\la\scbr{\cdot}:= \psi_1{\al_0} +\ldots+ \psi_0(\mu\scbr{\cdot})+\ldots +l$ is an ordinal context for $\be_j$ in $\al$, and 
        $\la\scbr{\be_j}= \psi_1{\al_0} +\ldots+ \psi_0(\mu\scbr{\be_j})+\ldots +l =\al $ for an $\al\in \OT_0$.
    \end{enumerate}
\end{enumerate}
\end{definition}
\smallskip
The following examples illustrate the definition of an ordinal context.
\begin{example}
Let $\al= \psi_1{(\om\cdot 3)}  + \psi_0{(\psi_1\om)}+ 11$. 
Define an ordinal context as $\la\scbr{\cdot} :=  \scbr{\cdot} + \psi_0{(\psi_1\om)} +11$. To obtain $\al$ from the ordinal context $\la\scbr{\cdot} $, set $\psi_1{(\om\cdot 3)}$ into the placeholder, so $\al = \la\scbr{\psi_1{(\om\cdot 3)}}$.
\medskip\\
Let $\be = \psi_1(\om)$, and define an ordinal context $\tau\scbr{\cdot} :=\scbr{\cdot}$. Then $\be=\tau\scbr{\psi_1\om}$. 
Define $\la_1\scbr{\cdot}:= \psi_1(\om\cdot 3) + \psi_0(\tau\scbr{\cdot}) +11$.
Then $\al =\psi_1{(\om\cdot 3)}  + \psi_0{(\psi_1\om)}+ 11$ can be written using the $\la_1$ context as:
$\al= \la_1\scbr{\psi_1\om}$.
\end{example}

\smallskip

\begin{definition}[Truncated ordinal context]
Let $\la\scbr{\cdot}$ be an ordinal context. The truncated context $\lambda^-\scbr{\cdot}$ is the context resulting from removing all terms hereditarily following the placeholder.
\begin{itemize}
    \item If $\la\scbr{\cdot} = \scbr{\cdot}$ then $\la^-\scbr{\cdot} =\scbr{\cdot}$.
    
    \item If $\la\scbr{\cdot}=\al_0 + \ldots +\al_{i-1} +\scbr{\cdot} + \al_{i+1}+\ldots +l$ then 
    $\la^-\scbr{\cdot}=\al_0 + \ldots +\al_{i-1} + \scbr{\cdot}$.
    
    \item If $\la\scbr{\cdot}= \al_0 + \ldots +\al_{i-1}+ \psi_1(\tau{\scbr{\cdot}})+\al_{i+1}+\ldots +l$ for an ordinal context $\tau$,
    then $\la^-\scbr{\cdot}=\al_0 + \ldots +\al_{i-1}+ \psi_1(\tau^-{\scbr{\cdot}})$.
    
    \item If $\la\scbr{\cdot}=\al_0 + \ldots +\al_{i-1}+ \psi_0(\tau{\scbr{\cdot}})+\al_{i+1}+\ldots +l$ for an ordinal context $\tau$,
    then $\la^-\scbr{\cdot}= \al_0 + \ldots +\al_{i-1}+ \psi_0(\tau^-{\scbr{\cdot}})$.		
\end{itemize}
\end{definition}
\smallskip

\begin{example}
We employ ordinal contexts to facilitate the comparison of ordinals, as demonstrated in the following examples.\\
Let  $\al= \psi_1{(\Om + \om\cdot 2)} + \psi_1{(\om\cdot 3)}  + \psi_0{(\psi_1 \om + \psi_07)}+ 11$.\medskip
\\
If $\ga = \psi_1{(\Om+ \om\cdot 2)} + \psi_1{(\om^2)}$, then the first terms $\al$ and $\ga$ differ are $\psi_1{(\om\cdot 3)}$ and $\psi_1(\om^2)$. Using the ordinal context  emphasizes their differences. \\
Define $\la\scbr{\cdot} :=  \psi_1{(\Om + \om\cdot 2)} + \scbr{\cdot} + \psi_0{(\psi_1 \om + \psi_07)} +11$. Then $\al = \la\scbr{\psi_1{(\om\cdot 3)}}$. To get $\ga$ from the context $\la$ place into the placeholder $\psi_1{(\om^2)}$ and remove the terms on its right; $\ga = \la^-\scbr{\psi_1{(\om^2})}$.\medskip
\\
Let $\be = \psi_1\om + \psi_07$. 
Define an ordinal context $\tau\scbr{\cdot} :=\scbr{\cdot} + \psi_07$. Then $\be = \tau\scbr{\psi_1\om}$.
\\
If $\de = \psi_1{(\Om + \om\cdot 2)} + \psi_1{(\om\cdot 3)}  +  \psi_0{(\psi_15 )}$, the first terms where $\al$ and $\de$ differ are $\psi_0(\psi_1 \om + \psi_07)$ and $\psi_0{(\psi_15 )}$ respectively.
Using the ordinal context $\tau$, 
define $\la\scbr{\cdot} := \psi_1{(\Om + \om \cdot 2)} + \psi_1{(\om\cdot 3)} + \psi_0(\tau\scbr{\cdot}) + 11$. Then $\al = \la\scbr{\psi_1\om}$. 
We get $\de$ by truncating $\la$ and  $\tau$; hence\\
$\la^-\scbr{\cdot} =  \psi_1{(\Om + \om \cdot 2)} + \psi_1{(\om\cdot 3)} + \psi_0(\tau^-\scbr{\cdot})$. Then $\de = \la^-\scbr{\psi_15}$.
\\
\end{example}

\begin{remark}
For sake of brevity, we will write $\al = \xi + \psi_1\al_i + \eta$ instead of $\al =\psi_1{\al_1}+\ldots+\psi_1\al_i +\ldots +l$. 
Similarly, we will write $\al= \xi + \psi_0\be_i +\eta$ instead of $\al =\psi_0{\be_1}+\ldots+\psi_0\be_i +\ldots +l$.
\end{remark}

\smallskip
\subsection{Properties of an ordinal context}
An ordinal context is a formal expression with a placeholder, and as such it can be an ordinal $\al$ in the system $\OT_0$, only if the right term is replaced in the placeholder. The following Lemmas show some properties of an ordinal context.\medskip


    

\begin{lemma} \label{lm:lala-}
Let $\la\scbr{\cdot}$ be an ordinal context. 
Then for all $\ga<\Om$ such that $\la\scbr{\ga} \in \OT_0$,  $\la^-\scbr{\ga} \leq_{1} \la\scbr{\ga}$.
\end{lemma}
\begin{proof}
Proof by induction on the context $\la$.
\begin{Cases}

\item If $\la\scbr{\cdot} = \xi + \scbr{\cdot} + \eta$, then $\la^-\scbr{\cdot} = \xi +\scbr{\cdot}$. 
If $\eta=1$, then $\la^-\scbr{\ga}=\la\scbr{\ga}[1]$.\\
Otherwise, $\la\scbr{\ga} = \xi+\ga+\eta \geq_1 \xi+\ga+\eta[1] \geq_1\ldots \geq_1 \xi+\eta =\la^-\scbr{\ga}$.


\item If $\la\scbr{\cdot} = \xi + \psi_0(\mu\scbr{\cdot}) + \eta$, then $\la^-\scbr{\cdot} = \xi + \psi_0(\mu^-\scbr{\cdot})$. 
Then by induction hypothesis, $\mu^-\scbr{\ga} \leq_1 \mu\scbr{\ga}$. Then $\psi_0(\mu^-\scbr{\ga}) \leq_1 \psi_0( \mu\scbr{\ga})$,
therefore 
$\la^-\scbr{\ga} = \xi + \psi_0(\mu^-\scbr{\ga}) \leq_1 \xi + \psi_0(\mu\scbr{\ga}) + \eta =\la\scbr{\ga}$.

\item If $\la\scbr{\cdot} = \xi + \psi_1(\mu\scbr{\cdot}) + \eta$, then $\la^-\scbr{\cdot} = \xi + \psi_1(\mu^-\scbr{\cdot})$. 
Then similarly to above,
$\la^-\scbr{\ga} = \xi + \psi_1(\mu^-\scbr{\ga}) \leq_1 \xi + \psi_1(\mu\scbr{\ga}) + \eta =\la\scbr{\ga}$.
\end{Cases}
\end{proof}

\begin{lemma}
Let $\la\scbr{\cdot}$ be an ordinal context such that $\al =\la\scbr{\ga}$ is in $\OT_0$ for $\ga\in \OT_0$. Then $\la^-\scbr{\ga}$ is also in $\OT_0$.\medskip
\end{lemma}
\begin{proof}
By induction on $\la$:
\begin{Cases}
    \item  Let $\la\scbr{\cdot}=\xi +\scbr{\cdot} + \eta$ and $\al =\la\scbr{\ga} \in \OT_0$ such that $\al =_{NF} \xi + \ga + \de$. 
    Then $\xi\geq \ga \geq \de$ are additive principal ordinals.\\ 
    By Theorem \ref{thm:PropFundSeqOT}, if $\al\in \OT_0$ then also $\al[x] \in \OT_0$ and by Lemma \ref{lm:lala-} above,
    $\la\scbr{\ga}= \xi + \ga + \de > \xi + \ga + \de[1] \geq_1 \xi + \ga= \la^-\scbr{\ga}$.
    Hence $\xi+\ga = \la^-\scbr{\ga}$ is also an ordinal in $\OT_0$.\medskip
    
    \item  Let $\la\scbr{\cdot}=\xi+ \psi_1(\tau{\scbr{\cdot}})+\eta$ for an ordinal context $\tau$ such that $\al=\la\scbr{\ga} =\xi+\psi_1(\tau\scbr{\ga})+\eta \in \OT_0$ for  $\xi\geq \psi_1(\tau^-\scbr{\ga}) \geq \eta$ and $\xi,  \psi_1(\tau^-\scbr{\ga}) $ additive principal ordinals.
    \\
    By induction hypothesis, $\tau^-{\scbr{\ga}}$ is in $\OT_0$, and if $G_1(\tau^-{\scbr{\ga}})< \tau^-{\scbr{\ga}}$, then 
    also $\psi_1(\tau^-\scbr{\ga}) \in \OT_0$.
    \\
    Then $\la\scbr{\ga}= \xi+\psi_1(\tau\scbr{\ga})+\eta \geq_1 \xi+ \psi_1(\tau^-{\scbr{\ga}}) =\la^-\scbr{\ga}$ is in $\OT_0$.\medskip
    
    \item Let $\la\scbr{\cdot}=\xi+ \psi_0(\tau{\scbr{\cdot}})+\eta$ for an ordinal context $\tau$ such that $\al= \xi+\psi_0(\tau\scbr{\ga})+\eta  \in \OT_0$ as in the case above.
    By induction hypothesis, $\tau^-{\scbr{\ga}}$ is in $\OT_0$, and if $G_0(\tau^-{\scbr{\ga}})< \tau^-{\scbr{\ga}}$, then 
    also $\psi_0(\tau^-\scbr{\ga}) \in \OT_0$.
    Then $\la\scbr{\ga}= \xi+\psi_0(\tau\scbr{\ga})+\eta \geq_1 \xi+ \psi_0(\tau^-{\scbr{\ga}}) =\la^-\scbr{\ga}$ is in $\OT_0$.
    
\end{Cases}
\end{proof}

\smallskip
\begin{example} The following example illustrates the application of the proved Lemmas:\\
Let $\la\scbr{\cdot} = \psi_0(\om) + \scbr{\cdot} +\om +11$.
For $\ga= \psi_03$, $\la\scbr{\psi_03} = \psi_0(\om) + \psi_03 +\om +11$ is in $\OT_0$. Also $\la^-\scbr{\psi_0 3} = \psi_0\om + \psi_03$  is in $\OT_0$ as well.
However, for $\ga=\psi_0\Om$, $\la\scbr{\psi_0\Om} = \psi_0\om + \psi_0\Om +\om+11$ is not in $\OT_0$, neither $\la^-\scbr{\psi_0(\Om)} = \psi_0\om +\psi_0\Om$.\medskip
\end{example}

\smallskip

\begin{lemma}
Let $\ga <\de\in \OT_0$, and $\de\in AP$ and $\la\scbr{\cdot}$ is an ordinal context, such that $\la\scbr{\ga}$  and $\la^-\scbr{\de}$ are in $\OT_0$, then $\la\scbr{\ga} < \la^-\scbr{\de}$.\medskip
\end{lemma}
\begin{proof}
Proof by induction on the context $\la$.
\begin{Cases}
    \item If $\la\scbr{\cdot}=\xi+  \scbr{\cdot}+\eta$ and $\la^-\scbr{\cdot}=\xi +\scbr{\cdot}$, 
    then for $\ga<\de \in AP$, $\ga+\eta<\de$.
    Hence $\la\scbr{\ga} =\xi + \ga + \eta < \xi + \de  =  \la^-\scbr{\de} $.\medskip
    
    \item  If $\la\scbr{\cdot}=\xi +  \psi_0(\mu\scbr{\cdot})+\eta$ and $\la^-\scbr{\cdot}=\xi+\psi_0(\mu^-\scbr{\cdot})$.\\
    Then by induction hypothesis, $\mu\scbr{\ga} <\mu^-\scbr{\de}$. If $G_0(\mu\scbr{\ga})< \mu\scbr{\ga}$ and $G_0(\mu^-\scbr{\de}) < \mu^-\scbr{\de}$, then
    $\la\scbr{\ga} =\xi+  \psi_0(\mu\scbr{\ga})+\eta < \xi+\psi_0(\mu^-\scbr{\de}) =\la^-\scbr{\de} $.\medskip
    
    \item  If $\la\scbr{\cdot}=\xi + \psi_1(\mu\scbr{\cdot})+\eta$ and $\la^-\scbr{\cdot}=\xi+\psi_1(\mu^-\scbr{\cdot})$.\\
    Then by induction hypothesis, $\mu\scbr{\ga} <\mu^-\scbr{\de}$.
    If $G_1(\mu\scbr{\ga})< \mu\scbr{\ga}$ and $G_1(\mu^-\scbr{\de}) < \mu^-\scbr{\de}$, then
    $\la\scbr{\ga} =\xi+  \psi_1(\mu\scbr{\ga})+\eta < \xi+\psi_1(\mu^-\scbr{\de}) =\la^-\scbr{\de} $.
\end{Cases}
\end{proof}

\smallskip

To ensure that the difference between ordinals can be analyzed directly via the fundamental sequences without interference from further collapses, we define the restriction to \emph{$\psi_0$-nesting free contexts} where the placeholder $\scbr{\cdot}$ is not inside the scope of a $\psi_0$ function. 
\\
This machinery allows us to relate the ordering $\alpha < \beta$ directly to the properties of the fundamental sequences $\beta[t]$, effectively providing a robust tool to perform comparisons analogous to those in standard ordinal arithmetic. \smallskip

\begin{definition}[$\psi_0-$nesting free ordinal context]
Let $\la\scbr{\cdot}$ be an ordinal context. 
Then $\la\scbr{\cdot}$ is a $\psi_0-$nesting free context if the placeholder $\scbr{\cdot}$ is not in the scope of a $\psi_0$ function.
\begin{enumerate}[label=(\roman*)]
    \item $\la\scbr{\cdot} =\xi+\scbr{\cdot}+\eta$ is $\psi_0-$nesting free context.
    \item If $\mu$ is a $\psi_0-$nesting free context, then $\la\scbr{\cdot} = \xi + \psi_1(\mu\scbr{\cdot}) +\eta$ is a $\psi_0-$nesting free context.
    Additionally, $\la\scbr{\cdot} = \xi +\mu\scbr{\cdot} +\eta$ is $\psi_0-$nesting free context.
\end{enumerate}
\end{definition}
\smallskip
We now present several illustrative examples.\smallskip

\begin{example}
Let $\al= \psi_1(\Om+1)+\psi_1({\psi_0\Om})  + \psi_0{3}  +11 $. Then $\la\scbr{\cdot}:=  \psi_1\Om + \psi_1{ \scbr{\cdot}} + \psi_0{3} +11 $ is a $\psi_0-$nesting free context for the placeholder.
Also $\la\scbr{\cdot} := \psi_1\Om + \scbr{\cdot} + \psi_0{3} +11$ is a $\psi_0-$nesting free ordinal context.
\\
However, $\la\scbr{\cdot} = \psi_1{\Om}+ \psi_1({\psi_0\scbr{\cdot}}) + \psi_0{3}  +11 $ is not a $\psi_0-$nesting free context since the placeholder $\scbr{\cdot}$ is under the scope of a $\psi_0$ function.
\\
\end{example}

\begin{remark}
If $\be = \xi +\psi_0(\eta + \Om)$, then $tp(\psi_0(\eta+\Om)) = \om$ and $tp(\eta+\Om) =\Om$. Then an ordinal context for $\psi_0(\eta +\Om)$ is always
$\la\scbr{\cdot} =\xi + \scbr{\cdot}$ and cannot be $\la\scbr{\cdot} =\xi+ \psi_0(\scbr{\cdot})$.
\\
Then $\be =\la\scbr{\psi_0(\mu\scbr{\Om})}$, for a $\psi_0-$nesting free context $\mu$. 
The context $\mu$ is used for the ordinals, in this particular case $\eta$, which don't affect the type of the inner argument of $\psi_0$.
\\
If $\be = \xi + \psi_1\Om$, then $tp(\psi_1(\Om)) =\Om$ and an ordinal context for $\psi_1\Om$ can be written as $\la\scbr{\Om}=\xi + \psi_1\scbr{\Om}$ or $\la\scbr{\psi_1(\Om)} = \xi +\scbr{\psi_1\Om}$ as in both cases the type of the ordinal $\psi_1\Om$ doesn't change.\\
\end{remark}

\begin{lemma} \label{lm:fundamOrd}
Let $\la\scbr{\cdot}$ be an ordinal context such that $\be = \la^-\scbr{\de} \in \OT_0$.
Then if $tp(\de) = \om$, $\be[t] = \la^-\scbr{\de[t]}$.
\end{lemma}

\begin{proof} 
Proof by induction on $\la$ by a subsidiary induction on $\de$.
\begin{Cases}
\item $\la^-\scbr{\cdot} =\xi + \scbr{\cdot}$. Then $\be=\xi+\de=\la^-\scbr{\de}$.

\begin{Cases}
\item ($\de=\om$). Then $\be[t]=\la^-\scbr{\om}[t] = \xi + \om[t]= \xi+t=\la^-\scbr{t} = \la^-\scbr{\om[t]}$.\medskip

\item $(\de = \psi_0(\ga+1))$. \label{case:psi0 lm:fundamOrd}
Then $\be[t] = \xi + \psi_0(\ga+1)[t] = \xi+ \psi_0(\ga)\cdot t = \la^-\scbr{\psi_0(\ga)\cdot t} = \la^-\scbr{\psi_0(\ga+1)[t]}$.\medskip

\item $(\de = \psi_0(\psi_1(\ga+1)))).$ \label{case:psi1 lm:fundamOrd}
Then $\be[t] = \xi + \psi_0(\psi_1(\ga+1))[t] = \xi+ \psi_0(\psi_1(\ga+1)[t]) = 
\xi+ \psi_0(\psi_1(\ga) \cdot t) = \la^-\scbr{\psi_0(\psi_1(\ga)\cdot t)} = \la^-\scbr{\psi_0(\psi_1(\ga+1))[t]}$.\medskip

\item ($\de= \psi_0(\mu\scbr{\Om})$ for $\mu$ a $\psi_0-$nesting free context. First, recall the definition of fundamental sequences for this case. 
Let $\al= \psi_0(\al_0)$,  $tp(\al) =\om$ and $\al_0 :=\mu\scbr{\Om}$. 
Then $\al[x] = \psi_0(\al_0[z_x])$  where $z_0= 0$ and $z_{x+1} = \psi_0(\al_0[z_{x}])$.
\\
Then $\be[t] = \xi+ \psi_0(\mu\scbr{\Om})[t] = \xi+\psi_0(\mu\scbr{\Om}[z_t]) = \la^-\scbr{\psi_0(\mu\scbr{\Om}[z_t])}  = \la^-\scbr{\psi_0(\mu\scbr{\Om})[t]}$.\medskip

\item $(\de= \psi_1(\ga+1)$. 
Then $\be[t] = \xi + \psi_1(\ga+1)[t] = \xi+ \psi_1(\ga)\cdot t = \la^-\scbr{\psi_1(\ga)\cdot t} = \la^-\scbr{\psi_1(\ga+1)[t]}$.\medskip

\end{Cases}

\item $\la^-\scbr{\cdot} =\xi + \psi_0(\scbr{\cdot})$. Then $\be=\xi+\psi_0(\de) = \la^-\scbr{\de}$.
\begin{Cases}
\item  $(\de = \sigma^-\scbr{\om})$ for an ordinal context $\sigma$.
In this case we consider, for instance, the cases $\be=\xi+\psi_0(\om)$ or $\de= \psi_0(\psi_0(\om))$.
\\
Then by induction hypothesis, $\sigma^-\scbr{\om}[t] =\sigma^-\scbr{\om[t]}$.
Define $\la^-\scbr{\cdot} = \xi +\sigma^-\scbr{\cdot}$, then $\be=\la^-\scbr{\om}$.
And we have $\be[t] = \xi +\sigma^-\scbr{\om}[t] = \xi + \sigma^-\scbr{\om[t]} = \la^-\scbr{\om[t]}$.\medskip

\item $(\de = \sigma^-\scbr{\psi_0(\ga+1)})$  for an ordinal context $\sigma$.
We cover cases such as $\be = \xi+ \psi_0(\psi_0(\ga+1))$.
By induction hypothesis \ref{case:psi0 lm:fundamOrd}, 
$\sigma^-\scbr{\psi_0(\ga+1)}[t] = \sigma^-\scbr{\psi_0(\ga+1)[t]}$.
\\
Define $\la^-\scbr{\cdot} = \xi +\sigma^-\scbr{\cdot}$, then $\be=\la^-\scbr{\psi_0(\ga+1)}$.
And we have $\be[t] = \xi +\sigma^-\scbr{\psi_0(\ga+1)}[t] = \xi + \sigma^-\scbr{\psi_0(\ga+1)[t]} = \la^-\scbr{\psi_0(\ga+1)[t]}$.\medskip

\item $(\de = \sigma^-\scbr{\psi_1\ga+1)})$  for an ordinal context $\sigma$.
We cover cases such as $\be = \xi+ \psi_0(\psi_1(\psi_1(\ga+1)))$.
By induction hypothesis \ref{case:psi1 lm:fundamOrd},
$\sigma^-\scbr{\psi_1(\ga+1)}[t] = \sigma^-\scbr{\psi_1(\ga+1)[t]}$.\\
Define $\la^-\scbr{\cdot} = \xi +\sigma^-\scbr{\cdot}$, then $\be=\la^-\scbr{\psi_1(\ga+1)}$.
And we have $\be[t] = \xi +\sigma^-\scbr{\psi_1(\ga+1)}[t] = \xi + \sigma^-\scbr{\psi_1(\ga+1)[t]} = \la^-\scbr{\psi_1(\ga+1)[t]}$.\medskip

\end{Cases}

\item $\la^-\scbr{\cdot} =\xi + \psi_1(\scbr{\cdot})$. Then $\be=\xi+\psi_1(\de) = \la^-\scbr{\de}$.
\begin{Cases}  
\item $(\de= \sigma^-\scbr{\psi_1(\ga+1)}$ for an ordinal context $\sigma$. 
Then $\be[t] = \xi + \psi_1(\ga+1)[t] = \xi+ \psi_1(\ga)\cdot t = \la^-\scbr{\psi_1(\ga)\cdot t} = \la^-\scbr{\psi_1(\ga+1)[t]}$.\medskip

\item ($\de =\sigma^-\scbr{\psi_0(\mu\scbr{\Om}}$).
In this case we consider, for instance, the cases $\be = \xi + \psi_1(\psi_0(\Om))$.
By induction hypothesis, $\sigma^-\scbr{\psi_0(\mu\scbr{\Om}}[t] = \sigma^-\scbr{\psi_0(\mu\scbr{\Om}[t]}$. 
Define $\la^-\scbr{\cdot} = \xi + \sigma^-\scbr{\cdot}$, then $\be = \la^-\scbr{\psi_0(\mu\scbr{\Om}}$.
And $\be[t] = \xi + \sigma^-\scbr{\psi_0(\mu\scbr{\Om}}[t] = \xi + \sigma^-\scbr{\psi_0(\mu\scbr{\Om}[t]} = \la^-\scbr{\psi_0(\mu\scbr{\Om}[t]}$.\medskip
\end{Cases}

\end{Cases}
\end{proof}

\begin{lemma}\label{lm:SandwichContext}
Let $\de,\ga,\tilde{\de}\in \OT_0$ and $\la$ be a context in $\OT_0$  such that 
$\la^-\scbr{\de} < \la\scbr{\ga} < \la^-\scbr{\tilde{\de}}$, 
then $\la^-\scbr{\de} < \la^-\scbr{\ga} < \la^-\scbr{\tilde{\de}}$.
\end{lemma}
\begin{proof}
By induction on $\la$.
\begin{Cases}
    \item If $\la=\xi+ \mu\scbr{\cdot}+\eta$,
    $$\xi + \mu\scbr{\de}<\xi + \mu\scbr{\ga}+\eta <\xi + \mu\scbr{\tilde{\de}},$$
    $$\mu\scbr{\de} < \mu\scbr{\ga}+\eta < \mu\scbr{\tilde{\de}},$$
    $$\mu\scbr{\de} < \mu\scbr{\ga} < \mu\scbr{\tilde{\de}},$$
    since $\mu\scbr{\ga},\eta \in AP$.\smallskip

\item 
If $\la=\xi+ \mu\scbr{\cdot}+\eta$ and $\la^-= \xi+ \mu^-\scbr{\cdot}$,
$$\xi + \mu^-\scbr{\de}<\xi + \mu\scbr{\ga}+\eta <\xi + \mu^-\scbr{\tilde{\de}}$$
By the induction hypothesis, 
$\mu^-\scbr{\de} < \mu^-\scbr{\ga} < \mu^-\scbr{\tilde{\de}}$.
\end{Cases}
\end{proof}

The following Lemma explores the differences between two uncountable ordinals. Our examination will focus on the specific ordinal context in which these distinctions occur. Additionally, we will impose the constraint of a $\psi_0$-nesting free context.

\begin{lemma}\label{lm:Context1}
Let $\al<\be \in \OT_0$ then one of the following occurs,
\begin{enumerate}
    \item $\al+1=\be$,
    \item $\al<\be[1]$,
    \item there is a $\psi_0$-nesting free context $\la$, an ordinal $\ga$ and a natural number $s$  such that $\al = \la\scbr{\psi_1(\ga) \cdot s}$ and $\be = \la^-\scbr{\psi_1(\ga+1)}$ and $\al<\be[s+1]$,
    \item there is a $\psi_0$-nesting free context $\la$, ordinals $\ga,\de$ with $\ga<\de< \Om$ and a natural number $t$, such that  $ \al=\la\scbr{\ga}$ and $\be=\la^-\scbr{\de}$ and	 $\be[t] \leq \al < \be[t+1]$,
    \item there is a $\psi_0$-nesting free context $\la$, an ordinal $\ga<\Om$ and an ordinal $\tau<\Om$ such that  $ \al=\la\scbr{\ga}$ and  $\be=\la^-\scbr{\Om}$ and $ \be[\tau] \leq \al<\be[\tau+1]$ .
\end{enumerate}
\end{lemma}
\begin{proof}
Proof by case distinction on the form of $\be$.
\\
Let $\al=\psi_1{\al_0} + \ldots + \psi_1{\al_m} + \al_{m+1} $ and $\be = \psi_1{\be_0}+ \ldots + \psi_1{\be_n} +\be_{n+1}$.\medskip
\\
If $m<n$ and for all $i\leq m$, $\al_i=\be_i$.  Then $\be=\al+1$ or $\al<\be[1]$.\medskip
\\
If $m=n$ and for all $i<n, \al_i=\be_i$, and $\al_{n+1}<\be_{n+1}$ then define $\psi_0-$free nesting context $\la\scbr{\cdot} := \psi_1\al_0 + \ldots + \psi_1\al_n+ \scbr{\cdot}$.
Let $\ga := \al_{m+1}$ and $\de := \be_{n+1}$, then $\al = \la\scbr{\ga}$ and $\be = \la^-\scbr{\de}$, 
then there exist a natural number $t\geq 1$ such that $\al < \be[t]$.\medskip
\\
Otherwise there exists an $i\leq \text{min}\{m,n\}$ such that for all $j<i  (\al_j=\be_j)$ and $\al_i<\be_i$. If $n>i$ then $\al<\be[1]$. \\
Otherwise assume $i=n$ and $\be_{n+1}=0$. We write the equal part as $\xi$ for sake of brevity and we have $\al = \xi + \psi_1\al_i+ \ldots +\al_{m+1} $ and $\be = \xi+ \psi_1\be_i$. 
By induction on $\be_i$ , the following cases occur:
\smallskip
\begin{Cases}
    \item ($\be_i=0$).
    Hence, $\psi_1\be_i=\Om$. Since $\al<\be$ , then $\al = \xi + \al_{m+1}$, since otherwise $\al>\be$.  Define the $\psi_0-$nesting free context  $\la\scbr{\cdot} =  \xi+  \scbr{\cdot}$. \\
    Let $\ga := \al_{n+1}$. Then $\al = \la\scbr{\ga}$ and $\be= \la^-\scbr{\Om}$ 
    and there is an ordinal $\tau \geq \ga$ such that
    $\al =\la\scbr{\ga}<\la^-\scbr{\ga+1} =\la^-\scbr{\Om[\ga+1]} \leq \la^-\scbr{\Om}[\tau+1] =\be[\tau+1]$ and $\be[\tau]= \la^-\scbr{\tau} \leq \la\scbr{\ga} = \al$. Hence $\be[\tau] \leq \al < \be[\tau+1]$.\medskip
    \\
    Assume $\be_i > 0$, then it must be that $\al_i<\be_i$, and we have the following cases:
    \smallskip
    
    \item ($\be_i= \ga+1$).
    If $\ga>\al_i$ then $\al<\be[1]$ .\\
    If $\ga=\al_i$, then  $\al = \xi+ \psi_1\al_i \cdot s+ \ldots+\al_{m+1} $ and  $\be =\xi+\psi_1(\al_i+1 )$ where $\al$ has at least one occurrence of $\psi_1\al_i$. 
    We count the number of occurrences of $\psi_1(\al_i)$ by $s\geq 1$.\\
    Define the $\psi_0-$nesting free context $\la\scbr{\cdot} =  \xi+  \scbr{\cdot} + \ldots+\al_{m+1} $. \\
    Let $\ga :=\al_i$. Then $\al = \la\scbr{\psi_1(\ga )\cdot s}$ and $\be = \la^-\scbr{\psi_1(\ga+1)}$.
    And we have $\al = \la\scbr{\psi_1(\ga)\cdot s} <\la^-\scbr{\psi_1(\ga)\cdot (s+1)} = \la^-\scbr{\psi_1(\ga+1)[s+1]} =\be[s+1]$.	
    
    \smallskip
    \item ($\be_i$ not a successor.)
    We have the following cases;
    \begin{Cases}
        \item Since $\al_i<\be_i$, then $\be_i \neq \al_i+1$.\medskip

        \item Assume by induction hypothesis that $\al_i<\be_i[1]$.\\
        Then $\psi_1(\al_i) < \psi_1(\be_i[1]) \in AP$ hence 
        $\al = \xi + \psi_1(\al_i) + \ldots + \al_{m+1} < \xi + \psi_1(\be_i[1]) = \xi+ \psi_1(\be_i)[1]  = \be[1]$.
        Hence $ \al< \be[1]$.
        \medskip
        
        \item Assume by induction hypothesis that  $\be_i = \mu^-\scbr{\Om}$ and $\al_i= \mu\scbr{\ga}$ for some $\psi_0-$nesting free context $\mu$, and $\be_i[\tau] \leq \al_i < \be_i[\tau+1]$ for an ordinal $\tau$.\\	
        Then $\psi_1(\be_i[\tau]) \leq \psi_1(\al) <\psi_1(\be_i[\tau+1])$.
        Define a new $\psi_0-$nesting free context\\
        $\la\scbr{\cdot} := \xi+ \psi_1(\mu\scbr{\cdot}) +...+ {\al_{m+1}}$, then  $\al = \la\scbr{\ga}$ and $\be = \la^-\scbr{\Om}$ and 
        then $\be[\tau]= \xi + \psi_1(\be_i[\tau])\leq \al = \xi + \psi_1(\al_i)+\ldots+ \al_{m+1} < \xi + \psi_1(\be_i[\tau+1])$.
        Hence $\be[\tau] \leq \al < \be[\tau+1]$.
        \medskip
        
        \item Assume by induction hypothesis that $\be_i=\mu^-\scbr{\psi_1(\ga+1)}$ and $\al_i= \mu\scbr{\psi_1\ga \cdot s}$  for some context $\mu$ and $\al_i<\be_i[s+1]$.\\
        Then $\psi_1(\al_i) <\psi_1(\be_i[s+1])$.
        Define a new $\psi_0-$nesting free context\\
        $\la\scbr{\cdot} = \xi + \mu\scbr{\cdot} +\ldots + {\al_{m+1}}$, then $\al = \la\scbr{\psi_1\ga \cdot s}$ and $\be = \la^-\scbr{ \psi_1(\ga+1)}$, 
        and $\al= \xi + \psi_1(\al_i)+\ldots+\al_{m+1}<\xi + \psi_1(\be_i[s+1])=\be[s+1]$.
        Hence $\al < \be[s+1]$.		
        \medskip
        
        \item  Assume by induction hypothesis that $\be_i=\mu^-\scbr{\de}$ and $\al_i= \mu\scbr{\ga}$ for some $\psi_0-$nesting free ordinal context $\mu$ where $\ga<\de<\Om$ and $\be_i[t]\leq \al_i < \be_i[t+1]$ for some natural number $t$.
        Then $\psi_1(\be_i[t]) \leq \psi_1(\al) < \psi_1(\be_i[t+1])$.
        Define the $\psi_0-$nesting free context\\
        $\la\scbr{\cdot} := \xi+ \psi_1(\mu\scbr{\cdot}) + \ldots+ {\al_{m+1}}$, then  $\al = \la\scbr{\ga}$ and $\be = \la^-\scbr{\de}$, 
        and $\be[t] = \xi + \psi_1(\be_i[t])\leq \al = \xi + \psi_1(\al_i) +\ldots+\al_{m+1} < \xi+\psi_1(\be_i[t+1])=\be[t+1]$.
        Hence  $\be[t] \leq \al < \be[t+1]$.
        
    \end{Cases}
\end{Cases}
\end{proof}

In the following Lemma, we analyze the differences between two countable ordinals.\begin{lemma}\label{lm:Context2}
Let $\al<\be<\Om \in \OT_0$ , then one of the following occurs,
\begin{enumerate}
    \item $\al+1=\be$,
    \item $\al<\be[1]$,
    \item there is a context $\la$, a natural number $t$ such that $\al=\la\scbr{t}$ and $\be=\la^-\scbr{\om}$ and $\al<\be[t+1]$,
    \item  there is a context $\la$, an ordinal $\ga<\Om$, a natural number $t$ such that $\al=\la\scbr{\ga}$ and $\be=\la^-\scbr{\psi_0(\mu\scbr{\Om})}$ and $\be[t]\leq\al<\be[t+1]$,
    \item \label{item:con1 lm:Context2} there is some context $\la$, an ordinal $\ga$ and some $s$ such that $\al=\la\scbr{\psi_0(\ga)\cdot s}$ and $\be=\la^-\scbr{\psi_0(\ga+1)}$ and $\al<\be[s+1]$.
    
    \item \label{item:con2 lm:Context2} there is some context $\la$, an ordinal $\ga$ and some $s$ such that $\al=\la\scbr{\psi_1(\ga)\cdot s}$ and $\be=\la^-\scbr{\psi_1(\ga+1)}$ and $\al<\be[s+1]$.
    
\end{enumerate}
\end{lemma}
\begin{proof}
Proof by case distinction on the form of $\be$.\\
Let $\al=\psi_0{\al_0} + \ldots + \psi_0{\al_m} +l$ and $\be = \psi_0{\be_0} + \ldots + \psi_0{\be_n} + d$.\medskip
\\
If $m<n$ and for all $i\leq m$, $\al_i=\be_i$, then $\be=\al+1$ or $\al<\be[1]$.\medskip
\\
If $m=n$ and $l<d$ then $\al+1=\be$ or $\al<\be[1]$. \medskip
\\
Otherwise there exists an $i\leq \text{min}\{m,n\}$ such that for all $j<i (\al_j=\be_j)$ and $\al_i<\be_i$.  If $n>i$ then $\al<\be[1]$.
Otherwise assume $i=n$, then $\al = \xi+\psi_0\al_i + \ldots + l$ and $\be =\xi+\psi_0\be_i$.
By induction on $\be_i$, the following cases occur:
\begin{Cases}
    \item ($\be_i = 0$).
    Hence $\be = \xi + \om$. Since $\al<\be$, then $\al = \xi + l$.
    Define a context  $\la\scbr{\cdot} = \xi+ \scbr{\cdot}$.
    Let $t := l$, then $\al = \la\scbr{t}$ and $\be = \la^-\scbr{\om}$, as a consequence $\al =  \la\scbr{t}< \la^-\scbr{\om[t+1]} = \be[t+1]$.\medskip
    \\
    Assume $\be_i > 0$, then it must be that $\al_i<\be_i$, and we have the following cases:
    \item ($\be_i=\ga+1$).
    If $\ga>\al_i$, then $\al<\be[1]$.\\
    If $\ga=\al_i$, then $\al =\xi+\psi_0\al_i \cdot s+ \ldots + l$ and $\be = \xi+\psi_0(\al_i+1) = \xi+\psi_0(\al_i) \cdot \om $, where $\al$ has at least one occurrence of $\psi_0\al_i$. We count the number of occurrences of $\psi_0\al_i$ by $s\geq 1$.\\
    Define a context $\la\scbr{\cdot} =  \xi +  \scbr{\cdot} + \ldots+ l$.
    Let $\ga:= \al_i$. Then $\al = \la\scbr{\psi_0(\ga)\cdot s }$ and $\be= \la^-\scbr{\psi_0(\ga+1)}$.
    Then $\al = \la\scbr{\psi_0(\ga)\cdot s}< \la^-\scbr{\psi_0(\ga)\cdot (s+1)}  = \la^-\scbr{\psi_0(\ga+1)[s+1]} = \be[s+1].$
    
    \item ($\be_i$ not a successor). Then, applying the induction hypothesis to $\al_i<\be_i$ where $\be_i$ can be of countable or uncountable cofinality, there are several cases to consider.
    First we consider the cases for $\be_i < \Om$. 
    \begin{Cases}
        \item Assume by induction hypothesis that $\al_i<\be_i[1]$, then $\al=\xi + \psi_0(\al_i) +\ldots+l < \xi +\psi_0(\be_i[1]) = \be[1]$ because $\psi_0(\be_i[1]) \in AP$. \medskip
        
        \item Assume by induction hypothesis that $\be_i = \mu^-\scbr{\om}$ and  $\al_i=\mu\scbr{t}$ for some context $\mu$ and $\al_i<\be_i[t+1]$.
        Then $\psi_0(\al_i)<\psi_0(\be_i[t+1])\in AP$.
        Define a new context\\
        $\la\scbr{\cdot} = \xi + \psi_0({\mu\scbr{\cdot}}) +\ldots+ \psi_0{\al_m}+l$. Then $\al= \la\scbr{t}$ and $\be=\la^-\scbr{\om}$.
        Then $\al= \xi +\psi_0(\al_i)+\ldots+l<\xi + \psi_0(\be_i[t+1]) =\be[t+1]$ since $\psi_0(\be_i[t+1]) \in AP$.
        \medskip
        
        \item Assume by induction hypothesis that $\be_i=\mu^-\scbr{\psi_0(\ga+1)}$ and $\al_i = \mu\scbr{\psi_0(\ga)\cdot s}$  for some context $\mu$ and $s$ is the number of occurrences of $\psi_0\ga$ in $\al_i$, and $\al_i<\be_i[s+1]$. 
        Then $\psi_0(\al_i) <\psi_0(\be_i[s+1])$.
        Define a new context\\
        $\la\scbr{\cdot} = \xi+ {\mu\scbr{\cdot}} + \ldots+ \psi_0{\al_m}+l$. Then $\al= \la\scbr{\psi_0(\ga)\cdot s}$ and $\be= \la^-\scbr{\psi_0(\ga+1)}$.
        And we have $\al =\xi + \psi_0(\al_i)+\ldots+l<\xi+\psi_0(\be_i[s+1] =\be[s+1]$ because $\psi_0(\be_i[s+1] \in AP$.
    \end{Cases}
    
    \item $(\be_i$ is not a successor \label{case:beiNotSucc}, $ \al_i<\Om \leq \be_i)$.
    \begin{Cases}
    \item Let $\be_i= \mu\scbr{\Om}$ for some $\psi_0-$nesting free context $\mu$.
    Let $\ga := \psi_0(\al_i)$.	Define the context
    $\la\scbr{\cdot} = \xi+\scbr{\cdot}+ \ldots + l$, then  $\al = \la\scbr{\ga}$ and $\be = \la^-\scbr{\psi_0(\mu\scbr{\Om}}$.
    \\
    Recall the fundamental sequences for $\psi_0(\al_0)$ where $tp(\psi_0(\al_0))=\om$ and $\al_0 :=\mu^-\scbr{\Om}$ for a $\psi_0-$free context $\mu^-$.
    $\psi_0(\al_0)[x] = \psi_0(\al_0[z_x])$ where $z_0=0$ and $z_{x+1}= \psi_0(\al_0[z_x])$.
    \\
    Then there is a natural number $t$ such that $\psi_0(\mu\scbr{\Om})[t] \leq \ga < \psi_0(\mu\scbr{\Om})[t+1]$,
    and we have
    $  \la^-\scbr{\mu\scbr{\Om}}[t] \leq \al = \la\scbr{\ga} < \la\scbr{\ga+1} \leq \la^-\scbr{\mu\scbr{\Om}}[t+1]$. 
    Hence $\be[t] \leq \al < \be[{t+1}]$.\medskip

    \item $(\be_i= \mu^-\scbr{\psi_1(\ga+1)})$, for some  $\psi_0-$nesting free context $\mu$. Since $\al_i$ is a countable ordinal, then $\al_i<\be_i[1]$. Hence $\al<\be[1]$.\medskip
    
    \end{Cases}
    
    \item $(\be_i$ is not a successor, $\Om \leq \al_i<\be_i)$.
    By Lemma \ref{lm:Context1}, the following cases occur;
    \begin{Cases}
        \item ($\be_i =\al_i+1$). Let $\al_i=\psi_1(\al_{i0})$ and $\be_i=\psi_1(\be_{i0})$.
        Then $\al = \xi + \psi_0(\psi_1(\al_{i0}))+\ldots+l$ and $\be=\xi + \psi_0(\psi_1(\al_{i0}+1)) =\xi + \psi_0(\psi_1(\al_{i0})\cdot \om)$.
        Then $\al< \be[2]$.\medskip

        \item ($\al_i<\be_i[1]$). 
        Then $\al = \xi + \psi_0(\al_i)+\ldots +l$ and $\be = \xi + \psi_0(\be_i[1])$.
        And we have $\al = \xi + \psi_0(\al_i)+\ldots +l <\xi + \psi_0(\be_i[1]) =\be[1]$.\medskip

        \item ($\be_i= \mu^-\scbr{\psi_1(\ga+1)}$ and $\al_i=\mu\scbr{\psi_1(\ga)\cdot s}$  for some  $\psi_0-$nesting free context $\mu$ and $s$  the number of occurrences of $\psi_1\ga$ in $\al_i$ and $\al_i<\be_i[s+1]$.)\\
        Define the context 
        $\la\scbr{\cdot} = \xi + \psi_0\mu\scbr{\cdot}+ \ldots+ \psi_0{\al_m} +l$. 
        \\Then $\al=\la\scbr{\psi_1(\ga)\cdot s}$ and $\be= \la^-\scbr{\psi_1(\ga+1)}$.
        We have $\al = \la\scbr{\psi_1(\ga)\cdot s} = \xi + \psi_0(\mu\scbr{\psi_1(\ga)\cdot s})+ \ldots+l
        < \xi + \psi_0(\mu^-\scbr{\psi_1(\ga+1)[s+1]}) =\la^-\scbr{\psi_1(\ga+1)}[s+1] =\be[s+1]$.\medskip

        \item $\be_i= \mu^-\scbr{\ga}$ and $\al_i=\mu\scbr{\de}$ and $\ga<\de<\Om$ for some $\psi_0-$nesting free context $\mu$.
        By induction hypothesis for $\ga,\de<\Om$ there are several cases to consider. We will consider only a few non-obvious ones here.\medskip
        \begin{Cases}
            \item There is a context $\sigma$ such that $\ga= \sigma\scbr{t}$ and $\de=\sigma^-\scbr{\om}$ and $\ga<\de[t+1]$ for a natural number $t$.
            Then $\al_i= \mu\scbr{\sigma\scbr{t}}$ and $\be_i= \mu^-\scbr{\sigma^-\scbr{\om}}$, and by Lemma \ref{lm:Context1}, $\be_i[t] \leq \al_i<\be_i[t+1]$.
            \\
            Define the context 
            $\la\scbr{\cdot} = \xi+ \psi_0(\mu\scbr{\sigma\scbr{\cdot}})+\ldots+l$.\\
            Then  $\al = \la\scbr{t}$ and $\be = \la^-\scbr{\om}$. Note that $\la^-\scbr{\cdot}$ means $\xi+\psi_0(\mu^-\scbr{\sigma^-\scbr{\cdot}})$.
            And we have 
            $\al= \la\scbr{t} = \xi+ \psi_0(\mu\scbr{\sigma\scbr{t}})+\ldots+l < \xi+\psi_0(\mu^-\scbr{\sigma^-\scbr{\om}}[t+1])
            = \la^-\scbr{\om}[t+1] = \be[1]$.\medskip
        \end{Cases}

        \item ($\be_i=\mu^-\scbr{\Om}$ and $\al_i=\mu\scbr{\ga}$ and $\be_i[\tau] \leq \al_i<\be_i[\tau+1]$ for some ordinal $\tau$ and some $\psi_0-$nesting free context $\mu$).
        This changes from the previous \ref{case:beiNotSucc} as now we consider, for instance, the cases such as $\al = \xi + \psi_0(\psi_1(\ga))$ and $\be = \xi  + \psi_0(\psi_1(\Om))$.\\
        Define a new context
        $\la\scbr{\cdot} = \xi+ \psi_0(\mu\scbr{\cdot})+\ldots + l$.\\
        Then $\al = \la\scbr{\ga}$ and $\be = \la^-\scbr{\psi_0(\mu\scbr{\Om})}$.
        And we have 
        $\al = \la\scbr{\ga} = \xi+ \psi_0(\mu\scbr{\ga})+\ldots + l < \xi+ \psi_0(\mu\scbr{\Om})$.
        \\
        Recall the fundamental sequences for $\psi_0(\al_0)$ where $tp(\psi_0(\al_0))=\om$ and $\al_0 :=\mu^-\scbr{\Om}$ for a $\psi_0-$free context $\mu^-$.
        $\psi_0(\al_0)[x] = \psi_0(\al_0[z_x])$ where $z_0=0$ and $z_{x+1}= \psi_0(\al_0[z_x])$.
        \\
        We have that $\al_i<\be_i[\tau+1]$, and if $G_0(\al_i)<\al_i$ then $\psi_0(\al_i) \in \OT_0$ and if $G_0(\be_i[\tau+1])< \be_i[\tau+1]$ then $\psi_0(\be_i[\tau+1]) \in \OT_0$.
        Then there exists an ordinal of the form $z_{t+1}\geq \tau+1$ for a natural number $t$ such that
        $\psi_0(\al_i) < \psi_0(\be_i[\tau+1]) \leq \psi_0(\be_i[z_{t+1}]) = \psi_0(\be_i)[t+1]$.
        \\
        And we have $\xi  +\psi_0(\be_i)[t] \leq \al =\xi + \psi_0(\al_i) +\ldots+\al_{m+1} < \xi  +\psi_0(\be_i)[t+1]$.
        Hence $\be[t] \leq \al < \be[t+1]$.\\
        
    \end{Cases}
\end{Cases}
\end{proof}

{\em In Section \ref{sec:GoodsteinID2}, in order not be very repetitive, we will write $\la\scbr{\psi_i(\ga)\cdot s}$ and $\la^-\scbr{\psi_i(\ga+1)}$ instead of the two cases \ref{item:con1 lm:Context2} and \ref{item:con2 lm:Context2} in Lemma \ref{lm:Context2}.
It will be clear from the context what we mean.}
\\

\section{Goodstein sequences for \texorpdfstring{$\mathrm{ID}_{2}$}{ID2}}
\label{sec:GoodsteinID2}
In this section, we construct a new Goodstein process associated with an extended version of the Hardy hierarchy, building from Section \ref{sec:PropID2}.

In \ref{subsec:HardyGoodstein} introduces the extended Hardy hierarchy $(H_\alpha)$ and defines the $k$-normal forms for which we prove basic growth properties using the fundamental sequences and Bachmann property from \ref{subsec:FundSeqBach}.

Subsection \ref{subsec:BaseChange} defines the base change operation and proves monotonicity and preservation of normal forms for base change (Lemma \ref{lm:MonBaseChange} and \ref{lm:NFpreservBaseChange}), relying on the combinatorial properties of $H_\alpha$.

In \ref{subsec:OrdAssig} the ordinal assignment ($\ok$) is defined and shown to be monotone and to preserve normal forms.
Finally, we show that the ordinal assignment produces strictly decreasing ordinals along the Goodstein process, using the structural properties of the ordinal notation system $\OT$ developed in Section \ref{sec:PropID2}.

\subsection{Hardy hierarchy and Goodstein process}\label{subsec:HardyGoodstein}
The $k$-representation of a number $m$ is constructed with respect to the fast-growing Hardy hierarchy. The ordinal $\al \in \OT_0 \cap \Om$ is a countable ordinal in $\psi_0$ normal form. It is known that $\OT_0 \cap \Om= \psi_0(\psi_1(\varepsilon_{\Om+1}))$ is the Bachmann-Howard ordinal. 
\begin{definition} Hardy hierarchy for $k\geq 0$;
\begin{enumerate}[label=(\roman*)]
    \item $H_0(k) = k$
    \item $H_{\al+1}(k) =H_{\al}(k) \cdot k$
    \item $ H_\lambda(k) =   H_{\lambda{\{k,k\}}} (k) \text{ where } \\
    \lambda{\{0,k\}} = \lambda[0] \text{ and } \lambda{\{b+1,k\}} = \lambda[H_{\lambda\{b,k\}} (k)] $.\medskip
    \\
    From now on, we will not write the second $k$, since it can be read from the function's argument. We will write $H_{\lambda{\{k\}}} (k)$ to say $ H_{\lambda{\{k,k\}}} (k)$ and $H_{\lambda{\{b\}}} (k)$ for $0<b<k$ to say $H_{\lambda{\{b,k\}}} (k)$. \medskip
\end{enumerate}
\end{definition}
The function $H_\al$ has the natural monotonicity properties due to the Bachmann property of the system of fundamental sequences stated in Lemma \ref{thm:bachmannProp}. The function $\al \rightarrow H_\al(k)$ is not monotone, but it behaves in a monotone way with respect to the relation $\leq_{x}$.
Recall that $\leq_{x}$ is the transitive and reflexive closure of $\{ (\al[x],\al ) : \al \in \OT_0\}$.

\begin{definition} Define the maximal coefficient, $mc(\al)$, of an ordinal $\al$ recursively as:
\begin{enumerate}[label=(\roman*)]
\item If $\al=0$, let $mc(0) := 0$.
\item If $\al =_{NF} \al_1+\ldots+\al_n +l$, define $mc(\al) := max \{ mc(\al_1), \ldots , mc(\al_n), l\}$.
\item If $\al=\psi_i(\al_0)$ then $mc(\al)=max(\al_0)$. 
\end{enumerate}
\end{definition}

We repeat some standard majorization properties of the Hardy hierarchy.
\begin{lemma} \label{lm:Hprop} For $\al,\be \in \OT_0\cap \Om$ the following hold;
\begin{enumerate}
    \item \label{case1:lm:Hprop} $H_\al(k)<H_{\al+1}(k)$,
    
    \item  \label{case2:lm:Hprop} If $\al <_{x} \be$ then $H_\al(k) \leq H_{\be}(k)$ for $1\leq x \leq k$,
    
    \item \label{case3:lm:Hprop} $H_\al(k) < H_\al(k+1)$,
    
    \item If $\be[x]<\al<\be$  then  $H_{\be[x]}(k)<H_\al(k)$,
    
    \item 	$N(\al) \leq H_\al(k)$.
\end{enumerate}
\end{lemma}
\begin{proof}
Adaptions from \cite{BuchholzCW1994} or \cite{Weiermann2006}. We give the proof of the first three assertions.
\begin{enumerate}[leftmargin=*]
\item Since $H_\al$ is monotone, then 
 $H_\al(k)<H_\al(k)\cdot 2< \ldots<H_\al(k)\cdot k = H_{\al+1}(k)$.

\item By induction on $\be$.
\begin{Cases}
    \item $(\be=\ga+1)$. Then $\al\leq \ga$ and by the induction hypothesis and assertion \ref{case1:lm:Hprop}, 
    $H_\al(k)\leq H_\ga(k) < H_{\ga+1} \leq H_\be(k)$.

    \item $(tp(\be=\om))$, then $\al\leq_x \be[x]$ by Lemma \ref{lm:propertyTC} item \ref{item:2 lm:propertyTC}, and \\
    $\be[k]\leq_x \be\{k\}$ by Lemma \ref{lm:propertyTC} item \ref{item:1 lm:propertyTC}.
    By the induction hypothesis  
    $H_\al(k)\leq H_{\be[x]}(k) \leq H_{\be[k]}(k) \leq H_{\be\{k\}}(k)=H_\be(k)$.
\end{Cases}

\item  By induction on $\al$.
\begin{Cases}
    \item $(\al=0)$. Then it's obvious.
    \item $(\al=\ga+1)$. Then the assertion follows immediately by the induction hypothesis.
    \item $(tp(\al=\om))$.
    By the induction hypothesis,
    $H_\al(k) = H_{\al\{k\}}(k)<H_{\al\{k\}}(k+1)$. Then 
    by Lemma \ref{lm:propertyTC} item \ref{item:1 lm:propertyTC} we have that $\al\{k\} \leq_x \al\{k+1\}$,
    and assertion \ref{case2:lm:Hprop}, we have that,
    $H_{\al\{k\}}(k+1) \leq H_{\al\{k+1\}}(k+1) = H_\al(k+1)$.
\end{Cases}
\end{enumerate}

\end{proof}


\begin{lemma}
For any natural number $m$, there will be at most, finitely many $\al \in \OT_0\cap \Om$ such that $H_\al(k)\leq m$.
\end{lemma}
\begin{proof}
For a given $m$, the set $M := \{ \al | N(\al)\leq m \}$ is finite, where recall $N(\al)$ is the term length in the normal form of $\al$. 
We have that $N(\al)\leq H_\al(k)$, and the Lemma follows.

\end{proof}

We use the Hardy function to define $k$ normal forms for natural numbers.\medskip

\begin{lemma}\label{lm:HardyNF}
Fix $k\geq 2$.  For all $m\geq k$, there exist unique $\al\in  \OT_0\cap \Om$ and $q<\om$ such that
\begin{enumerate}
    \item $m=H_\al(k) +q$,
    \item $\al$ is maximal with $H_\al(k) \leq m$. 
    \item $q$ is the natural number with $q < H_{\al+1}(k)$ such that $H_\al(k) +q \leq m < H_\al(k) +q+1$.
\end{enumerate}
\end{lemma}
We write $m=_{k} H_\al(k)+q$ in this case and say that $m$ is $k$ normal form. This means that we have in mind a fixed $k$ and a number $m$ for which we have uniquely determined $\al$ and $q$. Also, $q$ will be written in $k$ normal form recursively.
If it is the case that $q>H_\al(k)$, then it will be convenient to count the number of occurrences of $H_\al(k)$ by $p$ and have $m=_{k} H_\al(k) \cdot p + q$, where $p<k$.
If $m\leq k$ then $m$ is its own $k$ normal form. 
If $m=_{k}H_\al +q$ and $\al$ is of the form $\al = \psi_0\al_0+\ldots +\psi_0\al_n+ l$ where $l>k$, then $l$ is written in $k$ normal form and $\al_0\geq\ldots \geq \al_n$ is written in $k$ normal form.\medskip
\\
The following properties follow from Bachmann property.
\begin{lemma}\label{lm:NForms} Fix $k\geq 2$.
\begin{enumerate}
    \item If $m=_{k} H_\al(k) \cdot p$ then $H_\al(k) \cdot (p-1)$ is in $k$ normal form as well.
    \item If  $m=_{k} H_\al(k) +q+1$ then $H_\al(k) + q$ is in $k$ normal form.
    \item If $m=_{k}H_\al(k)$ and $\al$ is a limit, then $H_{\al\{b\}}(k)$ for $0<b<k$ is in $k$ normal form. 
    \item If $m=_{k}H_{\al}(k)$ and $\al =\al_0+l$ is successor for $l$ a natural number,
    then $H_{\al_0+l-1}(k)$ is in $k$ normal form. 
\end{enumerate}
\end{lemma}
\begin{proof}
By induction on $\al$, using the definition of fundamental sequences from \ref{subsec:FundSeqBach} and the Bachmann property (Theorem \ref{thm:bachmannProp}).

Proof of the third assertion.\\
Fix $\de>\al\{b\}$. 
If $\al\{b\} <\de<\al$, the Bachmann property yields $H_{\al\{b\}}(k) <H_\de(k)$.
\
If $\de>\al$ then $H_\de(k)\leq H_\al(k)$ which would contradict the fact that $H_\al(k)$ is in $k$ normal form. 
\\
If $\de=\al$, by the Bachmann property $H_\de(k) = H_{\de\{k\}} (k)> H_{\al\{b\}}(k)$ since $b<k$.
In all cases we find $H_{\de}(k)>H_{\al\{b\}}(k)$.\medskip\\
\end{proof}

\begin{lemma} \label{lm:MeasureH}
 Assume by Lemma \ref{lm:lala-} that $\la^-\scbr{\ga} \leq_1 \la\scbr{\ga}$. Then $H_{\la^-\scbr{\ga}}(k) \leq H_{\la\scbr{\ga}}(k)$.
\end{lemma}
\begin{proof}
Follows from assertion \ref{case2:lm:Hprop} in Lemma \ref{lm:Hprop}.    
\end{proof}
\smallskip

In the following Lemma, we prove some bounds on ordinals of the $k$ normal form of $m$. These bounds will prove crucial later on.
\begin{lemma}\label{lm:OrdMeasure}
Assume that  $\al$ is maximal with $H_\al(k)\leq m$. Then for all contexts $\la$,
\begin{enumerate}
    \item If $\al = \la\scbr{t}$ and $\al<\de= \la^-\scbr{\om}$ and $\al< \de[t+1]$  then $t <H_{\de\{k-1\}}(k)$.
    
    \item If $\al =\la\scbr{\ga}$ and $\al< \de=\la^-\scbr{\psi_0(\mu\scbr{\Om})}$ and if there is some $t$ such that
    $\de[t]\leq\al<\de[t+1]$ then $t<H_{\de\{k-1\}}(k)$.
    
    \item If $\al = \la\scbr{\psi_i(\ga)\cdot s}$ and $\al < \de = \la^-\scbr{\psi_i(\ga+1)}$  for $i \in \{0,1\}$ and $\al < \de[s+1]$ then $s< H_{\de\{k-1\}}(k)$. 
\end{enumerate}
\end{lemma}
\begin{proof}
First assertion. 
Assume towards a contradiction $t\geq H_{\de\{k-1\}}(k)$ then
\begin{align*}
    H_\de(k) &=H_{\de\{k\}}(k) \\
    &=H_{\de[H_{\de\{k-1\}}(k)]} (k) \\
    &= H_{(\la^-\scbr{\om})[H_{\de\{k-1\}}(k)]}(k)\\
    &= H_{\la^-\scbr{\om[H_{\de\{k-1\}}(k)]}}(k)\\
    &= H_{\la^-\scbr{H_{\de\{k-1\}}(k)}} (k)\\
    & \leq H_{\la^-\scbr{t}}(k)\\
    & \leq H_{\la\scbr{t}}(k) =  H_\al(k).
\end{align*}
since  $\de[l] = (\la^-\scbr{\om})[l] =  \la^-\scbr{l} $ for every $l$ by Lemma \ref{lm:fundamOrd}, \\
and $H_{\la^-\scbr{H_{\de\{k-1\}}(k)}} (k) \leq H_{\la^-\scbr{t}}(k) $ by Bachmann property and the last step is by Lemma \ref{lm:MeasureH}.
Contradiction to the maximality of $\al$.\medskip
\\\\
Second assertion. \\
Recall the fundamental sequences for $\psi_0(\al_0)$ where $tp(\psi_0(\al_0))=\om$ and $\al_0 :=\mu^-\scbr{\Om}$ for a $\psi_0-$free context $\mu^-$.
$\psi_0(\al_0)[x] = \psi_0(\al_0[z_x])$ where $z_0=0$ and $z_{x+1}= \psi_0(\al_0[z_x])$.\\
Assume for a contradiction  $t\geq H_{\de\{k-1\}}(k)$ and $\de[t]\leq \al<\de[t+1]$.
Then  
\begin{align*}
    H_\de(k) & = H_{\de\{k\}}(k) \\
    & = H_{\de[H_{\de\{k-1\}}(k)]}(k)\\
    &= H_{(\la^-\scbr{\psi_0(\mu^-\scbr{\Om})})[H_{\de\{k-1\}}(k)]} (k) \\
    &= H_{\la^-\scbr{\psi_0((\mu^-\scbr{\Om})[z_{H_{\de\{k-1\}}(k)}] )}}(k) \\ 
    &\leq H_{\la^-\scbr{\psi_0(\mu^-\scbr{\Om}[z_t])} }(k) \\
    &\leq H_{\la^-\scbr{\psi_0(\mu^-\scbr{\Om})}[t] }(k) \\
    &\leq H_\al(k).
\end{align*}
bsince $\de[l] = \la^-\scbr{\psi_0(\mu^-\scbr{\Om})}[l] = \la^-\scbr{\psi_0(\mu^-\scbr{\Om}[z_l])}$ and the first inequality is by Bachmann property and the last step is by Lemma \ref{lm:MeasureH}.
This is in contradiction with the maximality of $\al$.\medskip\\
\\
Third assertion.
Assume for a contradiction $s\geq H_{\be\{k-1\}}(k)$ then
\begin{align*}
    H_\de(k) &=H_{\de\{k\}}(k) \\
    &=H_{\de[H_{\de\{k-1\}}(k)]} (k) \\
    &= H_{(\la^-\scbr{\psi_i(\ga+1)}) [H_{\de\{k-1\}}(k)] }(k)\\
    &= H_{\la^-\scbr{\psi_i(\ga) \cdot  H_{\de\{k-1\}}(k) } }(k)\\	
    &\leq H_{\la^-\scbr{\psi_i(\ga) \cdot  s } }(k)\\
    &= H_{\la\scbr{\psi_i(\ga) \cdot s }}(k) = H_{\la}(k).
\end{align*}
since $\de[l] = \la^-\scbr{\psi_i(\ga+1)}[l] =  \la^-\scbr{\psi_i(\ga) \cdot \om}[l] =  \la^-\scbr{\psi_i(\ga) \cdot l}$ for every $l$  and
and the inequality is by Bachmann property and by Lemma \ref{lm:MeasureH}.
This is in contradiction with the maximality of $\al$.\medskip\\
\end{proof}

\begin{remark}
Note that the measure of the bound for  $\al = \la\scbr{\psi_i(\ga)\cdot s}$ indicates that $s$ is not a maximal coefficient bound, rather it represents the limit on the number of occurrences of $\psi_i(\ga)$ in $\al$.  In contrast, for $\al=\la\scbr{t}$, the bound on $t$ is a proper bound on the coefficient of the ordinal in $\al$. Likewise, for $\al=\la\scbr{\ga}$, the bound on $t$ is a limit on the number of iterations of the fundamental sequence for $\la^-\scbr{\mu\scbr{\Om}}$.\medskip
\end{remark}

To prove the preservation of normal forms after performing a base change operation the following Lemma will be of key importance.\smallskip

\begin{lemma}\label{lm:OrdMeasure2}
Let $\al<\Om$. Then for all contexts $\la$, ordinal $\ga$ and natural numbers $s,t$, one of the following occurs;
\begin{enumerate}
    \item \label{case1:lm:OrdMeasure2}If $\al=\la\scbr{t}$ then $t<H_{(\la^-\scbr{\om})\{k-1\}}(k)$,
    \item  \label{case2:lm:OrdMeasure2}If $\al=\la\scbr{\ga}$ then $\la^-\scbr{\psi_0(\mu\scbr{\Om})}[t]\leq\al<\la^-\scbr{\psi_0(\mu\scbr{\Om})}[t+1]$ where $\mu$ is a $\psi_0-$nesting free context, then $t<H_{\la^-\scbr{\psi_0(\mu\scbr{\Om})}\{k-1\}}(k)$,
    \item  \label{case3:lm:OrdMeasure2}If $\al=\la\scbr{\psi_0(\ga)\cdot s}$ then $s<H_{\la^-\scbr{\psi_0(\ga+1)}\{k-1\}}(k)$.
\end{enumerate}
Then if $\al<\be<\Om$, either  $\al+1=\be$ or $\al+1\leq \be[H_{\be\{k-1\}}(k)]=\be\{k\}$. \\
Moreover $H_{\al+1}(k)\leq H_\be(k)$.\medskip
\end{lemma}
\begin{proof}
First assertion. By Lemma \ref{lm:Context2}, there are five cases to consider;
\begin{Cases}
    \item Clear if $\be= \al+1$.
    
    \item If $\al<\be[1]$ then $\al+1 \leq \be[H_{\be\{k-1\}}(k)]$.
    
    \item 	Assume by Lemma \ref{lm:Context2}, that $\al=\la\scbr{t}$ and  $\be=\la^-\scbr{\om}$.\\
    Then  $\be[l] = \la^-\scbr{\om}[l] =  \la^-\scbr{l} $ for every $l$, and $\be\{k\}= \be[H_{\be\{k-1\}}(k)]$.\\
    By assumption,  $t< H_{\la^-\scbr{\om}\{k-1\}}(k)$. Then
    \begin{align*}
        \al = \la\scbr{t} &< \la^-\scbr{t+1} \\
        &\leq \la^-\scbr{H_{\la^-\scbr{\om}\{k-1\}}(k)} \\
        &= \be[H_{\la^-\scbr{\om}\{k-1\}}(k)]\\
        &= \be[H_{\be\{k-1\}}(k)].
    \end{align*}
    
    \item 
    Assume by Lemma \ref{lm:Context2} now that  $\al=\la\scbr{\ga}$ and $\be=\la^-\scbr{\psi_0(\mu\scbr{\Om})}$.\medskip\\
    Recall the fundamental sequences for $\al = \psi_0(\al_0)$ where $tp(\al)=\om$ and $\al_0=\mu\scbr{\Om}$ where $\mu$ is a $\psi_0-$nesting free context, we have 
    $\psi_0(\al_0)[x] = \psi_0(\al_0[z_x])$ where $z_0=0$ and $z_{x+1}= \psi_0(\al_0[z_x])$.\medskip\\
    Then $\be[l] = \la^-\scbr{\psi_0(\mu\scbr{\Om})}[l] = \la^-\scbr{\psi_0(\mu\scbr{\Om}[z_l])}$ and $\be\{k\}= \be[H_{\be\{k-1\}}(k)]$.\\
    Since $\al> \be[1]$ and $\al<\be$ then there is some $t$ such that $\be[t]\leq \al<\be[t+1]$, that is $\la^-\scbr{\psi_0(\mu\scbr{\Om})}[t]\leq \al < \la^-\scbr{\psi_0(\mu\scbr{\Om})}[t+1]$. Then, by assumption, $t<H_{\la^-\scbr{\psi_0(\mu\scbr{\Om})}\{k-1\}}(k)$.
    Then
    \begin{align*}
        \al = \la\scbr{\ga}
        &< \la^-\scbr{\psi_0(\mu\scbr{\Om}[z_{t+1}])}\\
        &\leq \la^-\scbr{\psi_0(\mu\scbr{\Om}[z_{H_{\la^-\scbr{\psi_0(\mu\scbr{\Om})}\{k-1\}}(k)}])}\\
        &= \la^-\scbr{\psi_0(\mu\scbr{\Om})}[H_{\la^-\scbr{\psi_0(\mu\scbr{\Om})}\{k-1\}}(k)]\\
        &=\be[H_{\la^-\scbr{\psi_0(\mu\scbr{\Om})}\{k-1\}}(k)]\\
        &= \be[H_{\be\{k-1\}}(k)].
    \end{align*}
    
    \item 
    Assume now that $\al=\la\scbr{\psi_i(\ga)\cdot s}$ and $\be=\la^-\scbr{\psi_i(\ga+1)}$.\\
    Then $\be[l]= \la^-\scbr{\psi_i\ga \cdot l}$ for every $l$ and $\be\{k\}= \be[H_{\be\{k-1\}}(k)]$.\\
    By assumption, $s<H_{\la^-\scbr{\psi_i(\ga+1)}\{k-1\}}(k)$. Then
    \begin{align*}
        \al =\la\scbr{\psi_i(\ga)\cdot s} &< \la^-\scbr{\psi_i(\ga) \cdot (s+1)} \\
        &\leq \la^-\scbr{\psi_i(\ga) \cdot H_{\la^-\scbr{\psi_i(\ga+1)}\{k-1\}}(k)}\\
        &= \be[H_{\be\{k-1\}}(k)].
    \end{align*}
\end{Cases}
The second assertion follows from the first assertion by induction on $\be$.  The claim is obvious when $\be=\al+1$. Otherwise, we have $\al<\be[H_{\be\{k-1\}}(k)]$ and the induction hypothesis gives $H_{\al+1}(k)\leq H_{ \be[H_{\be\{k-1\}}(k)]} = H_\be(k)$.\medskip\\

\end{proof}

\begin{lemma}\label{lm:lex}
Let $m=_{k}H_\al(k)+ q$  and  $n=_{k}H_\be(k) +v$ be in $k$ normal form.
Then $m<n$ if and only if the tuple $\langle \al , q \rangle $ is lexicographically smaller($<_{lex}$) than the tuple 
$\langle \be, v \rangle$.
\end{lemma}
\begin{proof} 
First we check that if $\langle\alpha,q\rangle <_{\lex} \langle\beta,v\rangle$, then $m<n$.
\\
If $\alpha<\beta$, then $m = H_\alpha(k)+q < H_{\alpha+1}(k)\le H_\beta(k)\le H_\beta(k)+v = n,$ so $m<n.$
\\
If $\alpha=\beta$ and $q<v$, then $m = H_\alpha(k)+q < H_\alpha(k)+v = n.$
\\
So $\langle\alpha,q\rangle <_{\lex} \langle\beta,v\rangle \implies m<n.$
\\
In the other direction, if $m<n$, then $\langle\alpha,q\rangle <_{\lex} \langle\beta,v\rangle$
Suppose $m<n$ but $\langle\alpha,q\rangle \not<_{\lex} \langle\beta,v\rangle.$
\\
If $\alpha>\beta$, then monotonicity in $\alpha$ gives $H_{\beta+1}(k) \le H_\alpha(k)$, so 
$n < H_{\beta+1}(k) \le H_\alpha(k) \le m,$
contradicting $m<n.$
Thus $\alpha = \beta$. If also $q\ge v$, then $m = H_\alpha(k)+q \ge H_\alpha(k)+v = n$, again contradicting $m<n.$\medskip

Hence the only remaining possibility is $\alpha<\beta$ or $\alpha=\beta, q<v,$ i.e. $(\langle\alpha,q\rangle <_{\mathrm{lex}} \langle\beta,v\rangle).$
So $m<n$ iff $\langle\alpha,q\rangle <_{\mathrm{lex}} \langle\beta,v\rangle.$

\end{proof}

\subsection{Base change operation}\label{subsec:BaseChange}
The following definition of the base change operation takes the base value $k$ to $k+1$.
We then define a Goodstein process based on this base change operator.\medskip

\begin{definition}Base change operation for natural numbers.
\begin{itemize}
    \item If $m<k$ then $m\kp=m$. 
    \item If $m=k$ then $m\kp=k+1$.
    \item If $m=_{k} H_\al(k)\cdot p + q$ then $m\kp:=H_{\al\kp} (k+1)\cdot p + q\kp$.\\
\end{itemize}
\end{definition}

Next, the base change operation $\kp$ is naturally extended to ordinals $\al \in \OT_0$ as follows: the operation maps ordinals $\al \in \OT_0$ to ordinals $\al \in T\cap \Om_1$. Later it will be shown that $\al\kp \in \OT_0$ by Lemma \ref{lm:MonBaseChange} on monotonicity of base change operation.

\begin{definition}Base change operation for ordinal $\al \in \OT_0$;
\begin{itemize}
    \item $0\kp =0$.
    \item If $\al =_{NF} \psi_i(\al_0)$  for $i \in \{0,1\}$ then $\al\kp = \psi_i(\al_0\kp)$.
    \item If $\al =_{NF} \al_0 +...+\al_n $ then $\al\kp=\al_0\kp +...+ \al_n\kp$.\\
\end{itemize}
\end{definition}
Next, define the base change operation for an ordinal context $\la\scbr{\cdot}$  to a context $\la\scbr{\cdot}\kp$. Again, if $\la$ is a context, then $\la'$ is not necessarily in $\OT_0$.

\begin{definition}Base change operation for context.
\begin{itemize}
    \item If $\la\scbr{\cdot}=\xi+ \scbr{\cdot} + \eta$ then 
    $\la\scbr{\cdot}\kp=\xi\kp+ \scbr{\cdot}+ \eta\kp$. 
    
    \item If $\la\scbr{\cdot}=\xi+ \psi_i(\scbr{\cdot}) + \eta$ then 
    $\la\scbr{\cdot}\kp=\xi\kp+ \psi_i(\scbr{\cdot})\kp+ \eta\kp$ for $i \in \{0,1\}$.
    
    \item If $\la\scbr{\cdot}=\xi+ \psi_i(\tau{\scbr{\cdot}})+\eta$ then \\
    $\la\scbr{\cdot}\kp=\xi\kp+ \psi_i(\tau{\scbr{\cdot}\kp})+\eta\kp$ for $i \in \{0,1\}$.
    
\end{itemize}
\end{definition}
\begin{remark}{\em We write for shorthand $m'$ or $\al'$ or $(\la\scbr{\cdot})'$ for  $m\kp$, $\al\kp$, $\la\scbr{\cdot}\kp$ respectively.} \\
\end{remark}

\begin{example}
Note that the base change operation does not affect the number of term occurrences. Let $k=3$ and $\al = \psi_0(3) \cdot 3$, then $\al' = \psi_0(4) \cdot 3$. 
Neither the occurrences of an $H_\al(k)$ term; if $m=H_{\al}(k) \cdot 3$ then $m'=H_{\al'}(k+1)\cdot 3$.
Another definition of the base change operation would be to consider also of the number of term occurrences.\medskip
\end{example}
We then define a new Goodstein process based on this base change operator.

\begin{definition}
Let $\ell<\om$ and $\ell$ is written in $k$ normal form for $k\geq2$.
Put $G_0(\ell):=\ell.$
Assume recursively that $G_k(\ell)$ is defined and $G_k(\ell)>0$.
Then, $G_{k+1}(\ell)=G_k(\ell)[k+2 \leftarrow k+3] -1 $. If $G_k(\ell)=0$, then $G_{k+1}(\ell):=0$.
\end{definition}
\smallskip

We will show that for every $\ell$, there is $k$ such that $G_k(\ell) = 0$.
To prove this, we need to establish two key properties: first, that the base change operation is monotone, and second, that it preserves normal forms. 
To assist us in proving these properties, the following Lemmas will be helpful.
\medskip

\begin{lemma} \label{lm:basechangeDiff}Let $m' :=m\kp$. Then 
$m \leq m'$ and $\al \leq \al'$.
\end{lemma}
\begin{proof} Proof by simultaneous induction on $m$ and $\al$.\\
Proof of $m \leq m'$ by induction on $m$.
The assertion is clear if $m=0$.\\
If $m<k$, then $m=m'$. \\
If $m=k$, then $m < k+1 =m'$.\\
If $m=_k H_\al(k) +q$, then the induction hypothesis yields $q\leq q'$.\\
Assume for simplicity that $m=_k H_\al(k)$. By the second assertion, $\al\leq \al'$.
\medskip\\
If $\al=\al'$,
then $H_{\al}(k)< H_{\al'}(k) < H_{\al'}(k+1)$ by assertion \ref{case3:lm:Hprop} in Lemma \ref{lm:Hprop}.
If $\al <\al'$, then since $\al$ is maximal for $m$, we have that 
$m<H_{\al'}(k) < H_{\al'}(k+1) =m'.$
\\
\\
Proof of the second assertion by induction on $\al$.\\
If $\al=0$ it is clear.
\\
If $\al=\al_0+ \ldots +\al_n$, then by induction hypothesis $\al_i\leq\al_i'$ for $0 \leq i \leq n$, then $\al=\al_0 + \ldots +\al_n \leq \al_0'+ \ldots +\al_n' =\al'$.
\\
If $\al=\psi_i(\al_0)$  for $i \in \{0,1\}$ then $\al'=\psi_i(\al_0')$. By induction hypothesis $\al_0\leq \al_0'$ then $\psi_i(\al_0)\leq \psi_i(\al_0')=\al'$.\\
\end{proof}


Before proving the following claims, recall the definition of $m=H_\al(k) = H_{\al\{k,k\}}(k)$ which we shortly wrote $H_{\al\{k\}}(k)$.
Then $H_\xi(k))' = H_{\xi'}(k+1) = H_{\xi'\{k+1,k+1\}}(k+1)$, which we shortly write $H_{\xi'\{k+1\}}(k+1)$.
\\
\begin{claim} \label{claim1}
Let $m =_k H_\xi(k)$ and ${\xi}=\la^-\scbr{\om}$ and $\xi' = (\la^-)'\scbr{\om}$.
Then for $b<k+1$, $ (\xi\{b\})'={\xi}'\{b\}$.
\end{claim}		
\begin{proof}
By induction on $b$. \\
If $b=0$, then ${\xi}'\{0\} = (\la^-)'\scbr{\om}\{0\}= (\la^-)'\scbr{0}= (\xi\{0\})'$.\\
Inductively for $b+1$, \\
${\xi'\{b+1\}} =\xi'[H_{\xi'\{b\}}(k+1)] =  (\la^-)'\scbr{\om[H_{\xi'\{b\}}(k+1)]} =
(\la^-)'\scbr{ H_{\xi'\{b\}}(k+1)}  =  (\la^-\scbr{ H_{\xi\{b\}}(k})' =(\xi\{b+1\})'$.
\end{proof} 

\smallskip

\begin{claim}\label{claim2}
Let $m =_k H_\xi(k)$ and $\xi = \la^-\scbr{\psi_0(\mu\scbr{\Om})}$ for $\mu$ a $\psi_0-$nesting free context such that $\xi' = (\la^-)'\scbr{\psi_0(\mu'\scbr{\Om})}$. Then for $b<k+1$ we have that ${\xi}'\{b\} = (\xi\{b\})'$.
\end{claim}
\begin{proof}			
By induction on $b$.\\
Recall the fundamental sequences for $\al = \psi_0(\al_0)$ where $tp(\al)=\om$ and $\al_0 :=\mu^-\scbr{\Om}$ for a $\psi_0-$nesting free context $\mu^-$.
$\psi_0(\al_0)[x] = \psi_0(\al_0[z_x])$ where $z_0=0$ and $z_{x+1}= \psi_0(\al_0[z_x])$.\\
If $b=0$, then ${\xi}'\{0\} = (\la^-)'\scbr{\psi_0(\mu'\scbr{\Om})}[0]=  (\la^-)'\scbr{\psi_0(\mu'\scbr{0})}= (\xi\{0\})'$.\\
Inductively for $b+1$, 
\begin{align*}
    {\xi'\{b+1\}} &=(\la^-)'\scbr{\psi_0(\mu'\scbr{\Om})}[H_{\xi'\{b\}}(k+1)] \\ 
    &=( \la^-)'\scbr{\psi_0(\mu'\scbr{\Om})[H_{(\xi\{b\})'}(k+1)] }\\
    &=( \la^-\scbr{\psi_0(\mu\scbr{\Om} )[{H_{\xi\{b\}}(k) }]})' \\
    &=(\xi\{b+1\})'.
\end{align*}
\end{proof}		

Since the base change operation moves ordinals from $\OT_0$ to the system $T\cap \Om$, retrieving an ordinal $\al$ from $\al'$ means that $\al$ may not necessarily in $\OT_0$. Hence in the following Lemma, the context $\la$ is not necessarily in normal form.
\smallskip

\begin{lemma} \label{lm:BaseChangeContext} Let $\al':=\al\kp$.
If $\al'=\tilde{\la} \scbr{\tilde{\ga}}$ then $\al=\la\scbr{\ga}$ for $\la \in T\cap \Om$, where $\la'=\tilde{\la} , \ga'= \tilde{\ga}, $ and $\al'=\la'\scbr{\ga'}$.
\end{lemma}
\begin{proof}
By induction on $N(\al)$.
\begin{Cases}
    \item Assume $\al' =\tilde{\la}\scbr{\tilde{\al_0}} =\psi_i(\tilde{\la_1}\scbr{\tilde{\ga}})$ where $\tilde{\la}\scbr{\cdot} = \psi_i(\tilde{\la_1}\scbr{\cdot})$ and $\tilde{\al_0}=\tilde{\la_1}\scbr{\tilde{\ga}}$ for $i \in \{0,1\}$.
    \\
    Since $N(\al_0)<N(\al)$, then by the induction hypothesis $\al_0=\la_1\scbr{\ga}$ where $\la_1'=\tilde{\la_1}$ and $\ga'=\tilde{\ga}$ and $\al_0'=\la_1'\scbr{\ga'}$.\\
    Define the context $\la\scbr{\cdot} = \psi_i(\la_1\scbr{\cdot})$,
    then $\al = \la\scbr{\ga} = \psi_i(\la_1\scbr{\ga}) = \psi_i(\al_0)$
    and $\la'\scbr{\cdot} = \psi_i(\la'\scbr{\cdot})= \psi_i(\tilde{\la_1}\scbr{\cdot}) = \tilde{\la_1}\scbr{\cdot}$ 
    and $\ga'=\tilde{\ga}$ by induction hypothesis,
    and $\al'= \psi_i(\la_1'\scbr{\ga'})$.
    
    \item 
    Assume  $\al' =\tilde{\la}\scbr{\tilde{\al_i}} =  \al_0' + \ldots + \tilde{\la_1}\scbr{\tilde{\ga}} + \ldots +\al_n' +l'$, 
    where $\tilde{\la}\scbr{\cdot} = \al_0' + \ldots + \tilde{\la_1}\scbr{\cdot} +\ldots +\al_n'+l'$
    and $\tilde{\al_i} = \tilde{\la_1}\scbr{\tilde{\ga}}$.\\
    Then by induction hypothesis, $\al_i' = \la_1'\scbr{\ga'}$ where $\tilde{\la_1}= \la_1'$ and $\tilde{\ga}=\ga'$ and $\al_i=\la_1\scbr{\ga}$.
    Define the context $\la\scbr{\cdot} =  \al_0 + \la_1\scbr{\cdot} + \al_n$, 	\\
    then $\al = \la\scbr{\ga} = \al_0 + \ldots+ \la_1\scbr{\ga}+\ldots+\al_n+l$  where $\la'=\tilde{\la}$ and $\ga' = \tilde{\ga}$
    and $\al' =  \al_0' + \ldots+ \la_1'\scbr{\ga'}+ \ldots+\al_n' +l'$. 
    
\end{Cases}
\end{proof}

\begin{lemma} \label{lm:BaseChangeContext2}
Define $\al' := \al\kp$.
Let $\al'=\la'\scbr{\ga'}$ and $\be' =(\la^-)'\scbr{\psi_0(\tilde{\mu}\scbr{\Om})}$ for a $\psi_0$-nesting free context $\tilde{\mu}$ and there is some $t$ such that $\be'[t]\leq \al'<\be'[t+1]$.
Then there is some $\psi_0-$nesting free context $\mu$ such that $\mu'=\tilde{\mu}$ and  $\be=\la^-\scbr{\psi_0(\mu\scbr{\Om})}$ and $\be'=(\la^-)'\scbr{\psi_0(\mu'\scbr{\Om})}$ and $(\be[t])' =\be'[t]$.
\end{lemma}
\begin{proof}
Assume that $\al'=\la'\scbr{\ga'}$ and $\be' =(\la^-)'\scbr{\psi_0(\tilde{\mu}\scbr{\Om})}$ and there is some $t$ such that $\be'[t] \leq \al' <\be'[t+1]$. Then replacing for $\al'$ and $\be'$ we have

\[(\la^-)'\scbr{\psi_0(\tilde{\mu}\scbr{\Om})}[t] \leq \la'\scbr{\ga'} < (\la^-)'\scbr{\psi_0(\tilde{\mu}\scbr{\Om})}[t+1].\]
By Lemma \ref{lm:fundamOrd},
\[(\la^-)'\scbr{\psi_0(\tilde{\mu}\scbr{\Om})[t]} \leq \la'\scbr{\ga'} < (\la^-)'\scbr{\psi_0(\tilde{\mu}\scbr{\Om})[t+1]}.\]
By Lemma \ref{lm:SandwichContext},
\[(\la^-)'\scbr{\psi_0(\tilde{\mu}\scbr{\Om})[t]} \leq (\la^-)'\scbr{\ga'} < (\la^-)'\scbr{\psi_0(\tilde{\mu}\scbr{\Om})[t+1]},\]
\[ \psi_0(\tilde{\mu}\scbr{\Om})[t]  \leq \ga' < \psi_0(\tilde{\mu}\scbr{\Om})[t+1].\] \medskip
We will show the last inequality in the following Claim:
\begin{claim}
    If	$\psi_0(\tilde{\mu}\scbr{\Om})[t] \leq \ga' < \psi_0(\tilde{\mu}\scbr{\Om})[t+1]$, then $\tilde{\mu} =\mu'$ for some $\psi_0-$nesting free context $\mu$.
\end{claim}
Proof by induction on $\ga'$ with subsidiary induction on $\tilde{\mu}$.
\begin{Cases}
    \item   If $\ga'$ is a sum, $\ga' = \ga_0' +\ldots +\ga_n'$, then the assertion follows by induction hypothesis. 
    
    \item Let $\ga' = \psi_0(\ga_0')$. Since 
    $$\psi_0(\tilde{\mu}\scbr{\Om})[t] \leq \psi_0(\ga_0') < \psi_0(\tilde{\mu}\scbr{\Om})[t+1],$$
    then 
    $$\tilde{\mu}\scbr{\psi_0(\tilde{\mu}\scbr{\Om})[t-1]} \leq \ga_0' < \tilde{\mu}\scbr{\psi_0(\tilde{\mu}\scbr{\Om})[t]}.$$
    \begin{Cases}
      \item  If $\tilde{\mu}$ is a sum, $\tilde{\mu} = \tilde{\mu}_0+\ldots +\tilde{\mu}_n$. Then $\ga_0'= \tilde{\mu}_0+\ldots + \de_0'$ and 
         $$\tilde{\mu}_0+\ldots + \tilde{\mu}_n\scbr{\psi_0(\tilde{\mu}\scbr{\Om})[t-1]} \leq \tilde{\mu}_0+ \ldots + \de_0' < \ \tilde{\mu}_0+\ldots+ \tilde{\mu}_n\scbr{\psi_0(\tilde{\mu}\scbr{\Om})[t]}.$$
    By induction hypothesis, $\tilde{\mu}_i=\mu_i'$ for $0=i\leq n$.\smallskip
    
    \item
    If $\tilde{\mu} = \psi_1(\tilde{\mu}_0)$ then
    $$\psi_1(\tilde{\mu}_0 \scbr{\psi_0(\tilde{\mu}\scbr{\Om})[t-1]}) \leq  \ga_0' < \psi_1(\tilde{\mu}_0 \scbr{\psi_0(\tilde{\mu}\scbr{\Om})[t]}).$$
    Then $\ga_0' = \psi_1(\de_1') + \eta'$ and by applying the induction hypothesis to $\de_1$, $\tilde{\mu}_0=\mu_0'$.
    \end{Cases}
\end{Cases}
This ends the proof of the Claim and of the Lemma.
\end{proof}
\smallskip
The next step in showing that the Goodstein process for our normal forms terminates is to show that normal forms are preserved after base-change, which we will use later to show that the ordinal assignment gives a decreasing sequence of ordinals.\smallskip

\begin{lemma} \label{lm:MonBaseChange} {Monotonicity on system $\OT_0$.}
Let $m':=m\kp$ and $\al':=\al\kp$.
\begin{enumerate}
    \item If $m<n$ then $m'<n'$. 
    \item If $\al<\be$ then $\al'<\be'$.
\end{enumerate}
\end{lemma}
\begin{proof}
By simultaneous induction on $n$ and $\be$.\medskip\\
The first assertion by induction on $n$.
Let $m=_k H_\al(k)\cdot p +q$ and $n=_k H_\be(k) \cdot v + u$.\\
If $H_\al(k)=H_\be(k)$ and $p<v$ then $m'<n'$.\\
If $H_\al(k)=H_\be(k)$ and $p=v$, then $q<u$ by induction hypothesis $q'<u'$ and the assertion follows.
We now assume that $m=_k H_\al(k)< H_\be(k)=_k n$. We have that $\al<\be$, since if $\be<\al$, then $H_\be(k)$ is not in normal form by a similar argument as in Lemma \ref{lm:basechangeDiff}.
\medskip
\\
We want to show that $m'<H_{\be'}(k+1)$. First, we need the following claim:\medskip
\begin{claim} 	 $\al' < \be'\{k+1\}$.
\end{claim}
\begin{proof}
By Lemma \ref{lm:Context2} there are several cases to consider.
We will consider the non-obvious ones.
\begin{Cases}
    \item First, for all contexts $\la$, ordinals $\ga$ and natural numbers $t$, 
    if $\al=\la\scbr{t}< \la^-\scbr{\om}=\be $ and $\al<\be[t+1]$, by Lemma \ref{lm:OrdMeasure}, $t<H_{\la^-\scbr{\om}\{k-1\}}(k)$. \smallskip\\
    To apply item \ref{case1:lm:OrdMeasure2} in Lemma \ref{lm:OrdMeasure2}, we must show that also  $t'< H_{(\la^-\scbr{\om})'\{k\}}(k+1) $. \\
    Assume now  that $\al'=\tilde{\la}\scbr{\tilde{t}}$, then by  Lemma \ref{lm:BaseChangeContext} there exists some $\la$ and $t$ such that  $\al= \la\scbr{t}$ and $\al' = \la'\scbr{t'}$.
    \medskip\\
    By the induction hypothesis on $t<H_{\la^-\scbr{\om}\{k-1\}}(k)$, we have that 
    \medskip\\
    $t'<(H_{\la^-\scbr{\om}\{k-1\}}(k))' = H_{(\la^-\scbr{\om}\{k-1\})'}(k+1)$.\medskip \\
    Using Claim \ref{claim1}, 
    $t' < H_{(\la^-\scbr{\om}\{k-1\})'}(k+1) \leq H_{(\la^-)'\scbr{\om}\{k-1\}}(k+1)$.
    \medskip\\
    Then $\al' < \be'\{k+1\}= \be'[H_{\be'\{k\}}(k+1)]$ by Lemma \ref{lm:OrdMeasure2}.
    \\
    
    \item  
    Secondly, for all contexts $\la$, ordinals $\ga$ and natural numbers $t$, 
    if $\al=\la\scbr{\ga}< \la^-\scbr{\psi_0(\mu\scbr{\Om})}=\be $ for $\mu$ a $\psi_0-$nesting free context and there is some $t$ such that $\be[t]\leq \al <\be[t+1]$ then by Lemma \ref{lm:OrdMeasure}, $t<H_{\la^-\scbr{\psi_0(\mu\scbr{\Om})}\{k-1\}}(k)$. To apply item \ref{case2:lm:OrdMeasure2} in Lemma \ref{lm:OrdMeasure2}, we must show that $t'<H_{(\la^-)'\scbr{\psi_0(\mu'\scbr{\Om})}\{k\}}(k+1)$.
    \medskip\\
    Assume now  that $\al'=\tilde{\la}\scbr{\tilde{\ga}}$ and $\be'=\tilde{\la}^-\scbr{\psi_0(\tilde{\mu}\scbr{\Om})}$ and there is some $t$ such that $\be'[t] \leq\al\leq\be'[t+1]$.
    Then by Lemma \ref{lm:BaseChangeContext}, $\al' = \la'\scbr{\ga'}$ where $\tilde{\la}=\la'$ and $\tilde{\ga}=\ga'$. 
    \smallskip\\
    Moreover, by Lemma \ref{lm:BaseChangeContext2}, $\be' = (\la^-)'\scbr{\psi_0(\mu'\scbr{\Om})}$ for $\tilde{\la^-} = (\la^-)'$ and $\tilde{\mu}=\mu'$.
    \smallskip\\
    Also, $\be'[t]\leq \al' <\be'[t+1]$ implies $\be[t] \leq \al<\be[t+1]$ by the Bachmann property.\smallskip
    \\
    By induction hypothesis on  $t<H_{\la^-\scbr{\psi_0(\mu\scbr{\Om})}\{k-1\}}(k)$, we have that
    \medskip\\
    $t' < (H_{\la^-\scbr{\psi_0(\mu\scbr{\Om})}\{k-1\}}(k))'< H_{(\la^-)'\scbr{\psi_0(\mu'\scbr{\Om})}\{k\}}(k+1)$, where the last assertion follows from Claim \ref{claim2}.
    \smallskip
\\
    Then $\al' < \be'\{k+1\}= \be'[H_{\be'\{k\}}(k+1)]$ by Lemma \ref{lm:OrdMeasure2}.
    \\

    \item
    Lastly, for all contexts $\la$, ordinals $\ga$ and natural numbers $s$, 
    if $\al=\la\scbr{\psi_i(\ga) \cdot s}< \la^-\scbr{\psi_i(\ga+1)}=\be $  for $i \in\{0,1\}$ then by Lemma \ref{lm:OrdMeasure}, $s<H_{\la^-\scbr{\psi_i(\ga+1)}\{k-1\}}(k)$.
    \medskip\\
    Assume now that $\al'=\tilde{\la}\scbr{\psi_i(\tilde{\ga}) \cdot s }$ and $\be'=\tilde{\la}^-\scbr{\psi_i(\tilde{\ga}+1)}$. Then by Lemma \ref{lm:BaseChangeContext} there exists some $\la$ and $\ga$ such that  $\al= \la\scbr{\psi_i(\ga) \cdot s }$ and $\be=\la^-\scbr{\psi_i(\ga+1)}$
    and $\al' = \la'\scbr{\psi_i(\ga') \cdot s}$ and  $\be'=(\la^-)'\scbr{\psi_i(\ga'+1)}$.
    The value $s$ is the number of occurrences of $\psi_i(\ga)$, and it does not change under base change operation.         
    \medskip\\
    Since $s<H_{\la^-\scbr{\psi_i(\ga+1)}\{k-1\}}(k)$ and $H_{\la^-\scbr{\psi_i(\ga+1)}\{k-1\}}(k)$ is in $k$ normal form, then from Lemma \ref{lm:basechangeDiff}, $s<H_{\la^-\scbr{\psi_i(\ga+1)}\{k-1\}}(k) <H_{(\la^-)'\scbr{\psi_i(\ga'+1)}'\{k\}}(k+1)$.\smallskip
    \\
    By Lemma \ref{lm:OrdMeasure2}, $\al' < \be'\{k+1\}= \be'[H_{\be'\{k\}}(k+1)]$.
\end{Cases}
\end{proof}
Finally we can conclude that,\\
$m'<  H_{\al'+1}(k+1) \leq H_{\be'[H_{\be'\{k\}}(k+1)]}(k+1) = H_{\be'\{k+1\}}(k+1) = H_{\be'}(k+1) = n' $.
\\
\\
\\
Proof of the second assertion by induction on $\be \in \OT_0$.
\begin{Cases}
    \item  $\al =_{NF} \al_0 +\ldots +\al_n +l$  and $\be =_{NF} \be_0 + \ldots + \be_m + s$.\\
    Then either $\al_0<\be_0$ and by induction hypothesis $\al_0'<\be_0'$,\\
    or there is some $0 \leq i \leq min(m,n)$ such that $\al_i=\be_i$  and $\al_{i+1}< \be_{i+1}$ and by induction hypothesis $\al_i'=\be_i'$ and $\al_{i+1}' < \be_{i+1}'$, \\
    or $\al_i=\be_i$ for all $i$ and $l<s$; by the previous assertion $l'<s'$,\\
    then $\al' =\al_0'+\ldots + \al_n'+l' < \be_0'+\ldots+\be_m'+s' =\be'$. 
    
    \item $\al =_{NF} \psi_0(\al_0) $ and $\be =_{NF} \psi_0(\be_0)$ and $G_0(\al_0)<\al_0$ and $G_0(\be_0)<\be_0$, 
    then $\al_0 < \be_0$  then by IH, $\al_0'< \be_0'$ and
    $G_0(\al_0') <\al_0'$
    thus $\al'$ is in normal form and $\al'<\psi_0(\al_0') <\psi_0(\be_0' )<\be'$.
    
    \item $\al =_{NF} \psi_1(\al_0) $ and $\be =_{NF} \psi_1(\be_0)$, 
    then $\al' = \psi_1(\al_0')$ and $\be' = \psi_1(\be_0')$.
    Since $\al<\be$ then $\al_0<\be_0$ and by induction hypothesis $\al_0'<\be_0'$ and $\al_0'$ is in normal form, and $\al' = \psi_1(\al_0') < \psi_1(\be_0')= \be_0'$.
    
    \item $\al =_{NF} \psi_i(\al_0)$ and $\be =\psi_{i+1}{\be_0}$ for $i \in \{0,1\}$ then 
    $\al' = \psi_i(\al_0') < \psi_{i+1}(\be_0')= \be'$.
\end{Cases}
\end{proof} 

\begin{lemma} Let $\al' := \al\kp$.\\
If $G_0\al <\be$ then $G_0\al'<\be'$.
\end{lemma}
\begin{proof}
By induction on  $2^{N\al} + 2^{N\be}$.
\begin{Cases}
    \item If $\al=_{NF} \al_0 +\ldots +\al_n +l$ and $G_0\al = G_0\al_0 \cup  \ldots  \cup G_0\al_n \cup G_0l< \be$. 
    Then by induction hypothesis, $G_0(\al') =  G_0(\al_0') \cup \ldots \cup G_0(\al_n') \cup G_0(l')< \be'$.
    
    \item If $\al=_{NF} \psi_0\al_0$ and $G_0\al = G_0\al_0 \cup \{\al_0\} \leq \be$.
    Then $G_0\al_0 <\al_0$ and by induction hypothesis $G_0(\al_0') <\al_0'$ and $G_0(\al_0')<\be'$.
    Then $G_0(\al')= G_0\psi_0(\al_0') = G_0(\al_0') \cup \{\al_0'\} < \be'$.
    
    \item $\al =_{NF} \psi_1\al_0$ and $G_0\al= G_0(\al_0) < \be$. By induction hypothesis $G_0(\al_0')< \be'$, then $G_0(\al') = G_0(\psi_1(\al_0'))= G_1(\al_0') <\be'$.
    
\end{Cases}
\end{proof}
Thus, the base-change operation is monotone. Next we show that it also preserves normal
forms.
\medskip

\begin{lemma}[Normal form preservation after base change]\label{lm:NFpreservBaseChange}
Let $m' := m\kp$\\
If $m=_{k} H_{\al}(k)\cdot p+ q$ then $m'=_{k+1} H_{\al'}(k+1) \cdot p+p'$.
\end{lemma}
\begin{proof}
We show that $\al'$ is maximal with respect to the value of denoting $m'$.
Assume that $p=1$and $q=0$, such that $m=_k H_{\al}(k)$. We will prove that $\al'$ is maximal thus showing that $H_{\al'}(k+1)$ is in $k+1$-normal form.\medskip
\begin{claim} 	Assume that $\al'<\xi$. Then $H_\xi(k+1)>m'$.
\end{claim} 
\begin{proof}
    We proceed towards a contradiction and consider the non-obvious cases. 
    \begin{Cases}
        \item Assume $\al'=\tilde{\la}\scbr{\tilde{t}}$. By Lemma \ref{lm:BaseChangeContext}, there exists some $\la$ and $t$ such that $\tilde{\la}=\la'$ and $\tilde{t}=t'$ and $\al=\la\scbr{t}$. 
        By Lemma \ref{lm:OrdMeasure} we have that $t<H_{\la^-\scbr{\om}\{k-1\}}(k)$.
        \\
        Let $\be= \la^-\scbr{\om}$. For $\be'=(\la^-)'\scbr{\om}$ and by  applying Claim \ref{claim1} we have that
        $\be'\{b\} =( \be\{b\})'$ for $b < k+1$.
        \medskip\\
        Hence $t'< (H_{\la^-\scbr{\om}\{k-1\}}(k))' \leq H_{(\la^-\scbr{\om}\{k-1\})'}(k+1)
        \leq  H_{(\la^-)'\scbr{\om}\{k\}}(k+1)$, \medskip
        \\since $(\la^-\scbr{\om}\{k-1\})' \leq (\la^-\scbr{\om}\{k\})' = (\la^-)'\scbr{\om}\{k\}$.
        \medskip
        \item Assume $\al'=\tilde{\la}\scbr{\tilde{\ga}}$ then by Lemma \ref{lm:BaseChangeContext}, there exists some $\la$ and $\ga<\Om$ such that $\tilde{\la}=\la'$ and $\tilde{\ga}=\ga'$ and $\al=\la\scbr{\ga}$.
        \medskip\\
        By Lemma \ref{lm:OrdMeasure}, if $\al=\la\scbr{\ga}$ and $\la^-\scbr{\psi_0(\mu\scbr{\Om})}[t]\leq \al<\la^-\scbr{\psi_0(\mu\scbr{\Om})}[t+1]$ for $\mu$ a $\psi_0-$nesting free context, then
        $t<H_{\la^-\scbr{\psi_0(\mu\scbr{\Om})}\{k-1\}}(k)$. 
        \medskip\\
        Let $\be = \la^-\scbr{\psi_0(\mu\scbr{\Om})}$. By Lemma \ref{lm:BaseChangeContext2}, for $\be' = (\la^-)'\scbr{\psi_0(\tilde{\mu}\scbr{\Om})}$ and for some $t$ such that $\be'[t]\leq\al'< \be'[t+1]$, there is some  some $\psi_0-$nesting free context $\mu$ such that $\mu'=\tilde{\mu}$ and $\be' = (\la^-)'\scbr{\psi_0(\mu'\scbr{\Om})}$.
        \medskip\\
        Moreover, by Claim \ref{claim2} we have that $\be'\{b\} =( \be\{b\})'$ for $b < k+1$.
        \smallskip\\
        Hence,  $t'< H_{(\la^-)'\scbr{\psi_0(\mu'\scbr{\Om})}\{k\}}(k+1)$.
        \medskip
        \item Assume $\al'=\tilde{\la}\scbr{\psi_i(\tilde{\ga})\cdot s})$ for $i\in \{0,1\}$ then by Lemma \ref{lm:BaseChangeContext}, there exists some $\la$ and $\ga$, such that $\tilde{\la}=\la'$, $\tilde{\ga} =\ga'$ and $\al=\la\scbr{\psi_i(\ga)\cdot s}$.
        \medskip\\
        Recall here that the value  $s$ is the number of occurrences of $\psi_i(\ga)$ and as such does not change under base change.
        \medskip\\
        By Lemma \ref{lm:OrdMeasure} we have that $s<H_{\la^-\scbr{\psi_i(\ga+1)}\{k-1\}}(k)$ and by Lemma \ref{lm:basechangeDiff},\medskip\\
        $s < H_{(\la^-)'\scbr{\psi_i(\ga'+1)}\{k\}}(k+1)$ .
    \end{Cases}
\end{proof}
Assume $\al'<\xi$ then $m'<H_{\al'+1}(k+1)\leq H_\xi(k+1)$ by Lemma \ref{lm:OrdMeasure2}.
So, for each case, $\al'$ meets the maximality condition and $m'$ is in $k+1$ normal form.\\
If $m=_k H_\al(k)\cdot p +q$ , then by the induction hypothesis, the assertion follows.\medskip
\\
If $p>1$ or ($p=1$ and $q>0$) then we may assume by the previous argument that $H_{\al'}(k+1)$ is in $k+1$ normal form and this yields the assertion.\\
\end{proof}

Having proven the monotonicity and normal form preservation under the base change operation, we can conclude the following Corollary.
\begin{corollary}
    The normal forms after the base change operation are well founded.
\end{corollary}

\smallskip
To prove that the new Goodstein process terminates, we need to assign ordinals to natural numbers in a way that ensures the process produces a decreasing (and thus finite) sequence. For each integer $k$, define a function $\ok(\cdot): \mathbb{N} \to \Lambda$, where $\Lambda$ is a suitable ordinal, such that $\ok(m)$ is computed from the $k$ normal form of $m$. 
Goodstein normal forms here correspond to ordinals below $\Lambda = \psi_0\psi_1(\varepsilon_{\Om_2+1})$, as illustrated by the following mapping.

\subsection{Ordinal Assignment}\label{subsec:OrdAssig}
The function $ \ok(\cdot)$ maps natural numbers $m$ in $k$ normal form and ordinals $\al \in \OT_0$ to ordinals in the system $T\cap \Om_2$.\\
By monotonicity (Lemma \ref{lm:OkMon}) it will be shown that the ordinal assignment is actually in $\OT$.\medskip
\\
The ordinal mapping for natural numbers is defined as follows:
\begin{itemize}
\item If $m<k$ then $\ok(m)=m$.
\item If $m=k$ then $\ok(k) =\omega$. 
\item If $m = H_\alpha(k) + q$ then $\ok(m)= \psi_0(\ok(\al) ) + \ok(q)$.\\
\\
Now we extend the ordinal mapping to $\al \in \OT_0$:

\item If $\al= \al_0+ \ldots +\al_n$ then $\ok(\al) = \ok(\al_0) + \ldots + \ok(\al_n)$. \textcolor{orange}
\item If $\al = \psi_0\al_0$, $G_0\al_0<\al_0$ 
then $\ok(\al)= \psi_1(\ok(\al_0))$.
\item If $\al = \psi_1\al_0$ then $\ok(\al) = \psi_2(\ok(\al_0))$.
\\
\\
Next we extend the ordinal mapping to a context $\la$:
\item If $\la\scbr{\cdot} = \al_0 +\ldots +\scbr{\cdot}+\ldots+\al_n$ then 
$\ok(\la\scbr{\cdot}) = \ok(\al_0)+ \ldots +\scbr{\cdot}+\ldots+\ok(\al_n)$.

\item If $\la\scbr{\cdot} = \psi_0(\la_0\scbr{\cdot})$ then $\ok(\la\scbr{\cdot}) = \psi_1(\ok(\la_0\scbr{\cdot}))$.

\item If $\la\scbr{\cdot} = \psi_1(\la_0\scbr{\cdot})$ then $\ok(\la\scbr{\cdot}) = \psi_2(\ok(\la_0\scbr{\cdot}))$.
\end{itemize}


\smallskip

\begin{lemma} \label{lm:okpe=ok} Let $m$ be in $k$ normal form and $m\kp$ the $k+1$ normal form after applying base change operation to $m$. Then the following hold;
\begin{enumerate}
    \item $\okpe(m') = \ok(m)$.
    \item $\okpe(\al') =\ok(\al)$.
    \item $\okpe((\la\scbr{\cdot})')=\ok(\la\scbr{\cdot})$.
\end{enumerate}
\end{lemma}
\begin{proof}
Proof of the first two assertions.\\
First assertion by induction on m.
\begin{Cases}
    \item ($m<k$). Then $\okpe(m') = m = \ok(m)$.
    \item ($m=k$). Then $\okpe(m') = \om = \ok(m)$.
    \item ($m=_{k} H_\al(k) + q$). Then $m' = H_{\al'}(k+1)+ q'$ is in $k+1$-normal form. From the second assertion we have that $\okpe(\al)=\ok(\al)$ and by the induction hypothesis $\okpe(q')=\ok(q)$	, 
    \begin{align*}
        \okpe(m')  &= \psi_0(\okpe(\al')) + \okpe(q')\\
        & = \psi_0(\ok(\al))  + \ok(q)\\
        &= \ok(m).
    \end{align*} 
\end{Cases}
Second assertion by induction on $\al \in \OT_0$.
\begin{Cases}
    \item ($\al<\om$). Then the assertion follows from the first one.
    \item ($\al= \psi_0(\al_0)$).
    Then $\okpe(\al') = \psi_1(\okpe(\al_0')) =_{ih} \psi_1(\ok(\al_0)) = \ok(\psi_0(\al_0)) = \ok(\al)$.
    \item ($\al=\psi_1(\al_0)$).
    then  $\okpe(\al') = \psi_2(\okpe(\al_0')) =_{ih} \psi_2(\ok(\al_0)) = \ok(\al)$.
    \item ($\al = \al_0 + \ldots + \al_n +l $). Then from the first assertion $\okpe(l')=\ok(l)$ and
    $\okpe(\al') = \okpe(\al_0') + \ldots + \okpe(\al_n') + \okpe(l) =_{ih} \ok(\al_0) + \ldots + \ok(\al_n) + \ok(l) = \ok(\al)$.
\end{Cases}
\end{proof}


\begin{lemma} \label{lm:OkMon} Monotonicity  of ordinal assignment. 
\begin{enumerate}
    \item If $m<n$ then $\ok m < \ok n$. 
    \item If $\alpha<\beta$ then $\ok \alpha<\ok \beta$.
\end{enumerate}
\end{lemma}
\begin{proof}
Simultaneously by induction $n$ and $\be$. \medskip\\
The first assertion is by induction on $n$.\\
Let $m=_k H_\al(k)\cdot p +q$ and $n=_k H_\be(k) \cdot v + u$.\\
If $H_\al(k)=H_\be(k)$ and $p<v$ then $\ok(m)<\ok(n)$.\\
If $H_\al(k)=H_\be(k)$ and $p=v$, then $q<u$ by induction hypothesis $\ok(q)<\ok(u)$ and the assertion follows.\smallskip\\
We now assume that $m=_k H_\al(k)< H_\be(k)=_k n$. We have that $\al<\be$, since if $\be<\al$, then $H_\be(k)$ is not in normal form by a similar argument as in Lemma \ref{lm:basechangeDiff}.
Let $m =_{k} H_\al(k) + q$ and $n=_{k} H_\be(k) + v$. 
We need to show that 
$\ok(m)= \psi_0(\ok(\al))< \psi_0(\ok(\be))= \ok(n).$\\
Since $\al< \be<\Om$, then by the second assertion,  $\ok(\al)<\ok(\be)$.
Then in system $\T$ we have that $\psi_0(\ok(\al)) <\psi_0(\ok(\be))$
and we can conclude the first assertion.\medskip
\\
\\
The second assertion by induction on $\be$.
\begin{Cases}
    \item $\al =_{NF} \al_0 +\ldots +\al_n +l$  and $\be =_{NF} \be_0 + \ldots + \be_m + s$.\\
    Then either $\al_0<\be_0$ and by induction hypothesis $\ok(\al_0)<\ok(\be_0)$,\\
    or there is some $0 \leq i \leq min(m,n)$ such that $\al_i=\be_i$  and $\al_{i+1}< \be_{i+1}$ and by induction hypothesis $\ok(\al_i)=\ok(\be_i)$ and $\ok(\al_{i+1}) < \ok(\be_{i+1})$, \\
    or $\al_i=\be_i$ for all $i$ and $l<s$; by the first assertion $\ok(l)<\ok(s)$,\\
    then $\ok(\al) =\ok(\al_0)+\ldots + \ok(\al_n)+\ok(l) < \ok(\be_0)+\ldots+\ok(\be_m)+\ok(s) =\ok(\be)$. 
    
    \item $\al =_{NF} \psi_0(\al_0)$ and $\be =\psi_1{\be_0}  $ then 
    $\ok(\al) = \psi_1(\ok(\al_0)) < \psi_2(\ok(\be_0))= \ok(\be)$.
    
    \item $\al =_{NF} \psi_0(\al_0) $ and $\be =_{NF} \psi_0(\be_0)$ and $G_0(\al_0)<\al_0$ and $G_0(\be_0)<\be_0$, 
    then $\al_0 < \be_0$  then by induction hypothesis $\ok(\al_0)< \ok(\be_0)$ and
    $G_1(\ok(\al_0)) <\ok(\al_0)$
    thus $ok(\al)$ is in normal form and $\ok(\al)<\psi_1(\ok(\al_0) ) <\psi_1(\ok(\be_0) )<\ok(\be)$.
    
    \item $\al =_{NF} \psi_1(\al_0) $ and $\be =_{NF} \psi_1(\be_0)$, 
    then $\ok(\al) = \psi_2(\ok(\al_0))$ and $\ok(\be) = \psi_2(\ok(\be_0))$.
    Since $\al<\be$ then $\al_0<\be_0$ and by induction hypothesis $\ok(\al_0)<\ok(\be_0)$ and $\ok(\al_0)$ is in normal form, and $\ok(\al) = \psi_2(\ok(\al_0)) <\psi_2(\ok(\be_0))$.
\end{Cases}
\end{proof}

\begin{lemma}If $G_0(\al) < \be $ then $G_1(\ok(\al)) < \ok(\be)$.
\end{lemma}
\begin{proof} Proof by induction on $2^{N\al} + 2^{N\be}$.
\begin{Cases}
    \item If $\al=_{NF} \al_0 +\ldots +\al_n +l$ and $G_0\al = G_0\al_0 \cup  \ldots  \cup G_0\al_n \cup G_0l< \be$. 
    Then by ih, $G_1(ok(\al))=  G_1(ok(\al_0)) \cup \ldots \cup G_1(ok(\al_n)) \cup G_1(ok(l))< \ok(\be)$.
    
    \item If $\al=_{NF} \psi_0\al_0$ and $G_0\al = G_0\al_0 \cup \{\al_0\} \leq \be$.
    Then $G_0\al_0 <\al_0$ and by ih $G_1(\ok(\al_0)) <\ok(\al_0)$ and $G_1(\ok(\al_0))<\ok(\be)$.
    Then $G_1(\ok(\al))= G_1\psi_1(\ok(\al_0)) = G_1(\ok(\al_0)) \cup \{\ok(\al_0)\} < \ok(\be)$.
    
    \item $\al =_{NF} \psi_1\al_0$ and $G_0\al= G_0(\al_0) < \be$. By induction hypothesis $G_1(\ok(\al_0))< \ok(\be)$, then $G_1(\ok(\al)) = G_1(\ok(\psi_1(\al_0))) = G_1(\psi_2(\ok(\al_0)))= G_1(\ok(\al_0)) < \ok(\be)$.
\end{Cases}
\end{proof}

We now show that the normal form is preserved under ordinal assignment. 
Recall what it means for an ordinal $\psi_0(\al)$ to be in $\OT$:
\begin{center}{\em $\psi_{0}\al \in \OT$ if and only if $G_0\al < \al$}.\end{center}
Specifically, we have to show that for $m$ written in $k$ normal form, $\ok(m)$ is in $\OT$.
Let $m =_k H_\al(k)\cdot p +q $, and $\ok(m)=\psi_0(\ok(\al))\cdot p + \ok(q)$. We must show that  $G_0(\ok(\al)) <  \ok(\al) $. Moreover, $\ok(q)\in \OT$ by the induction hypothesis. 
Since $q<H_{\al}(k)$, then by monotonicity of ordinal assignment, $\ok(q)<\ok(H_{\al}(k))=\psi_0(\ok(\al))<\psi_0(\ok(\al))\cdot p=\ok(H_\al(k)) \cdot p$.\\
If $m=_k H_\al(k)$, then $\ok(m) =\psi_0(\ok(\al))$, and we must show only that  $G_0(\ok(\al)) <  \ok(\al) $.\\
Finally, we must check that $\ok(\al)$ is in $\psi_1$ normal form.
\medskip\\
To prove the normal form preservation of $\ok(m)$, first we need the following Claim. 

\begin{claim}\label{cl:maxCoef}
Let $\al$ be a countable ordinal in $\OT_0$. Then there exists a context $\la$ and a natural number $t>0$ such that $\al = \la\scbr{t}$ where $t$ is the maximal coefficient of $\al$.
Moreover $max G_0(\ok\al) = maxG_0(\ok(t))$.
\end{claim}
\begin{proof}
Proof by induction on $\al$.
For $\al=0$ it is obvious.
\begin{Cases}
    \item ($\al= \al_0+ \ldots +\al_n +l$). Then $\ok(\al) =\ok(\al_0)+\ldots+\ok(\al_n)+\ok(l)$ and $G_0(\ok(\al)) = G_0(\ok(\al_0)) \cup \ldots \cup G_0(\ok(\al_n)) \cup G_0(\ok(l))$.
    Assume that $mc(\al)=t$.
    
    \begin{Cases}
        \item  If $l=t$, define the ordinal context 
    $\la\scbr{\cdot} = \al_0+\ldots +\ldots+\al_n+\scbr{\cdot}$. Then $\al =\al_0+\ldots+\al_n+t=\la\scbr{t}$.\\
    $G_0(\ok(\al))=G_0(\ok(\al_0))\cup\ldots \cup G_0(\ok(\al_n))\cup G_0(\ok(t))=G_0(\ok(t))$, 
    since $tp(\ok(\al_i))$ for $0\leq i\leq n$ is uncountable.\medskip

    \item 
    If $mc(\al)=\al_i$, define
    $\la\scbr{\cdot} = \al_0+\ldots +\la_i\scbr{\cdot}+\ldots+\al_n$ where $\al_i=\la_i\scbr{t}$ by induction hypothesis
    and $G_0(\ok(t))=G_0(\ok(\al_i))$.
    \\
    Then $G_0(\ok(\al)) = G_0(\ok(\al_0)) \cup \ldots \cup G_0(\ok(\al_n)) \leq max(\{\ok(l) \cup \ok(\al_i)\}) < \ok(\al)$.
    \end{Cases}

    \item $(\al = \psi_0(\al_0))$.
      By induction hypothesis $\al_0=\la_0\scbr{t}$ for a context $\la_0$. Then define $\la\scbr{\cdot} = \psi_0\la_0\scbr{\cdot}$, thus $\al = \la\scbr{t}$.
    Moreover $\ok(\al)= \psi_1(\ok(\al_0))$ and $G_0(\ok(\al))= G_0(\ok(\al_0))$.            
    
\end{Cases}
\end{proof}
Take for example, $\al = \psi_0(\psi_10 +\psi_0t)$ then $\ok(\al)=\psi_1(\psi_20 + \psi_1 \ok(t))$ and $G_0(\ok(\al))= G_0(\psi_20) \cup G_0(\psi_1\ok(t))= G_0(0) \cup G_0(\ok(t))\leq G_0(\ok(t))$.
\medskip\\

\begin{lemma}[NF preservation]
Let $m$ be in $k$ normal form, then $\ok(m) \in \OT$.
\end{lemma}
\begin{proof}
Proof by induction on $m$.
\begin{Cases}
\item ( $m<k$). Then $\ok(m)=m$ and $G_0(\ok(m))= \emptyset$.
\item ($m = k$). Then $\ok(m) = \psi_00$ and $G_0(\ok(m)) = \{0\}$.
\item ($m= H_\al(k)\cdot p+ q$). Then $\ok(m)=\psi_0(\ok(\al)) \cdot p+ \ok(q)$.
By monotonicity of ordinal assignment, $\ok(q) < \psi_0(\ok(\al))\cdot p$, and $\ok(q)\in \OT$ by induction hypothesis. We will check that $G_0(\ok(\al))< \ok(\al)$ in the following case.

\item ($m=_k H_\al(k)$). Then $\ok(m) = \psi_0(\ok(\al))$ and we need to check that $G_0(\ok(\al))< \ok(\al)$. 
By Claim \ref{cl:maxCoef}, we can write $\al=\la\scbr{t}$ where $t$ is the maximal coefficient of $\al$ and $\la$ is an ordinal context and $max G_0(\ok(\al)) = maxG_0(\ok(t))$.
\medskip
\\
Let $\be$ be such that $\al<\be$. Since $m$ is in $k$ normal form, then by Lemma \ref{lm:OrdMeasure2}, $\al<\be\{k\}$.
Pick $d$ minimal such that $\al < \be\{d\}$.
Clearly, $d>0$.\medskip\\
Then from $\la\scbr{t}=\al<\be\{d\}=\la^-\scbr{\om}\{d\}=\la^-\scbr{H_{\be\{d-1\}}(k)}$, and we conclude that $t< H_{\be\{d-1\}}(k)$.
Since $H_{\be\{d-1\}}(k)$  is in $k$ normal form by Lemma \ref{lm:NForms}, then by monotonicity of the ordinal assignment, Lemma \ref{lm:OkMon}, $\ok(t)< \ok(H_{\be\{d-1\}}(k))$.\medskip\\
Hence $G_0(\ok(t))\leq G_0(\ok(H_{\be\{d-1\}}(k)))$,  using the fact that for ordinals $\de<\ga<\Om$, we have that $G_0\de\leq G_0\ga$.
Now we consider the bounds of $\al$ in the following cases;\smallskip

\begin{Cases}
    \item ($\be\{d-1\} < \al$). Then
    \begin{align*}
        G_0(\ok(\al))& \leq G_0(\ok(t)) \\
        &\leq G_0(\ok( H_{\be\{d-1\} }(k))) \\
        &= G_0(\psi_0( {\ok(\be\{d-1\})}))\\
        &= G_0(\ok(\be\{d-1\})) \cup \{\ok(\be\{d-1\})\} \\
        &< G_0(\ok(\be\{d-1\})) \cup \{ \ok(\al)\} 	\\
        &\leq \ok(\al).
    \end{align*}
    Since $\ok(\be\{d-1\}) < \ok(\al)$ by monotonicity of ordinal assignment and
    \\
    $ G_0(\ok(\be\{d-1\})) =G_0(\ok(H_{\be\{d-2\}}(k))) 
    \leq G_0(\ok(\be\{d-2\}))~\cup~ \{H_{\ok(\be\{d-2\})}(k)\}\\
    \leq \ok(\al) \cup G_0(\ok(\be\{d-2\})) $.
%
    \\
    \item ($\be\{d-1\} = \al$). Then
    $\la\scbr{t}=\al=\be\{d-1\} = \la^-\scbr{H_{\be\{d-2\}}(k)}$, 
    and we obtain that $t=H_{\be\{d-2\}}(k)$ and
    $G_0(\ok(\al))= G_0(\ok(t)) = G_0(\ok(H_{\be\{d-2\}}(k)))$ and we argue as in the previous case. 
\end{Cases}
\end{Cases}
As a result, we have 
$G_0(\ok(\al))  < \ok(\al) $, hence $\ok(m)\in \OT$.
\end{proof}
Next, we have to check that $\ok(\al)  = \psi_1(\al_0)$ is in $\psi_1$-normal form.
That is, for  $\al=\psi_0(\al_0)$ and  $\ok(\al) =\psi_1(\ok(\al_0))$  then $G_1(\ok(\al_0)) < \ok(\al_0)$.\\ By Lemma \ref{lm:OkMon} third assertion, if $G_0(\al_0)<\al_0$ then $G_1(\ok(\al_0))<\ok(\al_0)$.\\
Hence we have checked that the ordinal assignment is in normal form and we can conclude the following

\begin{corollary}
The normal forms after the ordinal assignment are well founded.
\end{corollary}

Now we can prove that the Goodstein process actually terminates.
\begin{theorem}
$PRA+ PRWO(\psi_0\psi_1(\varepsilon_{\Om_2+1})) \vdash \forall \ell \exists k G_k(\ell)=0$.
\end{theorem}

\begin{proof}
Let $O_k(\ell) := \ok(G_{k}(\ell))$.
If $G_k(\ell) >0$ then 
\begin{align*}
    O_{k+1}(\ell) &= \okpe(G_{k+1}(\ell)) \\
    &= \okpe(G_k(\ell)[k+2 \leftarrow k+3]-1 ) \\
    &< \okpe (G_k(\ell)[k+2\leftarrow k+3]) \\
    &= \ok(G_k(\ell)) = O_k(\ell).
\end{align*}

Since $(O_k(l))_{k<\om}$ cannot be an infinite decreasing sequence of ordinals, there must be some $k$ with $O_k(l)=0$, which yields $G_k(l)=0$.     
\end{proof}

\section{Independence of the Goodstein process}\label{sec: Indep}
In this section we compare the length of our Goodstein sequences with the fast‑growing hierarchy at $|\ID_2|$. First we establish the following technical Lemma \ref{lm:LemmaInd} that will be crucial for proving our main independence result Theorem 60). Independence will be established combining the ordinal assignment from 4.3 with the general hierarchy comparison (Theorem 16, Proposition 17, Theorem 8).

\begin{lemma}\label{lm:LemmaInd} Let $m':= m\kp$. Given $k,m<\om$ with $k\geq 2$,
\[\okpe(m'-1) \geq \ok(m)[k].\]
\end{lemma}
\begin{proof}
We prove the claim by induction on $m$. 
The cases $m<k$ and $m=k$ are immediate from the definition of $o_k$. For $m=H_\alpha(k)\cdot p+q$ in $k$–normal form we split into $q>0$, $q=0<p$, and $q=0=p=1$, and in each case apply Lemmas \ref{lm:HardyNF}-\ref{lm:OrdMeasure} and \ref{lm:okpe=ok}-\ref{lm:OkMon} to track how base–change interacts with the Hardy hierarchy and with fundamental sequences.
\\
If  $m<k$ then $\okpe(m'-1) = \okpe(m-1)=m-1= \ok(m)[k]$.
\\
If $m=k$ then $m= H_0(k)$ and 
\begin{align*} 
    \okpe(m'-1)& =\okpe(H_{0} (k+1) -1)\\
    &= \okpe(k ) =k \\
    &= \psi_0(0)[k]=\ok(m)[k].
\end{align*}

Let $m=_k H_\al(k)\cdot p +q$.
\begin{Cases}
\item If $q>0$, then by Lemma \ref{lm:NForms}, $H_\al(k) \cdot p+ q-1$ is in $k$ normal form. Thus we can apply the ordinal assignment. By induction hypothesis $ \okpe(q'-1) \geq \ok(q)[k]$ and by Lemma \ref{lm:okpe=ok}, $\okpe(\al') = \ok(\al)$.
\begin{align*} 
        \okpe(m'-1)& =\okpe(H_{\al'} (k+1) \cdot p + q'-1)\\
        & =\psi_0(\okpe(\al'))\cdot p + \okpe(q'-1)\\
        &\geq \psi_0(\ok(\al))\cdot p + \ok(q)[k]=\ok(m)[k].
\end{align*}

\item If $q=0$ and $p>1$ then by induction hypothesis in step 3 to 4, as well as by Lemma \ref{lm:okpe=ok}, $\okpe(\al')=\ok(\al)$, we have the following,
    \begin{align*} 
            \okpe(m'-1)& =\okpe(H_{\al'}(k+1) \cdot p - 1)\\
            &= \okpe(H_{\al'}(k+1) \cdot (p-1) + H_{\al'}(k+1)-1)\\
            & =\psi_0(\okpe(\al')) \cdot (p-1) + \okpe(H_{\al'}(k+1)-1)\\
            &\geq_{ih}  \psi_0(\ok(\al))\cdot (p-1)  + \ok(H_\al (k))[k]\\
            &= \ok(H_\al(k)\cdot (p-1) + H_\al (k))[k] \\
            &=\ok(H_\al(k)\cdot p)[k].
        \end{align*}      
\item If $q=0$ and $p=1$ then $m=_k H_\al(k)$ where $\al = \psi_0\al_0 +... +\psi_0\al_n+l$, for simplicity we will write, $\al = \xi + l$.
There are several sub-cases to consider depending on the type of $\al$:
\begin{Cases}
        
        \item ($\al$ is a successor, $\al = \xi + l, 0<l<k$). Then $\ok(\al)$ is a successor; $\ok(\al)= \ok(\xi)+l $. Since $H_{\xi'+l-1}(k+1) \cdot k $ is in $k+1$ normal form by Lemma \ref{lm:NForms}, we can apply the ordinal assignment in steps 3 to 4.\\ Additionally by Lemma \ref{lm:okpe=ok}, $\okpe(\xi')=\ok(\xi)$.
        \begin{align*} 
            \okpe(m'-1)& =\okpe(H_{\xi'+l} (k+1) -1)\\
            & = \okpe(H_{\xi'+l-1} (k+1) \cdot (k+1) -1)\\
            & \geq \okpe(H_{\xi'+l-1}  (k+1)\cdot k )\\
            & = \psi_0({\okpe(\xi') +l-1}) \cdot k \\
            &= \psi_0({\ok(\xi)} + l-1) \cdot k\\
            &= \psi_0(\ok(\xi)+ l)[k] =\ok(m)[k]\\
        \end{align*}
        
        \item ($\al$ is a successor, $\al=\xi+l, l\geq k$). Then $l$ is written in $k$ normal form. By induction hypothesis, since $l<m$ then $\okpe(l'-1)\geq \ok(l)[k]$. Also, in $\OT$, $\psi_0(\okpe(l'-1)) \geq \psi_0(\ok(l)[k])$.\\
        Since the cofinality of $\ok(l) =\om$, we can apply the fundamental sequence to $\ok(l)$ in the last step. 
        By Lemma \ref{lm:okpe=ok}, $\okpe(\xi')=\ok(\xi)$.
        \begin{align*} 
            \okpe(m'-1)& =\okpe(H_{\xi'+l'} (k+1) -1)\\
            & \geq \okpe(H_{\xi'+l'-1}  (k+1) )\\
            & = \psi_0({\okpe(\xi')} + \okpe(l'-1)) \\
            &\geq \psi_0({\ok(\xi)}+ \ok(l)[k])\\
            &=\psi_0({\ok(\xi)}+ \ok(l))[k]=\ok(m)[k]\\
        \end{align*}

        \item ($\al$ is a limit of countable cofinality). Hence $\al = \xi + \psi_0(\al_n)$ where $\al_n$ can be $0$, a successor, of countable cofinality or of uncountable cofinality. We will differentiate on the type of $\al_n$ and $\ok(\al_n)$ in the following sub-cases:\medskip
        
        \begin{Cases}
            \item ($\al_n = 0$). Thus $\al = \xi + \om$ is of countable cofinality and $\ok(\al)=\ok(\xi)+\Om$ is of uncountable cofinality.
            Since $H_{\al\{d\}}(k)$ for $d<k$ is in $k$-normal form by Lemma \ref{lm:NForms} then we can apply the ordinal assignment.
            By Lemma \ref{lm:okpe=ok}, $\okpe(\xi')=\ok(\xi)$.
            \begin{align*} 
                \okpe(m'-1) & =\okpe(H_{\xi'+\om}(k+1) -1)\\
                & = \okpe(H_{(\xi'+\om)\{k+1\}}(k+1) -1)\\
                & \geq  \okpe(H_{(\xi'+\om)\{k\}}(k+1))\\
                & \geq_{(Claim \ref{claim0})} \psi_0( \ok(\xi)+ \Om )[k]\\
                &=\ok(H_{\xi+\om}(k)[k].\\
            \end{align*}
            We prove the following Claim \ref{claim0} to explain steps 3 to 4.
            First, we will define the fundamental sequences for our case.\\
            
            \begin{definition}Fundamental sequences for $\psi_0(\ok(\xi) + \Om)[t] = \psi_0(\ok(\xi) +\Om[z_t])$: \\
                $z_0 := 0$\\
                $z_1 := \psi_0( \ok(\xi)+ \Om)[0] = \psi_0( \ok(\xi)+ \Om[z_0]) = \psi_0( \ok(\xi) +z_0)=\psi_0(\ok(\xi)) $\\
                $z_{t+1} := \psi_0(\ok(\xi) + \Om)[t] = \psi_0(\ok(\xi) + \Om[z_t])=
                \psi_0( \ok(\xi)+ \Om[\psi_0(\ok(\xi) +  \Om[z_{t-1}])])
                = \psi_0( \ok(\xi)+ \psi_0(\ok(\xi) +  z_{t-1}))$.\\
            \end{definition}
            \begin{claim}\label{claim0}
                \[ \okpe(H_{(\xi'+\om) \{t\}}(k+1))\geq \psi_0( \ok(\xi)+ \Om)[t].\]
            \end{claim}
            \begin{proof} We prove the claim by induction on $t$.\\
                We use the fact that $ H_{(\xi'+\om)\{t\}}(k+1)$ is in $k+1$ normal form for $t<k+1$ by Lemma \ref{lm:NForms}.\\ 
                For $t=0$, by Lemma \ref{lm:okpe=ok}, $\okpe(\xi')= \ok(\xi)$, and in $\OT$ we have that $ \psi_0(\okpe(\xi')) =  \psi_0(\ok(\xi))$.
                \begin{align*}
                    \okpe(H_{(\xi'+\om)\{0\}} (k+1)) &=\okpe(H_{(\xi'+\om)[0]}(k+1)) \\
                    &=\okpe(H_{\xi'+\om[0]}(k+1))\\
                    &= \okpe(H_{\xi'}(k+1))\\
                    &= \psi_0(\okpe(\xi')) = \psi_0( \ok(\xi)) \\
                    &= \psi_0( \ok(\xi)+ \Om[z_0]) = \psi_0( \ok(\xi)+ \Om)[0].\\
                \end{align*} 
                
                Since $t=0$ takes the argument to $0$, we will also consider at $t=1$.
                Again using  Lemma \ref{lm:okpe=ok}, $ \psi_0(\okpe(\xi') + \psi_0(\okpe(\xi'))) =  \psi_0(\ok(\xi) + \psi_0(\ok(\xi)))$.
                \begin{align*}
                    \okpe(H_{(\xi'+\om)\{1\}} (k+1)) &=\okpe(H_{\xi'+\om[H_{(\xi' + \om)\{0\}}(k+1)]}(k+1)) \\
                    &= \okpe(H_{\xi'+H_{\xi'}(k+1)}(k+1))\\
                    &= \psi_0(\okpe(\xi') + \psi_0(\okpe(\xi')))\\
                    &= \psi_0( \ok(\xi)+ \psi_0(\ok(\xi))) \\
                    &= \psi_0( \ok(\xi)+\Om[\psi_0(\ok(\xi) +\Om[z_0]))]) \\
                    &= \psi_0( \ok(\xi)+ \Om[z_1]) = \psi_0( \ok(\xi)+ \Om)[1].\\
                \end{align*}
                Induction step for $t+1$;
                \begin{align*}
                    \okpe(H_{(\xi'+\om)\{t+1\}} (k+1)) &=\okpe(H_{\xi'+\om[H_{(\xi'+\om)\{t\}}(k+1)]}(k+1)) \\
                    &= \okpe(H_{\xi'+H_{(\xi'+\om)\{t\}}(k+1)}(k+1)) \\
                    &=\psi_0(\okpe(\xi')+ \okpe(H_{(\xi'+\om)\{t\}}(k+1)))\\
                    &\geq \psi_0( \ok(\xi)+\psi_0(\ok(\xi)+ \Om)[t]) \\
                    &= \psi_0( \ok(\xi)+\Om[\psi_0(\ok(\xi)+ \Om[z_t])]) \\
                    &= \psi_0( \ok(\xi)+ \Om[z_{t+1}] ) = \psi_0( \ok(\xi)+ \Om)[t+1].\\
                \end{align*}
            \end{proof}
            
            { \em The following cases, $\al=\xi+\psi_0(\ga+l)$, $\al = \xi+\psi_0(\de+\psi_0(\ga+l))$, $\al= \xi+\psi_0(\de + \psi_1(\ga+l))$ and so on, have the same procedure, hence we will consider them together in the following case, and use a context to write $\al$.} \\
            
            \item \label{CaseSucc}($\al_n = \ga+l$ or $\la^-\scbr{\psi_i(\ga+l)}$, for $i\in \{0,1\}$, and $0<l<k$ and  $\ga\in Lim$ or $0$).\smallskip
            \\
            Then $\al= \xi+ \la^-\scbr{\psi_i(\ga+l)}$ and $\ok(\al) = \ok(\xi) + \ok(\la^-)\scbr{\psi_{i+1}(\ok(\ga)+l)}$ are both of countable cofinality.
            \\
            \begin{example}
                For example; \\
                For $\al=\xi+\psi_0(\ga+l)$, the context is 
                $\la^-\scbr{\cdot}=\scbr{\cdot}$ and $\al=\xi +\la^-\scbr{\psi_0(\ga+l)}$.
                \\
                For $\al = \xi+\psi_0(\de+\psi_0(\ga+l))$, the context is $\la^-\scbr{\cdot} = \psi_0(\de + \scbr{\cdot})$, then $\al= \xi +\la^-\scbr{\psi_0(\ga+l)}$.
                \\
                If $\al = \xi + \psi_0(\de + \psi_1(\ga+l))$, the context is $\la^-\scbr{\cdot} = \psi_0(\de + \scbr{\cdot})$, then
                $\al= \xi+ \la^-\scbr{\psi_1(\ga+l)}$.
                \medskip
                \\
                The following example illustrates the ordinal assigment:\\
                $\ok(H_{\psi_0\psi_0\psi_01})[k]= \psi_0(\psi_1\psi_1\psi_11)[k] = \psi_0(\Om^3\cdot \om)[k]= \psi_0(\Om^3 \cdot k)$.
                \\
                Consider another example to ordinal assignment as in the following: \\
                $\ok(H_{\psi_0(\psi_1\Om + \psi_1(\om+1))})[k] = \psi_0(\psi_1(\psi_2\Om_2+ \psi_2(\Om+1)))[k]\\=
                \psi_0(\psi_1(\psi_2\Om_2 + \psi_2\Om\cdot k))$. 
            \end{example}
            
            By Lemma \ref{lm:NForms} we have that
            $H_{\xi'+(\la^-)'\scbr{\psi_i(\xi'+l-1)\cdot k}}(k+1)$ is in $k+1$ normal form, hence in steps 3 to 4 we can apply the ordinal assignment.
            \begin{align*} 
                \okpe(m'-1) & =\okpe(H_{\xi'+(\la^-)'\scbr{\psi_i(\ga'+l)}}(k+1) -1)\\
                & = \okpe(H_{(\xi'+(\la^-)'\scbr{\psi_i(\ga'+l)})\{k+1\}}(k+1) -1)\\
                & \geq  \okpe(H_{\xi'+(\la^-)'\scbr{\psi_i(\ga'+l-1)\cdot k}}(k+1))\\
                & =  \psi_0(\okpe(\xi')+ \okpe((\la^-)')\scbr{\psi_{i+1}(\okpe(\ga') + l-1)\cdot k})\\
                & = \psi_0( \ok(\xi)+ \ok(\la^-)\scbr{\psi_{i+1}(\ok(\ga) +l-1) \cdot k})\\
                & = \psi_0( \ok(\xi)+ \ok(\la^-)\scbr{\psi_{i+1}(\ok(\ga) +l)} )[k]\\
                &=\ok(H_{\xi+\la^-\scbr{\psi_i(\ga+l)}}(k))[k].\\
            \end{align*}

            \item \label{CaseNotSucc} $(\al_n=\ga+l$ or $\mu^-\scbr{\psi_i(\ga+l)}$ for $i \in \{0,1\}$,  $l\geq k$ and $\ga\in Lim$ or $0$).\\
            As above we will cover several cases with the same procedure and we will use a context $\mu^-\scbr{\cdot}$ to write $\al$. Since $l\geq k$, then also $l$ is written in $k$  normal form. Then $\al = \xi + \mu^-\scbr{\psi_i(\ga+l)}$ and $\ok(\al)= \ok(\xi)+ \ok(\mu^-)\scbr{\psi_{i+1}(\ok(\ga)+\ok(l))}$.
            \\
            \begin{example}
                Here we cover ordinals such as $\al = \xi + \psi_0(\ga+l)$ or $\al =\xi+ \psi_0(\de+\psi_0(\ga+l))$ or $\al = \xi+ \psi_0(\de + \psi_1(\ga+l))$. 
                \\
                \\
                As for the ordinal assignment, consider an example for $k=3$ and $\al =\psi_0\psi_0\psi_03$: \\
                $\ok(H_{\psi_0\psi_0\psi_03})[k]= \psi_0(\psi_1\psi_1\psi_1\om)[k] = \psi_0(\Om^3\cdot \om^\om)[k]= \psi_0(\Om^3\cdot \om^{k}).$
                \\
                Or consider $\al = \psi_0(\psi_1(\om + 3 ))$, such that\\ 
                $\ok(H_{\psi_0(\psi_1(\om + 3 ))})[k] = \psi_0(\psi_1( \psi_2(\Om + \om)))[k]  = \psi_0( \psi_1(\psi_2(\Om + \om[k]))).$
                \\
            \end{example}
            In step 3, by Lemma \ref{lm:NForms} we can apply the ordinal assignment, 
            \\as $ H_{\xi'+(\mu^-)'\scbr{\psi_i(\ga'+l'-1)}}(k+1)$ is in $k+1$ normal form. \\
            In steps 4 to 5, by induction hypothesis $\okpe(l-1)\geq \ok(l)[k]$. Additionally
            $\psi_{i+1}(\okpe(l-1)) \geq \psi_{i+1}(\okpe(l)[k])$ in $\OT$.\\
            Since the cofinality of $\ok(l)=\om$ we can apply the fundamental sequence to $\ok(l)$.
            
            \begin{align*} 
                \okpe(m'-1) & =\okpe(H_{\xi'+(\mu^-)'\scbr{\psi_i(\ga'+l')}}(k+1) - 1)\\
                & = \okpe(H_{\xi'+(\mu^-)'\scbr{\psi_i(\ga'+l')}\{k+1\}}(k+1)-1)\\
                & \geq \okpe(H_{\xi'+(\mu^-)'\scbr{\psi_i(\ga'+l'-1)}}(k+1))\\
                & = \psi_0(\okpe(\xi')+ \okpe((\mu^-)')\scbr{\psi_{i+1}(\okpe(\ga')+\okpe(l'-1))}))\\
                & \geq \psi_0( \ok(\xi)+ \ok(\mu^-)\scbr{\psi_{i+1}(\ok(\ga)+\ok(l)[k])})\\
                & = \psi_0(\ok(\xi)+ \ok(\mu^-)\scbr{\psi_{i+1}(\ok(\ga)+\ok(l))})[k]\\
                &=\ok(H_{\xi+ (\mu^-)\scbr{\psi_i(\ga+l)}}(k))[k].\\
            \end{align*}

            \item ($\al_n=\tau^-\scbr{\om}$).\label{CaseCountCof}
            Then $\al= \xi + \psi_0(\tau^-\scbr{\om})$ is of countable cofinality and
            $\ok(\al)= \ok(\xi) + \psi_1(\ok(\tau^-)\scbr{\Om})$ is of uncountable cofinality.\medskip
            \\
            Here we  cover cases such as  $\al = \xi + \psi_0(\de + \om)$ or $\al = \xi + \psi_0(\de+ \psi_0(\ga +\om))$  or $\al = \xi + \psi_0(\de+ \psi_1(\ga +\om))$. \\	    
            Recall the fundamental sequences for $\al = \psi_0(\al_0)$ where $tp(\al)=\om$ and $\al_0 :=\mu^-\scbr{\Om}$ for a $\psi_0-$free context $\mu^-$.
            $\psi_0(\al_0)[x] = \psi_0(\al_0[z_x])$ where $z_0=0$ and $z_{x+1}= \psi_0(\al_0[z_x])$.\medskip
            
            \begin{example}
                The following examples clarify the fundamental sequences that arise in this setting:\\
                For simplicity, let $\al = \psi_0(\om)$, then $\ok(H_{\psi_0(\om)}(k))[x]= \psi_0(\psi_1(\Om))[x]= \psi_0(\psi_1(\Om)[z_x])$ since $\psi_1(\Om)$ is of cofinality $\Om$. Then $\psi_0(\psi_1(\Om)[z_x])= \psi_0(\psi_1(\Om[z_x]))$. \\
                The following example illustrates further:\\
                Let $\al = \psi_0(\psi_1(\om))$ then $\ok(H_{\psi_0(\psi_1(\om))}(k))[x] = \psi_0(\psi_1(\om^{\Om_2+\Om}))[x]= \psi_0(\psi_1(\om^{\Om_2+\Om})[z_x])$ since cofinality of $\psi_1(\om^{\Om_2+\Om})$ is $\Om$.\\
                Then
                $ \psi_0(\psi_1(\om^{\Om_2+\Om})[z_x])= \psi_0(\psi_1(\om^{\Om_2+\Om}[z_x]))= \psi_0(\psi_1(\om^{\Om_2+\Om[z_x]}))$ since the cofinality of $\om^{\Om_2+\Om}$ is $\Om$.\\
            \end{example}
            
            Since $H_{\al\{d\}}(k)$ for $d<k$ is in $k+1$-normal form by Lemma \ref{lm:NForms} then we can apply the ordinal assignment.
            By Lemma \ref{lm:okpe=ok}, $\okpe(\xi')=\ok(\xi)$.
            
            \begin{align*} 
                \okpe(m'-1) & =\okpe(H_{\xi'+\psi_0((\tau^-)'\scbr{\om})}(k+1) -1)\\
                & = \okpe(H_{(\xi'+\psi_0((\tau^-)'\scbr{\om}))\{k+1\}}(k+1) -1)\\
                & \geq  \okpe(H_{(\xi'+\psi_0((\tau^-)'\scbr{\om}))\{k\}}(k+1))\\
                & \geq_{\text{Claim \ref{ClaimIndOm}}} \psi_0( \ok(\xi)+ \psi_1(\ok(\tau^-)\scbr{\Om})) [k]\\		
                &=\ok(H_{\xi+\psi_0(\tau^-\scbr{\om})}(k)[k].\\
            \end{align*}
            We prove Claim \ref{ClaimIndOm} to explain steps 4 to 5.
            First, we give the following definition.
            \begin{definition} Definition of fundamental sequences for 
                $\psi_0(\ok(\xi)+ \psi_1(\ok(\tau^-)\scbr{\Om}))[t]$;\\
                $z_0 =0$,\\
                $z_1 =	\psi_0( \ok(\xi)+ \psi_1(\ok(\tau^-)\scbr{\Om}))[0]                =\psi_0( \ok(\xi)+ \psi_1(\ok(\tau^-)\scbr{\Om})[z_0])
                =\psi_0( \ok(\xi) + \psi_1(\ok(\tau^-\scbr{\Om[z_0]}))$, 
                \\
                $z_{t+1} = \psi_0( \ok(\xi)+ \psi_1(\ok(\tau^-)\scbr{\Om}))[t]
                = \psi_0( \ok(\xi)+ \psi_1(\ok(\tau^-)\scbr{\Om})[z_t])\\
                =\psi_0( \ok(\xi)+ \psi_1(\ok(\tau^-)\scbr{\Om[z_t]}))
                = \psi_0( \ok(\xi)+\psi_1(\ok(\tau^-) \scbr{ \psi_0( \ok(\xi)+ \psi_1(\ok(\tau^-)\scbr{\Om[z_{t-1}]}) )}). $\medskip
            \end{definition}
            
            \begin{claim}\label{ClaimIndOm}
                \[\okpe(H_{(\xi'+\psi_0((\tau^-)'\scbr{\om}))\{t\}}(k+1))  \geq \psi_0( \ok(\xi)+ \psi_1(\ok(\tau^-)\scbr{\Om}))[t].\]
            \end{claim}
            \begin{proof}
                By induction on $t$. $H_{\xi'+\psi_0((\tau^-)'\scbr{\om})[x]}(k+1)$ is in $k+1$ normal form for $x<k+1$, hence we can apply the ordinal assignment. For $t=0$ we obtain
                \begin{align*}
                    \okpe(H_{(\xi'+\psi_0((\tau^-)'\scbr{\om}))\{0\}}(k+1)) &= \okpe(H_{\xi'+\psi_0((\tau^-)'\scbr{\om})[0]}(k+1)) \\
                    &= \okpe(H_{\xi'+\psi_0((\tau^-)'\scbr{0})}(k+1)) \\
                    &= \psi_0(\okpe(\xi')+\psi_1(\okpe((\tau^-)')\scbr{0}))\\
                    &= \psi_0( \ok(\xi)+ \psi_1(\ok(\tau^-)\scbr{0 }))\\
                    &= \psi_0( \ok(\xi)+ \psi_1(\ok(\tau^-)\scbr{\Om}))[0].\\
                \end{align*}    
                For $t=1,$
                \begin{align*}
                    \okpe(H_{(\xi'+\psi_0((\tau^-)'\scbr{\om}))\{1\}}(k+1)) 
                    &= \okpe(H_{\xi'+\psi_0((\tau^-)'\scbr{\om}) [H_{(\xi' + \psi_0((\tau^-)'\scbr{\om}))\{0\}}(k+1)])}(k+1)) \\
                    &=\okpe(H_{\xi'+\psi_0((\tau^-)'\scbr{H_{\xi' + \psi_0((\tau^-)'\scbr{0})}(k+1)})}(k+1)) \\
                    &= \psi_0(\okpe(\xi')+\psi_1(\okpe((\tau^-)')\scbr{\psi_0(\okpe(\xi')+\psi_1(\okpe((\tau^-)'\scbr{0}))}) )\\
                    &= \psi_0( \ok(\xi)+ \psi_1(\ok(\tau^-)\scbr{\psi_0(\ok(\xi)+ \psi_1(\ok(\tau^-\scbr{0}))) }))\\
                    &= \psi_0( \ok(\xi)+ \psi_1(\ok(\tau^-)\scbr{z_1}))\\
                    &= \psi_0( \ok(\xi)+ \psi_1(\ok(\tau^-)\scbr{\Om}))[1].
                \end{align*}
                Induction step for $t+1$ by the induction hypothesis;
                \begin{align*}
                    \okpe(H_{(\xi'+\psi_0((\tau^-)'\scbr{\om}))\{t+1\}}(k+1)) &=\okpe(H_{\xi'+\psi_0((\tau^-)'\scbr{\om})	[H_{(\xi' + \psi_0(\tau'\scbr{\om}))\{t\}}(k+1)]}(k+1)) \\
                    &= \okpe(H_{\xi'+\psi_0((\tau^-)'\scbr{ H_{(\xi'+\psi_0((\tau^-)'\scbr{\om}))\{t\}}(k+1)})}(k+1)) \\
                    &=\psi_0(\okpe(\xi')+ \psi_1(\okpe((\tau^-)') \\
                    &	\scbr{	\okpe( H_{(\xi'+\psi_0(\tau'\scbr{\om}))\{t\}})}))\\
                    &\geq \psi_0( ok(\xi)+\psi_1(\ok(\tau^-) \scbr{ \psi_0( \ok(\xi)+ \psi_1(\ok(\tau)\scbr{z_t}) )})) \\
                    &=\psi_0( \ok(\xi)+ \psi_1(\ok(\tau^-)\scbr{z_{t+1}}))\\
                    &=\psi_0( \ok(\xi)+ \psi_1(\ok(\tau^-)\scbr{\Om}))[t+1].
                \end{align*}
            \end{proof}

            \item ($\al_n = \mu^-\scbr{\Om}$) is of uncountable cofinality of the form $\psi_1(\ga_0)$ and $\ga_0$ is either 0 or of type $\Om$. The case where $\ga_0$ is a successor is covered in  \ref{CaseSucc} and \ref{CaseNotSucc}, and the case where $\ga_0$  is of countable cofinality is covered above in \ref{CaseCountCof}.
            We use the context to write $\al= \xi +\psi_0(\mu^-\scbr{\Om})$  of countable cofinality, then  $\ok(\al)= \ok(\xi) + \psi_1(\ok(\mu^-)\scbr{\Om_2})$ is of countable cofinality.
            We cover cases such as $\al = \xi + \psi_0(\Om)$ or $\al = \xi +\psi_0(\psi_1(\Om))$.\smallskip
            \\
            Recall the fundamental sequences for $\al = \psi_0(\al_0)$ where $tp(\al)=\om$ and $\al_0 :=\mu^-\scbr{\Om}$ for a $\psi_0-$free context $\mu^-$.
            $\psi_0(\al_0)[x] = \psi_0(\al_0[z_x])$ where $z_0=0$ and $z_{x+1}= \psi_0(\al_0[z_x])$.\medskip
            
            \begin{example}
            We now present an example to clarify the fundamental sequences involved in this case:\\
                Let $\al = \psi_0(\Om)$ then $\ok(H_\al(k))[x]= \psi_0(\psi_1(\Om_2))[x]= \psi_0(\psi_1(\Om_2)[x])$
                since $tp(\psi_1(\Om_2))=\om$.
                Instead, $\psi_0(\Om)[x]=\psi_0(\Om[z_x])$, where $z_0=0$ and $z_{x+1}=\psi_0(\psi_1(\Om_2)[z_x])$.
                This can be written fully as $\psi_0(\Om)[x] = \underbrace{\psi_0(\psi_0(\ldots\psi_00\ldots))}_{k \text{ times}}$.
                Whereas, $H_{\al\{x\}}(k)= H_{\psi_0(\Om)\{x\}}(k)\geq H_{\psi_0(\Om)[x]}(k).$\smallskip\\
                The case described below represents a generalized instance of the preceding analysis.\medskip
            \end{example}
            We use the fact that $H_{\al'\{d\}}(k+1)$ is $k+1$ in normal form for $d<k+1$.
            As well as Lemma \ref{lm:okpe=ok} for $\okpe(\xi') = \ok(\xi)$.
            
            \begin{align*} 
                \okpe(m'-1) &= \okpe(H_{\xi'+\psi_0((\mu^-)'\scbr{\Om})}(k+1) -1)\\
                & = \okpe(H_{(\xi'+\psi_0((\mu^-)'\scbr{\Om}))\{k+1\}}(k+1) -1)\\
                & \geq  \okpe(H_{(\xi'+\psi_0((\mu^-)'\scbr{\Om}))\{k\}}(k+1))\\
                & \geq  \okpe(H_{(\xi'+\psi_0((\mu^-)'\scbr{\Om}))[k]}(k+1))\\
                &= \okpe(H_{(\xi'+\underbrace{\psi_0(
                (\mu^-)'[\psi_0((\mu^-)'[\ldots [\psi_0( (\mu^-)'  \scbr{0})] \ldots ])]
                ))}_{k \text{ times}}
                }(k+1))\\
                &= \psi_0(\okpe(\xi') + \underbrace{
                \psi_1(\okpe(\mu^-)'[\psi_1(\okpe(\mu^-)'[\ldots [\psi_1( \okpe(\mu^-)'  \scbr{0})] \ldots ])])}_{k \text{ times}})\\
                &= \psi_0(\ok(\xi) + \underbrace{
                \psi_1(\ok(\mu^-)[\psi_1(\ok(\mu^-)[\ldots [\psi_1( \ok(\mu^-) \scbr{0})] \ldots ])])}_{k \text{ times}})\\
                &= \psi_0( \ok(\xi)+ \psi_1(\ok(\mu^-)\scbr{\Om_2})[{k}] )\\
                &=\ok(H_{\xi+\psi_0(\mu^-\scbr{\Om})}(k))[k]\\
                &= \ok(m)[k].
            \end{align*}			
        \end{Cases}
    \end{Cases}
\end{Cases}
\end{proof}

\begin{theorem}[Independence]
$\ID_2 \not\vdash \forall l \exists k G_k(l)=0$
\end{theorem}
\begin{proof}
Let $\Om^1_1:=0$ and $\Om_1^{n+1} =\Om_1^{\Om^n_1}$, and similarly $\Om^2_1:=0$ and $\Om_2^{n+1} =\Om_2^{\Om^n_1}$.\smallskip
\\
Define $\ell_0=0$ and $\ell_{n+1} =H_{\ell_n}(2)$ and for each $k$ put
$$O_k(\ell_n) := o_{k+2}(G_{k}(\ell)).$$
Then $O_0 (\ell_{n+1}) = \psi_0(\psi_1\Om_2^n)$.\smallskip
\\
Moreover, for all $k$ and $\ell$ Lemma \ref{lm:LemmaInd} and Lemma \ref{lm:OkMon} yield 
$$O_k(H_{\ell_{n+1}} (2))[k+1] \leq O_{k+1}(H_{\ell_{n+1}} (2) ) < O_k (H_{\ell_{n+1}} (2)  ),$$
so the sequence $(O_k (H_{\ell_{n+1}}))_{k\in \Nat}$ is obtained by iterated fundamental sequence descent below $\psi_0\psi_1\varepsilon_{\Om_2+1} = |\ID_2|$.\smallskip
\\
By Proposition \ref{prop:Majorize}, the least $k$ such that $G_k(\ell_{n+1})=0$ is therefore at least $F_{\psi_0\psi_1\varepsilon_{\Om_2+1}}(n)$.\smallskip
\\
Since $|{\ID}_2| = \psi_0\psi_1\varepsilon_{\Om_2+1}$,~~
$F_{|\ID_2|}$is not provably total in ${\ID}_2$, so $|\ID_2|$
cannnot prove that this $k$ exists for all $n$, and hence
$${\ID}_2\not \vdash \forall \ell \exists k \ G_k(H_\ell (2)))=0.$$

\end{proof}


\section{Conclusions and Future Work}
We introduced a novel Goodstein process based on a Hardy hierarchy defined over an ordinal notation system constructed via a two-step collapsing procedure. By avoiding the simultaneous definition typical of standard approaches, this method reveals the inner workings of ordinal assignments and provides a transparent framework for analyzing proof-theoretic strength. We demonstrated that the termination of this process is independent of $\ID_2$, thereby establishing a new independence result at the second proof-theoretic threshold.

The results presented here suggest a natural progression to generalize our approach to stronger theories, specifically $ \ID_n, \ID_\om$, with the ultimate aim of reaching $\ACA_0 +(\Pi_1^1-TR)$.
For $\ID_n$, we anticipate that the two-step collapsing procedure can be generalized by iterating the collapse to depth $n$. Addressing the stronger system $\ID_\omega$, however, will require a transition back to the simultaneous definition frameworks by Buchholz and Sch\"utte \cite{BuchholzSchutte}. Developing these notation systems and their corresponding fundamental sequences remains the next step in characterizing the proof-theoretic strength of these theories through the lens of Goodstein-type principles.

\bibliographystyle{amsplain}
\bibliography{main-ref}
\end{document}

%% file: mycommands.tex
\newcommand{\begincases}{\begin{enumerate}[label*={\sc Case \arabic*},wide, labelwidth=!, labelindent=0pt]}

\newcommand{\begincasesa}{\begin{enumerate}[label={\sc Case ({\rm \roman*})},wide, labelwidth=!, labelindent=0pt]}
\newcommand{\beginsubcases}{\begin{enumerate}[label*= {\rm .\arabic*},wide, labelwidth=!, labelindent=0pt]}

\newlist{Cases}{enumerate}{9}
\setlist[Cases,1]{label={\sc Case {\rm \arabic*}},wide, labelwidth=!, labelindent=0pt}
\setlist[Cases,2]{label*= {\rm .\arabic*},wide, labelwidth=!, labelindent=0pt}
\setlist[Cases,3]{label*= {\rm .\arabic*},wide, labelwidth=!, labelindent=0pt}
\setlist[Cases,4]{label*= {\rm .\arabic*},wide, labelwidth=!, labelindent=0pt}
\setlist[Cases,5]{label*= {\rm .\arabic*},wide, labelwidth=!, labelindent=0pt}
\setlist[Cases,6]{label*= {\rm .\arabic*},wide, labelwidth=!, labelindent=0pt}
\setlist[Cases,7]{label*= {\rm .\arabic*},wide, labelwidth=!, labelindent=0pt}
\setlist[Cases,8]{label*= {\rm .\arabic*},wide, labelwidth=!, labelindent=0pt}
\setlist[Cases,9]{label*= {\rm .\arabic*},wide, labelwidth=!, labelindent=0pt}

\newlist{CasesA}{enumerate}{1}
\setlist[CasesA,1]{label={\sc Case {\rm (\alph*)}},wide, labelwidth=!, labelindent=0pt}

\def\scbrl{[\hspace{-.14em}[}
\def\scbrr{]\hspace{-.14em}]}
\def\scbr#1{\scbrl#1\scbrr}

\newcommand\T{\mathsf{T}}
\newcommand\OT{\mathsf{OT}}
\newcommand{\ID}{\mathsf{ID}}
\newcommand{\PA}{\mathsf{PA}}
\newcommand{\PRA}{\mathsf{PRA}}
\newcommand{\ATR}{\mathsf{ATR}}
\newcommand{\ACA}{\mathsf{ACA}}

\newcommand{\lex}{\mathrm{lex}}

\newcommand{\kp}{[k\leftarrow k+1]}

\renewcommand{\phi}{{\overline{\varphi}}}

\newcommand{\Om}{{\Omega}}

\newcommand{\okpe}{o_{k+1}}

\newcommand{\ok}{o_{k}}

\newcommand{\al}{{\alpha}}

\newcommand{\be}{{\beta}}

\newcommand{\ga}{\gamma}

\newcommand{\de}{\delta}

\newcommand{\la}{\lambda}

\newcommand{\om}{\omega}

\def\scbrl{[\hspace{-.14em}[}
\def\scbrr{]\hspace{-.14em}]}
\def\scbr#1{\scbrl#1\scbrr}

\newcommand\Nat{\mathbb{N}}